\documentclass[english]{amsart}
\usepackage{lmodern}
\usepackage{lmodern}
\usepackage[T1]{fontenc}
\usepackage[latin9]{inputenc}
\usepackage{geometry}
\usepackage{mathrsfs}
\usepackage{amstext}
\usepackage{amsthm}
\usepackage{amssymb}

\makeatletter
\numberwithin{equation}{section}
\numberwithin{figure}{section}
\theoremstyle{plain}
\newtheorem{thm}{\protect\theoremname}
  \theoremstyle{plain}
  \newtheorem{prop}[thm]{\protect\propositionname}

\usepackage{amsthm}

\theoremstyle{plain}
\newtheorem{theorem}{Theorem}[section]

\newtheorem{lemma}[theorem]{Lemma}
\newtheorem{corollary}[theorem]{Corollary}
\newtheorem{hyp}[theorem]{Hypothesis}
\theoremstyle{definition}

\theoremstyle{remark}
\newtheorem{remark}[theorem]{Remark}

\makeatother

\usepackage{babel}
  \providecommand{\propositionname}{Proposition}
\providecommand{\theoremname}{Theorem}

\begin{document}

\title{Modified scattering for the cubic Schrödinger equation small data
solution on product space}

\author{Grace Liu}
\begin{abstract}
In this paper we consider the long time behavior of solutions to the
cubic nonlinear Schr\"{o}dinger equation posed on the spatial domain
$\mathbb{R}\times\mathbb{T}^{d}$, $1\leq d\leq4$. We first prove
the local well-posedness in $C(I;L_{x}^{2}H_{y}^{s})\cap C(I;L_{x,y}^{4})$
for solutions with initial data $u_{0}\in H_{x}^{0,1}L_{y}^{2}\cap L_{x}^{2}H_{y}^{s}$.
Then, for sufficiently small, smooth, decaying data, we prove global
existence and derive modified asymptotic dynamics by using the wave
packet method and normal form corrections. The modified scattering
behavior on $\mathbb{R}\times\mathbb{T}^{d}$ combines the modified
scattering of the cubic NLS on real line $\mathbb{R}$ with cubic
NLS dynamics on torus. We also consider the corresponding asymptotic
completeness problem. 
\end{abstract}

\maketitle

\section{Introduction}

In this paper, we work with the cubic defocusing nonlinear Schrödinger
equation (NLS) which has the form
\begin{equation}
\left(i\partial_{t}+\frac{1}{2}\partial_{x}^{2}+\frac{1}{2}\triangle_{y}\right)u=|u|^{2}u,\label{eq main}
\end{equation}
where $u$ is a complex-valued function on the spatial domain $(x,y)\in\mathbb{R}\times\mathbb{T}^{d}$,
$1\leq d\leq4$. Let $\mathbb{T}:=\mathbb{R}/\left(2\pi\mathbb{Z}\right)$.The
equation is known to be well-posed in $H_{x,y}^{1}$ \cite{IP}. A
suitable solution will satisfy the conservation law for mass
\begin{equation}
\int_{\mathbb{R}\times\mathbb{T}^{d}}\left|u(t)\right|^{2}dxdy=\int_{\mathbb{R}\times\mathbb{T}^{d}}\left|u(0)\right|^{2}dxdy,\label{eq:mass}
\end{equation}
and also for energy 
\begin{equation}
\int_{\mathbb{R}\times\mathbb{T}^{d}}\left|\partial_{x}u(t)\right|^{2}+\left|\nabla_{y}u(t)\right|^{2}+\left|u(t)\right|^{4}dxdy=\int_{\mathbb{R}\times\mathbb{T}^{d}}\left|\partial_{x}u(0)\right|^{2}+\left|\nabla_{y}u(0)\right|^{2}+\left|u(0)\right|^{4}dxdy.
\end{equation}
The purpose of this work is to show that we can apply the newly developed
tools of wave packet testing \cite{IT,IT3,IT4} and normal form correction
\cite{shatah} to establish the same asymptotic behavior as in the
work of Hani, Pausader, Tzvetkov and Visciglia \cite{HPTV} yet in
a simpler manner and with lower regularity compared to the result
they produced. This work also gives a more comprehensive relation
between the resonant dynamic and the energy norm. The results discussed
in this paper can be directly extended to the case of the cubic focusing
NLS ($-\left|u\right|^{2}u$ on the left hand side of (\ref{eq main})).
Nevertheless, this approach could be used in similar situations when
investigating long time behavior.

The problem of asymptotic behavior of small solutions of nonlinear
evolution equations has been an important topic in the study of NLS.
The task of identifying nonlinear global dynamics is extremely challenging.
If a solution $u$ behaves like a linear one for a large time, we
say the equation has long range scattering. If the solution $u$ behaves
like a linear solution with a phase correction term, we say the equation
has modified scattering. The interesting feature of the asymptotic
dynamics of (\ref{eq main}) is that it is not simply a phase correction
term when $d\geq2$, but rather a quasi-periodic frequency dynamic
described by (\ref{eq:eqtorus}) and (\ref{eq:mainerr}). 

It is fair to say that the cubic NLS on the product space inherits
properties from both the cubic NLS on $\mathbb{R}$ and cubic NLS
on torus $\mathbb{T}^{d}.$ It is well known that on the space $\mathbb{R}^{n}$,
the NLS equations

\begin{equation}
\left(i\partial_{t}+\frac{1}{2}\triangle\right)u=\lambda\left|u\right|^{p}u,\qquad u\left(0,x\right)=u_{0}\label{eq:lp}
\end{equation}
have small data scattering for $\frac{2}{n}<p\leq\frac{4}{n}$ when
$n\geq1$ and modified scattering for $p=\frac{2}{n}$, when $n=1,2,3.$
For the specific case $n=1$, the cubic NLS on $\mathbb{R}$ is an
integrable system, which means it has infinitely many conservation
laws, and has modified scattering property \cite{IT,KP,HN}. Let $u$
be a solution to (\ref{eq:lp}) on $\mathbb{R}$ with $p=2$, and
suitable small initial data $u_{0}$. Then we have the modified scattering
form

\begin{equation}
u(t,x)\approx\frac{1}{\sqrt{t}}e^{i\frac{\left|x\right|^{2}}{2t}+i\lambda\left|W\left(\frac{x}{t}\right)\right|^{2}\log t}W\left(\frac{x}{t}\right).\label{eq:1dasy}
\end{equation}

In contrast to NLS on the real line, the solutions for NLS on the
torus exhibit no scattering property and even the global existence
becomes difficult \cite{BGT,GK,H1,HTT1,HTT2,IP}. In this case, many
different long time behaviors can be sustained even on arbitrary small
open data sets around zero. The NLS on torus were first studied by
Bourgain \cite{B1}\cite{B2}. The Strichartz estimate for $\mathbb{R}^{n}$
fails in the case of $\mathbb{T}^{d}.$ The cubic NLS with $H_{x}^{1}$
initial data is locally well-posed in $C\left([0,T);H_{x}^{1}\right)\cap X^{1}\left([0,T)\right)$
for $d=1,2,3$. For $d=4$, the equation reaches the energy critical
exponent, hence we do not expect local well-posedness for $d\geq5.$
Since the cubic NLS on torus has no dispersive property, we do not
expect any decay in time. The asymptotic equation can be written as
an infinite dimensional dynamical system. Notice that in the case
$d=1$, $p=2$ , due to the complete integrable property, it has infinitely
many conservation laws. But for $d>1$, for any given parameters $s>1$,
$K\gg1$, and $0<\delta\ll1$, there exists a global solution $u(t,x)$
of (\ref{eq:lp}) and a time $T>0$ such that 
\[
\left\Vert u(0)\right\Vert _{H^{s}}\leq\delta\quad\text{ and }\quad\left\Vert u(T)\right\Vert _{H^{s}}\geq K.
\]
Therefore the orbit of any $H^{s}-$neighborhood of the origin under
the nonlinear flow of the cubic NLS is not uniformly bounded in $H^{s}.$

The sharp contrast in behavior between $\mathbb{R}^{n}$ and $\mathbb{T}^{d}$
has generated considerable interest in questions of long time behavior
on the product spaces. Considering the product spaces $\mathbb{R}^{n}\times\mathbb{T}^{d}$,
there is the expectation that at least if $\frac{np}{2}\geq1$ the
solutions will globally exist and decay like $t^{-\frac{np}{2}}$
for sufficiently small initial data. When $\frac{np}{2}>1$, the global
solutions scatter to linear solutions. When $\frac{np}{2}=1$, the
global solutions exhibit some modified scattering. This paper offers
a specific case of the latter scenario, $n=1,\quad p=2,\quad1\leq d\leq4$. 

\begin{theorem}\label{Theoremmain} Let $1\leq d\leq4$. Consider
the equation (\ref{eq main}) with initial data $u_{0}$, which satisfies
\begin{equation}
\left\Vert xu_{0}\right\Vert _{L_{x,y}^{2}}+\left\Vert D_{y}^{s}u_{0}\right\Vert _{L_{x,y}^{2}}\leq\epsilon,\label{eq:ini}
\end{equation}
where $s=3\alpha$, $\alpha>\frac{d}{2}$ is an arbitrary positive
number, and $\epsilon<\epsilon(d)$. Then: 

(a) There exists a unique global solution $u\in C(\mathbb{R};L_{x}^{2}H_{y}^{s})$
for (\ref{eq main}) with initial data $u_{0}$, and $xe^{-it\partial_{x}^{2}/2}u\in C\left(\mathbb{R};L_{x,y}^{2}\right).$
Moreover we have the time decay estimate
\begin{equation}
\left\Vert u(t)\right\Vert _{L_{x}^{\infty}H_{y}^{1}}\lesssim\epsilon\left|t\right|^{-\frac{1}{2}},\label{eq:maindecay}
\end{equation}
and the energy bound
\begin{equation}
\left(\left\Vert xe^{-it\partial_{x}^{2}/2}u\right\Vert _{L_{x,y}^{2}}^{2}+\left\Vert D_{y}^{s}u\right\Vert _{L_{x,y}^{2}}^{2}\right)^{\frac{1}{2}}\lesssim\epsilon\left(1+t\right)^{C\epsilon^{2}}.\label{eq:mainenergy}
\end{equation}

(b) There exists $W\left(t,v,y\right)\in C\left([1,\infty);L_{v,y}^{2}\cap L_{v}^{\infty}H_{y}^{\alpha}\right)$,
which along rays $v=\text{constant}$ is a solution to the equation
\begin{equation}
i\partial_{t}W+\frac{1}{2}\triangle_{y}W=\frac{1}{t}\left|W\right|^{2}W,\label{eq:eqtorus}
\end{equation}
 such that for $t\geq1$
\begin{equation}
u\left(t,x,y\right)=\frac{1}{\sqrt{t}}e^{i\frac{x^{2}}{2t}}W\left(t,\frac{x}{t},y\right)+err_{x}\label{eq:mainerr}
\end{equation}
where 
\[
err_{x}\in\epsilon\left(O_{L_{v,y}^{2}}\left(\left(1+t\right)^{-\frac{1}{2}+C\epsilon^{2}}\right)\cap O_{L_{x}^{\infty}H_{y}^{\alpha}}\left(\left(1+t\right)^{-\frac{7}{12}+C\epsilon^{2}}\right)\right).
\]
\end{theorem}

A similar statement holds as $t\rightarrow-\infty.$

Complementing the above scattering result, we also have the asymptotic
completeness property.

\begin{theorem}\label{theoremasycom} Let $1\leq d\leq4$ and $C$
be a large universal constant. There exists $\epsilon=\epsilon(d)>0$
such that if $W_{1}$ satisfies 
\[
\left\Vert D_{v}^{1+C\epsilon^{2}}W_{1}\right\Vert _{L_{v,y}^{2}}+\left\Vert D_{y}^{s}D_{v}^{C\epsilon^{2}}W_{1}\right\Vert _{L_{v,y}^{2}}+\left\Vert D_{y}^{s}W_{1}\right\Vert _{L_{v,y}^{2}}\ll\epsilon\ll1,
\]
then there exists $W(t)$ solving (\ref{eq:eqtorus}) on $t\in[1,\infty)$
with initial data $W(1)=W_{1}$. Moreover, there exists a solution
$u$ of (\ref{eq main}) with initial data $u_{0}$ which satisfies
(\ref{eq:ini}), hence (\ref{eq:mainerr}) holds for $u$.

\end{theorem}

The work starts with the proof of Theorem \ref{Theoremmain} (a).
In section \ref{sec:Small-Data-Scattering}, first we prove that when
$\epsilon$ is small enough, there exists $T\approx e^{\frac{C}{\epsilon^{2}}}\gg1$
such that the local well-posedness and (\ref{eq:maindecay}), (\ref{eq:mainenergy})
hold on the interval $[0,T]$. These imply that Theorem \ref{Theoremmain}
(b) holds on $[0,T]$. Then in section\ref{sec:The-Energy-Estimate},
by a more careful analysis (\ref{eq:eqtorus}), (\ref{eq:mainerr})
will give us better bounds for (\ref{eq:maindecay}), (\ref{eq:mainenergy}).
Hence the interval $[0,T]$ can be extended to $[0,\infty)$. For
the second part, in section \ref{sec:Asymptotic-Completeness} we
will show Theorem \ref{theoremasycom} by applying a contraction mapping
argument to the resulting equation for the difference $\tilde{w}:=u-u_{app}.$
Here $u_{app}$ is an approximate solution constructed from $W$.

\begin{remark} In the original paper \cite{HPTV} the authors prove
instead that $u$ approaches a solution for the resonant equation
(\ref{eq dyna}). Here we use (\ref{eq:eqtorus}) to characterize
the asymptotic profile, which is similar to cubic NLS equation on
$\mathbb{T}^{d}.$ We will prove later in Proposition \ref{prop:wproperty}
that these two forms actually are equivalent.

\end{remark}

In this paper, the author uses the wave packet method, originally
developed in the work of Ifrim and Tataru on the $1d$ cubic NLS\cite{IT}
and $2d$ water waves \cite{IT3}\cite{IT4}. Let $\mathcal{X}\in C_{0}^{\infty}(\mathbb{R})$
be a real-valued function localized in both space and frequency near
$0$ at scale $\sim1$. To simplify the computations, we normalized
$\int\mathcal{X}=1$. We define a wave packet adapted to the ray $\varUpsilon_{v}=\{x=tv\}$
by 
\begin{equation}
\Psi_{v}=e^{i\phi}\mathcal{X}\left(\frac{x-tv}{\sqrt{t}}\right),
\end{equation}
where the phase function is defined by 
\[
\phi=\frac{x^{2}}{2t}.
\]
We expect the approximate solution will stay coherent on the time
scale $\Delta t\ll t$. By direct computations, we have 
\begin{equation}
\left(i\partial_{t}+\frac{1}{2}\partial_{x}^{2}\right)\Psi_{v}=\frac{1}{2t}e^{i\phi}\left[t^{\frac{1}{2}}\mathcal{X}^{\prime}\left(\frac{x-tv}{\sqrt{t}}\right)+i\left(x-vt\right)\mathcal{X}\left(\frac{x-tv}{\sqrt{t}}\right)\right].
\end{equation}
To measure the decay of $u$ along the ray $\varUpsilon_{v}$ we use
the function 
\[
\gamma\left(t,v\right)=\int u\overline{\Psi}_{v}dxdy.
\]
The function $\gamma$ satisfies the differential equation
\begin{equation}
i\partial_{t}\gamma+\frac{1}{2}\triangle_{y}\gamma=\frac{1}{t}\left|\gamma\right|^{2}\gamma+R(t,v).
\end{equation}
Let $r_{i}$ be some fixed numbers greater that $0$. The remainder
$R(t,v)$ satisfies the bound
\[
\left\Vert R(t,v)\right\Vert _{Y}\lesssim t^{-1-r_{1}}\left\Vert u(t)\right\Vert _{X^{+}},
\]
where $\left\Vert u(t)\right\Vert _{X^{+}}$ is an appropriate energy.
The definition of $Y$ and $X^{+}$ are given in (\ref{Ynorm}) and
(\ref{X+norm}) respectively.

To close the bootstrap argument, we need to prove that under the asymptotic
relation (\ref{eq:mainerr}), the energy $\left\Vert u(t)\right\Vert _{X^{+}}$
grows slower than $t^{r}$. The main difficulty arises in the energy
cascade phenomenon, in which the energy of the system moves from low
frequencies towards arbitrarily high frequencies. We need the relation
\[
\partial_{t}\left\Vert u(t)\right\Vert _{X^{+}}^{2}\lesssim\left\Vert u(t)\right\Vert _{L^{\infty}}^{2}\left\Vert u(t)\right\Vert _{X^{+}}^{2}\lesssim t^{-1}\left\Vert u(t)\right\Vert _{X^{+}}^{2}
\]
to prove the desired bound. Since for $d\geq2$ ,and any $s>1$, the
$t^{-\frac{1}{2}}$ decay of $\left\Vert u(t)\right\Vert _{L_{v}^{\infty}H_{y}^{s}}$
does not hold, it barely allows to close any polynomial-growth bootstrap
for the energy $\left\Vert u(t)\right\Vert _{X^{+}}$. Therefore we
apply normal form correction to non-resonant frequencies on long time
scale, and apply sharp resonant estimate on short time scale intervals.
The factors with non-resonant frequencies can be bound by $t^{-1-r_{2}}\left\Vert u(t)\right\Vert _{X^{+}}^{2}$
and the factors with resonant frequencies can be bound by $\left\Vert u(t)\right\Vert _{L_{v}^{\infty}H_{y}^{1}}^{2}\left\Vert u(t)\right\Vert _{X^{+}}^{2}$.
Hence we have
\[
\partial_{t}\left\Vert u(t)\right\Vert _{X^{+}}^{2}\lesssim t^{-1-r_{2}}\left\Vert u(t)\right\Vert _{X^{+}}^{2}+\left\Vert u(t)\right\Vert _{L_{v}^{\infty}H_{y}^{1}}^{2}\left\Vert u(t)\right\Vert _{X^{+}}^{2}\lesssim t^{-1-r_{2}}\left\Vert u(t)\right\Vert _{X^{+}}^{2}+t^{-1}\left\Vert u(t)\right\Vert _{X^{+}}^{2},
\]
which allows closing the bootstrap argument.

\subsection{Standard notations.}

In this section we briefly collect some notations, definitions and
estimates used throughout this thesis. Given two quantities $A$\textbf{,$B$}
we will write $A\lesssim B$ if there exists some constant $C>0$
so that $A\leq CB$, and write $A\sim B$ if $A\lesssim B$ and $B\lesssim A$.
If $C=C(k)$ we will write $A\lesssim_{k}B$. We write $A\ll B$ if
$A\lesssim B$ and the constant is sufficiently small. We denote the
sets of integers, real numbers and complex numbers by $\mathbb{Z}$,
$\mathbb{R}$ and $\mathbb{C}$ respectively. If $E\subset\mathbb{R}^{n}$
we denote the indicator function of the set $E$ by $\mathbf{1}_{E}$.
We denote the Euclidean norm by $\left|\cdot\right|$ and define the
bracket $\left\langle \cdot\right\rangle =\left(1+\left|\cdot\right|^{2}\right)^{\frac{1}{2}}$.
If $X$ is a normed space we denote its norm by $\left\Vert \cdot\right\Vert _{X}$.
We denote the torus $\mathbb{R}/(2\pi\mathbb{Z})$ by $\mathbb{T}$.

In this work, we define the Fourier transform $\mathcal{F}_{x}f$
and $\widehat{f}$ with respect to $x$ by 
\[
\mathcal{F}_{x}f:=\frac{1}{\sqrt{2\pi}}\int_{\mathbb{R}}e^{-ix\xi}f\left(x\right)dx.
\]
Similarly we also have the full spatial Fourier transform with respect
to both $x$ and $y$, 
\[
\left(\mathcal{F}f\right)\left(\xi,\mathbf{k}\right):=\frac{1}{\left(2\pi\right)^{\left(d+1\right)/2}}\int_{\mathbb{R}}\int_{\mathbb{T}^{d}}f\left(x,y\right)e^{-ix\xi}e^{-i\left\langle \mathbf{k},y\right\rangle }dydx.
\]
Since we need to switch frequently between $f\left(v,y\right)$ and
$\left(\mathcal{F}_{y}f\right)\left(v,\mathbf{k}\right)$ in this
work, we use the bold character for the Fourier transform in the $y$
variable:
\[
f\left(t,v,\mathbf{k}\right)=\frac{1}{\left(2\pi\right)^{d/2}}\int_{\mathbb{T}^{d}}f\left(t,v,y\right)e^{-i\left\langle \mathbf{k},y\right\rangle }dy.
\]

We use $\partial_{x}f$, $f_{x}$ and $f^{\prime}$ to denote a (partial)
derivative in the variable $x$. We define the fractional derivative
by using the Fourier transform
\[
D^{\alpha}f=\mathcal{F}_{x}^{-1}\left|\xi\right|^{\alpha}\widehat{f}.
\]

If $X$ is a normed space and $I\subset\mathbb{R}$ is an interval,
we denote the space of continuous functions $f:I\rightarrow X$ by
$C\left(I;X\right)$ equipped with the sup norm. We use the notation
$C^{k}$ to denote $k$-continuously differentiable functions; $C^{\infty}:=\cap C^{k}$
to denote smooth functions; $C_{c}^{\infty}$ to denote compact supported
$C^{\infty}$ functions.

For $1<p<\infty$ we use $L_{x}^{p}(F)$ (where $F=\mathbb{R}$ or
$\mathbb{C}$) to denote the space of Lebesgue-measurable functions
$f:X\rightarrow F$ such that 

\[
\left\Vert f\right\Vert _{L^{p}}^{p}=\int\left|f(x)\right|^{p}dx<\infty,
\]
with the usual modification for $p=\infty$. We will typically omit
the domain and codomain when they are evident. We denote the $L^{2}$-
inner product by 
\[
\left\langle u,v\right\rangle =\int_{\mathbb{R}}u(x)\overline{v}(x)dx.
\]
We will use Littlewood-Paley projections in $x$ 
\[
\left(\mathcal{F}P_{\leq N}f\right)\left(\xi,k\right)=\mathcal{X}\left(\frac{\xi}{N}\right)\mathcal{F}f\left(\xi,k\right),
\]
where $\mathcal{X}\in C_{c}^{\infty}\left(\mathbb{R}\right),$ $\mathcal{X}\left(x\right)=1$
when $|x|\leq1$ and $\mathcal{X}\left(x\right)=0$ when $|x|\geq2$.
Next, define 
\[
P_{N}=P_{\leq N}-P_{\leq N/2},\quad P_{\geq N}=1-P_{\leq N/2}.
\]
When we concentrate on the frequency in $y$ only, we will denote
\[
\left(\mathcal{F}P_{\leq N}^{y}f\right)\left(\xi,\mathbf{k}\right)=\varphi\left(\frac{\left|\mathbf{k}\right|}{N}\right)\left(\mathcal{F}f\right)\left(\xi,\mathbf{k}\right).
\]
 We denote the linear Schr\"{o}dinger operator $e^{it\partial_{x}^{2}/2}$
on $\mathbb{R}$ by 
\[
e^{it\partial_{x}^{2}/2}f=\mathcal{F}_{x}^{-1}e^{-it|\xi|^{2}/2}\widehat{f}.
\]
Similarly, the linear Schr\"{o}dinger operator $e^{it\triangle_{y}/2}$
on $\mathbb{T}^{d}$ is given by 
\[
e^{it\triangle_{y}/2}f=\mathcal{F}_{y}^{-1}e^{-it|\mathbf{k}|^{2}/2}\mathcal{F}_{y}f.
\]
 Define the linear Schr\"{o}dinger evolution on $\mathbb{R}\times\mathbb{T}^{d}$
as $U(t)=e^{it\triangle_{y}/2}e^{it\partial_{x}^{2}/2}.$

\subsubsection{Norms and $H^{s}$ spaces}

Define the weighted Sobolev norm for $x\in\mathbb{R}$ by 

\begin{equation}
\left\Vert f\right\Vert _{H_{x}^{0,1}}^{2}=\left\Vert f\right\Vert _{L_{x}^{2}}^{2}+\left\Vert xf\right\Vert _{L_{x}^{2}}^{2}.
\end{equation}
Define the Sobolev norm $H_{x}^{s}$ for variable $x\in\mathbb{R}^{n}$
by 
\begin{equation}
\left\Vert f\right\Vert _{H_{x}^{s}}=\left\Vert \left(1+\left|\xi\right|^{2}\right)^{\frac{s}{2}}\left(\mathcal{F}_{x}f\right)\left(\xi\right)\right\Vert _{L_{\xi}^{2}}.
\end{equation}
Define the Sobolev norm $H_{y}^{s}$ for variable $y\in\mathbb{T}^{n}$
by
\begin{equation}
\left\Vert f\right\Vert _{H_{y}^{s}}=\left[\sum_{\mathbf{k}\in\mathbb{Z}^{n}}\left(\left(1+\left|\mathbf{k}\right|^{2}\right)^{\frac{s}{2}}\left|\left(\mathcal{F}_{y}f\right)\left(\mathbf{k}\right)\right|\right)^{2}\right]^{\frac{1}{2}}.
\end{equation}
 Define the homogeneous Sobolev norm $\dot{H}_{x}^{s}$ for variable
$x\in\mathbb{R}^{n}$ by 
\begin{equation}
\left\Vert f\right\Vert _{\dot{H}_{x}^{s}}=\left\Vert \left|\xi\right|^{\frac{s}{2}}\left(\mathcal{F}_{x}f\right)\left(\xi\right)\right\Vert _{L_{\xi}^{2}}.
\end{equation}
Define the homogeneous Sobolev norm $\dot{H}_{y}^{s}$ for variable
$y\in\mathbb{T}^{n}$ by 
\begin{equation}
\left\Vert f\right\Vert _{\dot{H}_{y}^{s}}=\left[\sum_{\mathbf{k}\in\mathbb{Z}^{n}}\left(\left|\mathbf{k}\right|^{\frac{s}{2}}\left|\left(\mathcal{F}_{y}f\right)\left(\mathbf{k}\right)\right|\right)^{2}\right]^{\frac{1}{2}}.
\end{equation}

We now define several norms which will be used in our problem, namely
the cubic NLS on $\mathbb{R}\times\mathbb{T}^{d}$. For local well-posedness,
we use the norm $X(I)$. For energy estimates we use the norm $X^{+}$.
For uniform bounds and the scattering itself, we use the norm $Y$.
These norms are defined as follows: 

When proving the local well-posedness result, we use the following
norm $X\left(I\right)$:
\begin{equation}
\left\Vert f\right\Vert _{X\left(I\right)}:=\left\Vert f\right\Vert _{L_{t}^{\infty}\left(I;L_{x}^{2}H_{y}^{s}\right)}+\left\Vert f\right\Vert _{L_{t}^{4}\left(I,L_{x,y}^{\infty}\right)},\label{normXI-1}
\end{equation}
where $I$ is some finite interval with length less than or equal
to 1.

The norm we will use to measure the size of the solutions will be
denoted as $X^{+}$. We expect this quantity to grow as $t^{\delta}$
where $\delta$ is a small positive number depending on the size $\epsilon$
of the initial data $u_{0}$: 
\begin{equation}
\left\Vert u\left(t\right)\right\Vert _{X^{+}}^{2}:=\left\Vert xe^{-it\partial_{x}^{2}/2}u\left(t\right)\right\Vert _{L_{x,y}^{2}}^{2}+\left\Vert D_{y}^{s}u\left(t\right)\right\Vert _{L_{x,y}^{2}}^{2}.\label{X+norm}
\end{equation}

While using the energy method, we also define the norm as:

\begin{equation}
\left\Vert f\right\Vert _{Y}=\left\Vert f\right\Vert _{L_{x}^{\infty}H_{y}^{\alpha}}+\left\Vert f\right\Vert _{L_{x,y}^{2}},\label{Ynorm}
\end{equation}
We will prove that the solution tends to the modified scattering profile
in the space $Y$. 

\subsubsection{Forms }

Here we introduce some sets and trilinear forms associated with the
cubic NLS on the torus.

We will use the following sets corresponding to momentum and resonance
level sets. For a fixed $\mathbf{k}$, we define the set 
\begin{equation}
\mathcal{M}\left(\mathbf{k}\right)=\left\{ \left(\mathbf{k}_{1},\mathbf{k}_{2},\mathbf{k}_{3},\mathbf{k}\right)\in\mathbb{Z}^{4d}:\mathbf{k}_{1}-\mathbf{k}_{2}+\mathbf{k}_{3}-\mathbf{k}=0\right\} ,
\end{equation}
which describes the triples of frequencies which yield output $\mathbf{k}$,
and the resonance level 
\begin{equation}
\Gamma_{\omega}\left(\mathbf{k}\right)=\left\{ \left(\mathbf{k}_{1},\mathbf{k}_{2},\mathbf{k}_{3},\mathbf{k}\right)\in\mathcal{M}:|\mathbf{k}_{1}|^{2}-|\mathbf{k}_{2}|^{2}+|\mathbf{k}_{3}|^{2}-|\mathbf{k}|^{2}=\omega\right\} .
\end{equation}
In particular, for the resonance level $\omega=0$,
\begin{equation}
\Gamma_{0}\left(\mathbf{k}\right)=\left\{ \left(\mathbf{k}_{1},\mathbf{k}_{2},\mathbf{k}_{3},\mathbf{k}\right)\in\mathbb{Z}^{4d}:\quad\mathbf{k}_{1}-\mathbf{k}_{2}+\mathbf{k}_{3}-\mathbf{k}=0,\quad|\mathbf{k}_{1}|^{2}-|\mathbf{k}_{2}|^{2}+|\mathbf{k}_{3}|^{2}-|\mathbf{k}|^{2}=0\right\} .\label{def:gamma_0}
\end{equation}
Restricting the trilinear interactions in the expression $u\rightarrow|u|^{2}u$
to resonant ones we obtain the resonant trilinear form 
\begin{align*}
R\left[f_{1},f_{2},f_{3}\right]\left(t,v,\mathbf{k}\right) & =\sum_{\Gamma_{0}\left(\mathbf{k}\right)}f_{1}\left(t,v,\mathbf{k}_{1}\right)\overline{f_{2}}\left(t,v,\mathbf{k}_{2}\right)f_{3}\left(t,v,\mathbf{k}_{3}\right).
\end{align*}
More generally, we will also consider the non-resonant trilinear forms
\[
\mathcal{E}\left[f_{1},f_{2},f_{3}\right]\left(t,v,y\right):=\sum_{\Gamma_{\omega}\left(\mathbf{k}\right),\omega\neq0}e^{i\omega t/2}f_{1}\left(t,v,\mathbf{k}_{1}\right)\overline{f_{2}}\left(t,v,\mathbf{k}_{2}\right)f_{3}\left(t,v,\mathbf{k}_{3}\right)e^{i\mathbf{k}\cdot y}.
\]
The expression

\[
\mathcal{D}\left[f_{1},f_{2},f_{3}\right]\left(t,v,y\right):=\frac{2}{t}\sum_{\Gamma_{\omega}\left(\mathbf{k}\right),\omega\neq0}\frac{e^{i\omega t/2}}{i\omega}f_{1}\left(t,v,\mathbf{k}_{1}\right)\overline{f_{2}}\left(t,v,\mathbf{k}_{2}\right)f_{3}\left(t,v,\mathbf{k}_{3}\right)e^{i\mathbf{k}\cdot y},
\]
will be used to define an energy correction later on.

\section{Small Data Scattering\label{sec:Small-Data-Scattering}}

\subsection{Local well-posedness}

Before we study the long time behavior, it is necessary to consider
the local well-posedness of equation (\ref{eq main}). Here we introduce
the vector field $L_{x}:=x+it\partial_{x}$,which is the conjugate
of $x$ with respect to the linear flow, $U(t)x=L_{x}U(t).$ The vector
field is also the generator for the Galilean group of symmetries.
The function $L_{x}u$ satisfies the following equation, 

\begin{equation}
\left(i\partial_{t}+\frac{1}{2}\partial_{x}^{2}+\frac{1}{2}\triangle_{y}\right)L_{x}u=2|u|^{2}L_{x}u-u^{2}\overline{L_{x}u},\label{eq:lin}
\end{equation}
which is the linearized equation of (\ref{eq main}). The operator
$L_{x}$ allows us to capture the effect of the initial data localization
$xu_{0}\in L_{x,y}^{2}$ as the time increases.
\begin{prop}
\label{prop:localwell}The equation (\ref{eq main}) is locally well-posed
for initial data $u_{0}\in H_{x}^{0,1}L_{y}^{2}\cap L_{x}^{2}H_{y}^{s}$
for any $s>\frac{d}{2}$. For such data we have a unique local solution
$u\in C\left(I;L_{x}^{2}H_{y}^{s}\cap L_{t}^{4}L_{x,y}^{\infty}\right)$.
If in addtion $xu_{0}\in L_{x,y}^{2}$, then the solution has $L_{x}u\in C\left(I;L_{x,y}^{2}\right).$ 
\end{prop}
In the case $d\leq3$, global existence can be established in a much
more general setting ($u_{0}\in H_{x,y}^{1}$), see \cite{IP}. 

\begin{proof}

Using the normed space $X$ defined in (\ref{normXI}) and the Strichartz
estimate from Lemma \ref{strichartz}, for arbitrary functions $f,g\in L_{x}^{2}H_{y}^{s}$,
there is the inequality 
\begin{equation}
\left\Vert \int_{0}^{t}U(t-s)\left(\left|f\right|^{2}f-\left|g\right|^{2}g\right)\left(s\right)ds\right\Vert _{X\left(I\right)}\lesssim\left(\left|I\right|^{\frac{1}{2}}+\left|I\right|^{\frac{3}{4}}\right)\left(\left\Vert f\right\Vert _{X\left(I\right)}^{2}+\left\Vert g\right\Vert _{X\left(I\right)}^{2}\right)\left\Vert f-g\right\Vert _{X\left(I\right)}.\label{eq:wellposed}
\end{equation}
The estimate (\ref{eq:wellposed}) allows us to obtain the unique
local solution through the contraction principle, if we make the interval
$I$ small enough. Therefore equation (\ref{eq:lin}) is locally well-posed
in the space $X\left(I\right).$ Due to the lack of uniform $L_{x}^{2}H_{y}^{s}$
estimates, for $d>1$ we can not extend the iteration directly to
global well-posedness. Later on we will prove that for solutions with
suitable small initial data, we will have the bound $\left\Vert u(t)\right\Vert _{L_{x}^{2}H_{y}^{s}}\lesssim t^{\delta}$,
which implies the global well-posedness. To prove the local well-posedness
for the linearized equation (\ref{eq:lin}), for any $f,g\in L_{x,y}^{2}$,
we have the following Strichartz inequality: 
\begin{equation}
\left\Vert \int_{0}^{t}U(t-s)\left(\left|u\right|^{2}f-\left|u\right|^{2}g\right)\left(s\right)ds\right\Vert _{L_{t}^{\infty}\left(I;L_{x,y}^{2}\right)}\lesssim\left|I\right|^{\frac{1}{2}}\left\Vert u\right\Vert _{L_{t}^{4}\left(I,L_{x,y}^{\infty}\right)}^{2}\left\Vert f-g\right\Vert _{L_{t}^{\infty}\left(I;L_{x,y}^{2}\right)}.
\end{equation}
When the interval $|I|$ is small enough, we obtain the desired properties.
A similar bound can be applied to the nonlinear term $u^{2}\overline{L_{x}u}$.
Hence the local well-posedness for the equation (\ref{eq:lin}) in
$L_{t}^{\infty}\left(I;L_{x,y}^{2}\right)$ can be readily obtained
once we have the well-posedness for the equation of $u$ in $X\left(I\right).$

\end{proof}

\subsection{The asymptotic equation}

Here we use the wave packet testing method following the work of Ifrim
and Tataru on the $1d$ cubic NLS \cite{IT}. A wave packet in the
context here is an approximate solution to the linear system with
$\mathcal{O}\left(1/t\right)$ errors. For each trajectory $\varUpsilon_{v}:=\left\{ x=vt\right\} $
traveling with velocity $v$ we establish decay along this ray by
testing with a wave packet moving along the ray. Here we use a slightly
different notation; one can verify that the function $\gamma$ here
is the same as in the original paper \cite{IT}. Define

\begin{equation}
w(t,v,y):=t^{\frac{1}{2}}e^{-i\frac{tv^{2}}{2}}u(t,tv,y),
\end{equation}
\begin{equation}
\gamma(t,v,y):=P_{\leq\sqrt{t}}w(t,v,y).
\end{equation}

Our first result asserts that $u$ is well approximated by a function
ass

ociated with $\gamma$ for as long as we have good control of the
energy bounds $\left\Vert L_{x}u\right\Vert _{L_{x,y}^{2}}$ and $\left\Vert D_{y}^{s}u\right\Vert _{L_{x,y}^{2}}$. 

\begin{lemma} The functions $\gamma$ and $w$ satisfy the following
bounds for any $\alpha\geq0$: 
\begin{equation}
\left\Vert \gamma\right\Vert _{L_{v}^{\infty}L_{y}^{2}\left(L_{v}^{\infty}H_{y}^{s}\right)}\lesssim\left\Vert w\right\Vert _{L_{v}^{\infty}L_{y}^{2}\left(L_{v}^{\infty}H_{y}^{s}\right)}=t^{\frac{1}{2}}\left\Vert u\right\Vert _{L_{x}^{\infty}L_{y}^{2}\left(L_{x}^{\infty}H_{y}^{s}\right)},
\end{equation}
\begin{equation}
\left\Vert D_{y}^{s}\gamma\right\Vert _{L_{v,y}^{2}}\leq\left\Vert D_{y}^{s}w\right\Vert _{L_{v,y}^{2}}=\left\Vert D_{y}^{s}u\right\Vert _{L_{x,y}^{2}},\quad\left\Vert \partial_{v}\gamma\right\Vert _{L_{v,y}^{2}}\leq\left\Vert \partial_{v}w\right\Vert _{L_{v,y}^{2}}=\left\Vert L_{x}u\right\Vert _{L_{x,y}^{2}},\label{eq:gammabound}
\end{equation}
\begin{equation}
\left\Vert \gamma\right\Vert _{L_{v}^{\infty}H_{y}^{\alpha}}\lesssim\left\Vert w\right\Vert _{L_{v}^{\infty}H_{y}^{\alpha}}\lesssim\left\Vert u\right\Vert _{L_{x,y}^{2}}^{\frac{1}{6}}\left\Vert L_{x}u\right\Vert _{L_{x,y}^{2}}^{\frac{1}{2}}\left\Vert D_{y}^{3\alpha}u\right\Vert _{L_{x,y}^{2}}^{\frac{1}{3}}.\label{eq:alpha bound}
\end{equation}
We also have the physical space bounds

\begin{align}
\left\Vert u\left(t,x,y\right)-\frac{1}{\sqrt{t}}e^{-i\frac{x^{2}}{2t}}\gamma\left(t,\frac{x}{t},y\right)\right\Vert _{L_{x}^{\infty}H_{y}^{\alpha}} & \lesssim t^{-\frac{7}{12}}\left\Vert L_{x}u\right\Vert _{L_{x,y}^{2}}^{\frac{2}{3}}\left\Vert D_{y}^{3\alpha}u\right\Vert _{L_{x,y}^{2}}^{\frac{1}{3}},\label{eq:diffalpha}
\end{align}
\begin{align}
\left\Vert u\left(t,x,y\right)-\frac{1}{\sqrt{t}}e^{-i\frac{x^{2}}{2t}}\gamma\left(t,\frac{x}{t},y\right)\right\Vert _{L_{x,y}^{2}} & \lesssim t^{-\frac{1}{2}}\left\Vert L_{x}u\right\Vert _{L_{x,y}^{2}},\label{eq:diffL2}
\end{align}
and the Fourier space bounds 
\begin{equation}
\left\Vert \hat{u}\left(t,\xi,k\right)-e^{-it\xi^{2}/2}\gamma\left(t,\xi,\mathbf{k}\right)\right\Vert _{L_{\xi}^{2}l_{\mathbf{k}}^{2}}\lesssim t^{-\frac{1}{2}}\left\Vert L_{x}u\right\Vert _{L_{x,y}^{2}}.
\end{equation}
\end{lemma} 

\begin{proof} By Bernstein's inequality and interpolation, we have
the straightforward bounds: 
\begin{align*}
\left\Vert \gamma\right\Vert _{L_{v}^{\infty}H_{y}^{\alpha}}\lesssim\left\Vert w\right\Vert _{L_{v}^{\infty}H_{y}^{\alpha}} & \lesssim\left\Vert D_{v}^{\frac{3}{4}}w\right\Vert _{L_{v,y}^{2}}^{\frac{2}{3}}\left\Vert D_{y}^{3\alpha}w\right\Vert _{L_{v,y}^{2}}^{\frac{1}{3}}\lesssim\left\Vert u\right\Vert _{L_{x,y}^{2}}^{\frac{1}{6}}\left\Vert L_{x}u\right\Vert _{L_{x,y}^{2}}^{\frac{1}{2}}\left\Vert D_{y}^{3\alpha}u\right\Vert _{L_{x,y}^{2}}^{\frac{1}{3}},
\end{align*}

\begin{align}
\left\Vert w-\gamma\right\Vert _{L_{v}^{\infty}H_{y}^{\alpha}} & \lesssim\left\Vert D_{v}^{\frac{3}{4}}P_{\geq\sqrt{t}}w\right\Vert _{L_{v,y}^{2}}^{\frac{2}{3}}\left\Vert D_{y}^{3\alpha}w\right\Vert _{L_{v,y}^{2}}^{\frac{1}{3}}\lesssim t^{-\frac{1}{12}}\left\Vert L_{x}u\right\Vert _{L_{x,y}^{2}}^{\frac{2}{3}}\left\Vert D_{y}^{3\alpha}u\right\Vert _{L_{x,y}^{2}}^{\frac{1}{3}},
\end{align}
and

\begin{align}
\left\Vert w-\gamma\right\Vert _{L_{v,y}^{2}} & \lesssim t^{-\frac{1}{2}}\left\Vert \partial_{v}w\right\Vert _{L_{v,y}^{2}}\lesssim t^{-\frac{1}{2}}\left\Vert L_{x}u\right\Vert _{L_{x,y}^{2}}.
\end{align}
\end{proof}
From now on, we will assume that $\alpha=\frac{d}{2}^{+}$and $s=3\alpha$.

The next objective is to show that $\gamma$ is an approximate solution
to the asymptotic equation (\ref{eq:eqtorus}). 

\begin{lemma}\label{uasymptote} If $u$ solves (\ref{eq main})
then we have 
\begin{equation}
i\partial_{t}\gamma+\frac{1}{2}\triangle_{y}\gamma=\frac{1}{t}\left|\gamma\right|^{2}\gamma+I,\label{eq:gamma}
\end{equation}
where the remainder $I$ satisfies 
\begin{equation}
\left\Vert I\right\Vert _{L_{v,y}^{2}}\lesssim t^{-\frac{3}{2}}\left(\left\Vert u\right\Vert _{L_{x,y}^{2}}^{\frac{1}{3}}\left\Vert u\right\Vert _{X^{+}}^{\frac{5}{3}}+1\right)\left\Vert u\right\Vert _{X^{+}}\label{eq:L2err}
\end{equation}
\begin{equation}
\left\Vert I\right\Vert _{L_{v}^{\infty}H_{y}^{\alpha}}\lesssim t^{-\frac{13}{12}}\left(\left\Vert u\right\Vert _{L_{x,y}^{2}}^{\frac{1}{3}}\left\Vert u\right\Vert _{X^{+}}^{\frac{5}{3}}+1\right)\left\Vert u\right\Vert _{X^{+}},\label{eq:LinftyYalphaerr}
\end{equation}
where $s=3\alpha$ and $\alpha=\frac{d}{2}^{+}$. \end{lemma}

\begin{proof}Let $\xi$ be the Fourier variable in $v$. A direct
computation yields 
\begin{align*}
\partial_{t}\gamma+\frac{1}{2}\triangle_{y}\gamma & =\mathcal{F}\left[\mathcal{X}^{\prime}\left(\frac{\xi}{\sqrt{t}}\right)\cdot\frac{\xi}{2t^{\frac{3}{2}}}+\frac{\left|\xi\right|^{2}}{2t^{2}}\mathcal{X}\left(\frac{\xi}{\sqrt{t}}\right)\right]\hat{w}+t^{-1}P_{\leq\sqrt{t}}\left|w\right|^{2}w.
\end{align*}
Hence we can write an evolution equation for $\gamma$ of the form
\[
i\partial_{t}\gamma+\frac{1}{2}\triangle_{y}\gamma=\frac{1}{t}\left|\gamma\right|^{2}\gamma+I,
\]
where the error term $I\left(t,v\right)$ can be written as a sum
of three quantities which can be easily bounded: 
\begin{align*}
I & :=\mathcal{F}\left[D\mathcal{X}\left(\frac{\xi}{\sqrt{t}}\right)\cdot\frac{\xi}{2t^{\frac{3}{2}}}+\frac{\left|\xi\right|^{2}}{2t^{2}}\mathcal{X}\left(\frac{\xi}{\sqrt{t}}\right)\right]\hat{w}+t^{-1}P_{\leq\sqrt{t}}\left|w\right|^{2}w-t^{-1}\left|\gamma\right|^{2}\gamma\\
 & :=\mathcal{F}\left[D\mathcal{X}\left(\frac{\xi}{\sqrt{t}}\right)\cdot\frac{\xi}{2t^{\frac{3}{2}}}+\frac{\left|\xi\right|^{2}}{2t^{2}}\mathcal{X}\left(\frac{\xi}{\sqrt{t}}\right)\right]\hat{w}+t^{-1}P_{\leq\sqrt{t}}\left(\left|w\right|^{2}w-\left|\gamma\right|^{2}\gamma\right)+t^{-1}P_{\geq\sqrt{t}}\left|\gamma\right|^{2}\gamma\\
 & :=I_{1}+I_{2}+I_{3}.
\end{align*}

The first term $I_{1}$ can be expressed as a convolution and by Young's
inequality we obtain the bound
\[
\left\Vert I_{1}(t,v,y)\right\Vert _{L_{v,y}^{2}}\lesssim t^{-\frac{3}{2}}\left\Vert P_{\leq\sqrt{t}}\partial_{v}w\right\Vert _{L_{v,y}^{2}}\lesssim t^{-\frac{3}{2}}\left\Vert L_{x}u\right\Vert _{L_{x,y}^{2}}\lesssim t^{-\frac{3}{2}}\left\Vert u\right\Vert _{X^{+}},
\]

\begin{align*}
\left\Vert I_{1}(t,v,y)\right\Vert _{L_{v}^{\infty}H_{y}^{\alpha}} & \lesssim t^{-\frac{5}{4}}\left\Vert P_{\leq\sqrt{t}}\partial_{v}w\right\Vert _{L_{v}^{2}H_{y}^{\alpha}}\lesssim t^{-\frac{13}{12}}\left\Vert P_{\leq\sqrt{t}}D_{v}^{\frac{2}{3}}w\right\Vert _{L_{v}^{2}H_{y}^{\alpha}}\\
 & \lesssim t^{-\frac{13}{12}}\left\Vert P_{\leq\sqrt{t}}\partial_{v}w\right\Vert _{L_{v,y}^{2}}^{\frac{2}{3}}\left\Vert P_{\leq\sqrt{t}}w\right\Vert _{L_{v}^{2}H_{y}^{3\alpha}}^{\frac{1}{3}}\\
 & \lesssim t^{-\frac{13}{12}}\left\Vert u\right\Vert _{X^{+}}.
\end{align*}
For the second and third terms $I_{2},I_{3}$, we apply Bernstein's
inequality in order to get: 
\begin{align*}
\left\Vert I_{2}\right\Vert _{L_{v,y}^{2}} & \lesssim t^{-1}\left(\left\Vert w\right\Vert _{L_{v}^{\infty}H_{y}^{\alpha}}^{2}+\left\Vert \gamma\right\Vert _{L_{v}^{\infty}H_{y}^{\alpha}}^{2}\right)\left\Vert w-\gamma\right\Vert _{L_{v,y}^{2}}\lesssim t^{-\frac{3}{2}}\left(\left\Vert w\right\Vert _{L_{v}^{\infty}H_{y}^{\alpha}}^{2}+\left\Vert \gamma\right\Vert _{L_{v}^{\infty}H_{y}^{\alpha}}^{2}\right)\left\Vert L_{x}u\right\Vert _{L_{x,y}^{2}}\\
 & \lesssim t^{-\frac{3}{2}}\left\Vert u\right\Vert _{L_{x,y}^{2}}^{\frac{1}{3}}\left\Vert L_{x}u\right\Vert _{L_{x,y}^{2}}^{2}\left\Vert D_{y}^{s}u\right\Vert _{L_{x,y}^{2}}^{\frac{2}{3}}\lesssim t^{-\frac{3}{2}}\left\Vert u\right\Vert _{L_{x,y}^{2}}^{\frac{1}{3}}\left\Vert u\right\Vert _{X^{+}}^{\frac{8}{3}}.
\end{align*}
\begin{align*}
\left\Vert I_{2}\right\Vert _{L_{v}^{\infty}H_{y}^{\alpha}} & \lesssim t^{-1}\left\Vert \left|w\right|^{2}w-\left|\gamma\right|^{2}\gamma\right\Vert _{L_{v}^{\infty}H_{y}^{\alpha}}\lesssim t^{-1}\left(\left\Vert w\right\Vert _{L_{v}^{\infty}H_{y}^{\alpha}}^{2}+\left\Vert \gamma\right\Vert _{L_{v}^{\infty}H_{y}^{\alpha}}^{2}\right)\left\Vert w-\gamma\right\Vert _{L_{v}^{\infty}H_{y}^{\alpha}}\\
 & \lesssim t^{-\frac{13}{12}}\left(\left\Vert w\right\Vert _{L_{v}^{\infty}H_{y}^{\alpha}}^{2}+\left\Vert \gamma\right\Vert _{L_{v}^{\infty}H_{y}^{\alpha}}^{2}\right)\left\Vert L_{x}u\right\Vert _{L_{x,y}^{2}}^{\frac{2}{3}}\left\Vert D_{y}^{s}u\right\Vert _{L_{x,y}^{2}}^{\frac{1}{3}}\\
 & \lesssim t^{-\frac{13}{12}}\left\Vert u\right\Vert _{L_{x,y}^{2}}^{\frac{1}{3}}\left\Vert u\right\Vert _{X^{+}}^{\frac{8}{3}}.
\end{align*}
Similarly, we have: 
\begin{alignat*}{1}
\left\Vert I_{3}\right\Vert _{L_{v,y}^{2}} & \lesssim t^{-\frac{3}{2}}\left\Vert \partial_{v}\left|\gamma\right|^{2}\gamma\right\Vert _{L_{v,y}^{2}}\lesssim t^{-\frac{3}{2}}\left\Vert \gamma\right\Vert _{L_{v}^{\infty}H_{y}^{\alpha}}^{2}\left\Vert \partial_{v}\gamma\right\Vert _{L_{v,y}^{2}}\lesssim t^{-\frac{3}{2}}\left\Vert \gamma\right\Vert _{L_{v}^{\infty}H_{y}^{\alpha}}^{2}\left\Vert L_{x}u\right\Vert _{L_{x,y}^{2}}\\
 & \lesssim t^{-\frac{3}{2}}\left\Vert u\right\Vert _{L_{x,y}^{2}}^{\frac{1}{3}}\left\Vert u\right\Vert _{X^{+}}^{\frac{8}{3}},
\end{alignat*}
\begin{align*}
\left\Vert I_{3}\right\Vert _{L_{v}^{\infty}H_{y}^{\alpha}} & \lesssim t^{-\frac{5}{4}}\left\Vert \partial_{v}\left|\gamma\right|^{2}\gamma\right\Vert _{L_{v}^{2}H_{y}^{\alpha}}\lesssim t^{-\frac{13}{12}}\left\Vert D_{v}^{\frac{2}{3}}\left|\gamma\right|^{2}\gamma\right\Vert _{L_{v}^{2}H_{y}^{\alpha}}\\
 & \lesssim t^{-\frac{13}{12}}\left\Vert \gamma\right\Vert _{L_{v}^{\infty}H_{y}^{\alpha}}^{2}\left\Vert D_{v}^{\frac{2}{3}}\gamma\right\Vert _{L_{v}^{2}H_{y}^{\alpha}}\lesssim t^{-\frac{13}{12}}\left\Vert \gamma\right\Vert _{L_{v}^{\infty}H_{y}^{\alpha}}^{2}\left\Vert L_{x}u\right\Vert _{L_{x,y}^{2}}^{\frac{2}{3}}\left\Vert D_{y}^{s}u\right\Vert _{L_{x,y}^{2}}^{\frac{1}{3}}\\
 & \lesssim t^{-\frac{13}{12}}\left\Vert u\right\Vert _{L_{x,y}^{2}}^{\frac{1}{3}}\left\Vert u\right\Vert _{X^{+}}^{\frac{8}{3}}.
\end{align*}

Since 
\[
\left\Vert I\right\Vert _{Y}\leq\left\Vert I_{1}\right\Vert _{Y}+\left\Vert I_{2}\right\Vert _{Y}+\left\Vert I_{3}\right\Vert _{Y},
\]
we obtain (\ref{eq:L2err}) and (\ref{eq:LinftyYalphaerr}).

\end{proof}

\subsection{The energy bound for $\gamma$.}

From the equation (\ref{eq:gamma}) there is the natural guess that
\[
\left\Vert \gamma(t)\right\Vert _{L_{v}^{\infty}H_{y}^{1}}\lesssim\epsilon
\]
for any $t\geq1.$ Indeed, multiplying (\ref{eq:gamma}) with $\overline{\gamma}_{t}$,
integrating over $y$, and taking the real part we have
\[
\frac{1}{2}\text{Re}\int\left(\triangle_{y}\gamma\right)\overline{\gamma_{t}}\,dy=\frac{1}{t}\text{Re}\int\left|\gamma\right|^{2}\gamma\overline{\gamma_{t}}\,dy+\text{Re}\int I\overline{\gamma_{t}}\,dy,
\]
which directly implies that 
\[
\partial_{t}\left\Vert \nabla_{y}\gamma\right\Vert _{L_{y}^{2}}^{2}+\frac{1}{t}\partial_{t}\left\Vert \gamma\right\Vert _{L_{y}^{4}}^{4}=-4\text{Re}\left\langle \gamma_{t},I\right\rangle _{H_{y}^{-1},H_{y}^{1}}.
\]
Use (\ref{eq:gamma}) again and the fact that in $\mathbb{R}^{d}$
for $d\leq4$ we have $L^{\frac{4}{3}}\subset H^{-1}$, $\left\Vert f\right\Vert _{H^{-1}}\lesssim\left\Vert f\right\Vert _{L^{\frac{4}{3}}}$
and $H^{s_{1}}\subset H^{s_{2}}$ if $s_{1}>s_{2}$,
\begin{align*}
\left\Vert \gamma_{t}\right\Vert _{H_{y}^{-1}} & \lesssim\left\Vert \triangle_{y}\gamma\right\Vert _{H_{y}^{-1}}+t^{-1}\left\Vert \left|\gamma\right|^{2}\gamma\right\Vert _{H_{y}^{-1}}+\left\Vert I\right\Vert _{H_{y}^{-1}}\\
 & \lesssim\left\Vert \gamma\right\Vert _{H_{y}^{1}}+t^{-1}\left\Vert \gamma\right\Vert _{L_{y}^{4}}^{3}+\left\Vert I\right\Vert _{H_{y}^{\alpha}}\lesssim\left\Vert \gamma\right\Vert _{H_{y}^{\alpha}}+t^{-1}\left\Vert \gamma\right\Vert _{H_{y}^{\alpha}}^{3}+\left\Vert I\right\Vert _{H_{y}^{\alpha}}.
\end{align*}
After integrating with respect to $t$, then taking the supremum over
$v$, we get 
\begin{equation}
\begin{aligned} & \left\Vert \nabla_{y}\gamma(T)\right\Vert _{L_{v}^{\infty}L_{y}^{2}}^{2}+\frac{1}{t}\left\Vert \gamma(T)\right\Vert _{L_{v}^{\infty}L_{y}^{4}}^{4}\\
\lesssim & \left\Vert \nabla_{y}\gamma(1)\right\Vert _{L_{v}^{\infty}L_{y}^{2}}^{2}+\left\Vert \gamma(1)\right\Vert _{L_{v}^{\infty}L_{y}^{4}}^{4}+\int_{1}^{T}t^{-2}\left\Vert u\right\Vert _{L_{x,y}^{2}}^{\frac{2}{3}}\left\Vert u(t)\right\Vert _{X^{+}}^{\frac{10}{3}}dt\\
 & +\int_{1}^{T}t^{-\frac{13}{12}}\left\Vert u\right\Vert _{L_{x,y}^{2}}^{\frac{1}{2}}\left\Vert u(t)\right\Vert _{X^{+}}^{\frac{7}{2}}+t^{-\frac{25}{12}}\left\Vert u\right\Vert _{L_{x,y}^{2}}^{\frac{5}{6}}\left\Vert u(t)\right\Vert _{X^{+}}^{\frac{31}{6}}+t^{-\frac{13}{6}}\left\Vert u\right\Vert _{L_{x,y}^{2}}^{\frac{2}{3}}\left\Vert u(t)\right\Vert _{X^{+}}^{\frac{16}{3}}dt.
\end{aligned}
\label{eq:gammaestimate}
\end{equation}

\section{The Energy Estimate\label{sec:The-Energy-Estimate}}

In this section, we aim to prove the energy bounds for $\left\Vert L_{x}u\right\Vert _{L_{x,y}^{2}}$
and $\left\Vert D_{y}^{s}u\right\Vert _{L_{x,y}^{2}}$. Here we will
work with the general linearized equation of (\ref{eq main}) which
is given by 
\begin{equation}
\left(i\partial_{t}+\frac{1}{2}\partial_{x}^{2}+\frac{1}{2}\triangle_{y}\right)\nu=2\left|u\right|^{2}\nu-u^{2}\overline{\nu}.\label{eq:linearized}
\end{equation}
Notice that the equation for $L_{x}u$, (\ref{eq:lin}) is the same
as (\ref{eq:linearized}). The function $D_{y}^{s}u$ does not directly
satisfy the linearized equation, but its equation can be written in
the form 
\begin{equation}
\left(i\partial_{t}+\frac{1}{2}\partial_{x}^{2}+\frac{1}{2}\triangle_{y}\right)D_{y}^{s}u=2|u|^{2}D_{y}^{s}u+u^{2}\overline{D_{y}^{s}u}+\text{cor}(t),\label{eq: lin dyu}
\end{equation}
where 
\[
\text{cor}(t):=D_{y}^{s}\left(\left|u\right|^{2}u\right)-\left(2|u|^{2}D_{y}^{s}u+u^{2}\overline{D_{y}^{s}u}\right).
\]
The correction term $\text{cor}(t)$ is nontrivial for $s\neq1$ but
has a commutator structure and satisfies favorable bounds. We will
leave the proof of bound for $\text{cor}(t)$ for the last part of
the section. 

To obtain $L_{x,y}^{2}$ estimates for the linearized equation (\ref{eq:linearized}),
we make the following assumptions on $u$:

\begin{hyp}\label{hyp} The solution $u$ for (\ref{eq main}) exists
in a time interval $[0,T]$, and satisfies the bounds
\begin{equation}
\left\Vert u\left(t\right)\right\Vert _{L_{v}^{\infty}H_{y}^{1}}\leq D\epsilon\left|t\right|^{-\frac{1}{2}},\label{eq:bootsass1}
\end{equation}
\begin{equation}
\left\Vert u\left(t\right)\right\Vert _{X^{+}}\leq D\epsilon\left(1+\left|t\right|\right)^{\delta},\label{eq:bootsass2}
\end{equation}
for $t\in[0,T]$. Here $D$ is a sufficiently large positive number
which does not depend on $u$.

\end{hyp}

Then we have the following bound for the linearized equation (\ref{eq:linearized}): 
\begin{prop}
\label{prop. nugrowth}Suppose $u$ is a solution to (\ref{eq main})
satisfying Hypothesis \ref{hyp} on $[0,T]$, then we will have that

(a) The equation (\ref{eq:linearized}) in $\nu$ is $L_{x,y}^{2}$
well-posed.

(b) There is the bound
\begin{equation}
\left\Vert \nu(t)\right\Vert _{L_{x,y}^{2}}\lesssim\left\Vert \nu(0)\right\Vert _{L_{x,y}^{2}}\left(1+t\right)^{2D^{3}\epsilon^{2}}\label{eq:nugrowth}
\end{equation}
for $t\in[0,T]$.
\end{prop}
The local well-posedness property of equation (\ref{eq:linearized})
is given by Proposition \ref{prop:localwell}; therefore it suffices
only to prove (\ref{eq:nugrowth}). Denote 
\begin{equation}
V\left(t,v,y\right):=e^{-\frac{itx^{2}}{2}}\sqrt{t}\nu\left(t,tx,y\right)
\end{equation}
 by making a substitution $v=tx$. Instead of computing the $L_{x,y}^{2}$
norm of $\nu$, it is easier to work with the $L_{v,y}^{2}$ norm
of $V(t,v,y)$( which is the same). By direct computation, $V$ satisfies
the equation
\begin{equation}
\left(i\partial_{t}+\frac{1}{2t^{2}}\partial_{v}^{2}+\frac{1}{2}\triangle_{y}\right)V\left(t,v,y\right)=t^{-1}\left[2\left|w\right|^{2}V(t,v,y)+w^{2}\overline{V}\left(t,v,y\right)\right].\label{eq:V}
\end{equation}
 Denote the associated linear evolution operator by 
\[
S(t)=e^{-it\triangle_{y}/2}e^{i\partial_{v}^{2}/\left(2t\right)},
\]
and transform the equation (\ref{eq:V}) into the form
\[
i\partial_{t}\left(S(-t)V(t,v,y)\right)=t^{-1}S(-t)\left[2\left|w\right|^{2}V(t,v,y)+w^{2}\overline{V}\left(t,v,y\right)\right].
\]

From the above equation we have the fact that

\begin{align*}
\partial_{t}\frac{1}{2}\left\Vert S(-t)V(t,v,y)\right\Vert _{L_{v,y}^{2}}^{2} & =t^{-1}\text{Im}\int_{\mathbb{R}\times\mathbb{T}^{d}}\left[\overline{S(-t)V}\right]\left[S(-t)w^{2}\overline{V}\right]dvdy,
\end{align*}
and 
\begin{equation}
\left\Vert \nu(t)\right\Vert _{L_{x,y}^{2}}=\left\Vert V(t)\right\Vert _{L_{v,y}^{2}}=\left\Vert S(-t)V(t)\right\Vert _{L_{v,y}^{2}}\label{eq:LVequal}
\end{equation}
for any $t\neq0.$ 

Here we define
\begin{equation}
\mathcal{Z}:=S(-t)V,
\end{equation}
and
\begin{equation}
W^{*}:=e^{-it\triangle_{y}/2}w.
\end{equation}
We define
\begin{equation}
G:=e^{-it\triangle_{y}/2}\gamma,
\end{equation}
which is the linear pullback of $\gamma$. 

By a direct computation we write the Fourier transform of the second
factor in the integrand in the form: 

\begin{equation}
\begin{aligned} & \mathcal{F}\left[S(-t)w^{2}\overline{V}\right]\left(t,\xi,\mathbf{k}\right)\\
 & =\sum_{\mathcal{M}\left(\mathbf{k}\right)}\int\int\exp\left(i\frac{t}{2}\left|\mathbf{k}\right|^{2}-i\frac{\xi^{2}}{2t}\right)\widehat{w}\left(t,\kappa,\mathbf{k}_{1}\right)\widehat{\overline{V}}\left(t,\xi-\eta-\kappa,\mathbf{k}_{2}\right)\widehat{w}\left(t,\eta,\mathbf{k}_{3}\right)d\kappa d\eta\\
 & =\sum_{\mathcal{M}\left(\mathbf{k}\right)}\int\int e^{i\Psi(t)}\widehat{W^{*}}\left(t,\kappa,\mathbf{k}_{1}\right)\widehat{\overline{\mathcal{Z}}}\left(t,\xi-\eta-\kappa,\mathbf{k}_{2}\right)\widehat{W^{*}}\left(t,\eta,\mathbf{k}_{3}\right)d\kappa d\eta.
\end{aligned}
\label{eq:nonlinV}
\end{equation}
Here the phase function $\Psi$ is defined as
\begin{equation}
\Psi\left(t\right):=\frac{1}{2t}\left(\left(\xi-\eta-\kappa\right)^{2}+\xi^{2}\right)+\frac{t}{2}\omega,\label{eq:refunc}
\end{equation}
where $\omega:=\left|\mathbf{k}_{1}\right|^{2}-\left|\mathbf{k}_{2}\right|^{2}+\left|\mathbf{k}_{3}\right|^{2}-\left|\mathbf{k}\right|^{2}$. 

It is natural to separate the right hand side of the equation (\ref{eq:nonlinV})
into four different parts according to the $v$ and $y$ frequencies
of the factors:

\begin{itemize}

\item Where the $v$-frequency of one of $w$ is large, $\left\{ \left(\kappa,\eta\right):\left|\kappa\right|\geq\sqrt{t}\right\} \bigcup\left\{ \left(\kappa,\eta\right):\left|\eta\right|\geq\sqrt{t}\right\} $.
After excluding this case, we are able to replace $w$ in (\ref{eq:V})
by $\gamma$. The corresponding term has the expression
\begin{equation}
e_{1}\left(t,v,y\right):=t^{-1}S(-t)\left[\left(w^{2}-\gamma^{2}\right)\overline{V}\right].
\end{equation}

\item Where the $y$ frequencies are resonant,

$\left\{ \left(\kappa,\eta\right):\left|\kappa\right|,\left|\eta\right|<\sqrt{t}\right\} \cap\left\{ \left(\mathbf{k}_{1},\mathbf{k}_{2},\mathbf{k}_{3}\right):\left(\mathbf{k}_{1},\mathbf{k}_{2},\mathbf{k}_{3}\right)\in\mathcal{M}\left(\mathbf{k}\right),\left(\mathbf{k}_{1},\mathbf{k}_{2},\mathbf{k}_{3}\right)\in\Gamma_{0}\left(\mathbf{k}\right)\right\} .$
This corresponds to
\begin{equation}
e_{2}\left(t,v,y\right):=t^{-1}S(-t)\sum_{\mathbf{k}}\left[\sum_{\Gamma_{0}\left(\mathbf{k}\right)}\gamma\left(t,v,\mathbf{k}_{1}\right)\overline{V}\left(t,v,\mathbf{k}_{2}\right)\gamma\left(t,v,\mathbf{k}_{3}\right)e^{i\mathbf{k}\cdot y}\right].
\end{equation}

\item Where the $v$ and $y$ frequencies are non-resonant,

$\left\{ \left(\kappa,\eta\right):\left|\kappa\right|,\left|\eta\right|<\sqrt{t}\right\} \cap\left\{ \left(\mathbf{k}_{1},\mathbf{k}_{2},\mathbf{k}_{3}\right):\left(\mathbf{k}_{1},\mathbf{k}_{2},\mathbf{k}_{3}\right)\in\mathcal{M}\left(\mathbf{k}\right)\cap\Gamma_{\omega}\left(\mathbf{k}\right),\omega\neq0\right\} \cap\left\{ \left|\Psi^{\prime}(t)\right|\gtrsim t^{-\frac{3}{8}}\right\} $.
Therefore we choose the region $\Omega_{t}^{1}$ as follows: 
\begin{equation}
\begin{aligned}\Omega_{t}^{1}\left(\xi,\omega\right)= & \left\{ \begin{array}{cc}
\begin{array}{c}
\omega\neq0,\end{array} & \left|\frac{\left(\xi-\kappa-\eta\right)^{2}}{t^{2}\omega}-\frac{1}{2}\right|\geq\frac{1}{2}t^{-\frac{3}{8}}\end{array}\right\} \bigcup\left\{ \omega\neq0,\left|\frac{\xi^{2}}{t^{2}\omega}-\frac{1}{2}\right|\geq2t^{-\frac{3}{8}}\right\} .\end{aligned}
\end{equation}
The corresponding term in the energy is given by
\begin{equation}
\begin{aligned} & e_{3}\left(t,v,y\right)\\
:= & t^{-1}\mathcal{F}_{\xi}^{-1}\sum_{\mathbf{k}}\left[\sum_{\omega\neq0}\sum_{\Gamma_{\omega}\left(\mathbf{k}\right)}\iint\mathcal{X}_{1}e^{i\Psi(t)}\widehat{G}\left(t,\kappa,\mathbf{k}_{1}\right)\widehat{\overline{\mathcal{Z}}}\left(t,\xi-\eta-\kappa,\mathbf{k}_{2}\right)\widehat{G}\left(t,\eta,\mathbf{k}_{3}\right)d\kappa d\eta\right]e^{i\mathbf{k}\cdot y}.
\end{aligned}
\end{equation}
Here $\mathcal{X}_{1}$ is a cutoff function selecting this region.
Precisely, we will define the frequency cutoff function $\mathcal{X}_{1}$
depending on $t,\xi,\kappa,\eta$ and $\omega$ by 
\begin{equation}
\mathcal{X}_{1}:=1-\mathcal{X}_{2},
\end{equation}
where the function $\mathcal{X}_{2}$ is given in (\ref{defX_2}).

\item Where the $v$ and $y$ frequencies are almost resonant,

$\left\{ \left(\kappa,\eta\right):\left|\kappa\right|,\left|\eta\right|<\sqrt{t}\right\} \cap\left\{ \left(\mathbf{k}_{1},\mathbf{k}_{2},\mathbf{k}_{3}\right):\left(\mathbf{k}_{1},\mathbf{k}_{2},\mathbf{k}_{3}\right)\in\mathcal{M}\left(\mathbf{k}\right)\cap\Gamma_{\omega}\left(\mathbf{k}\right),\omega\neq0\right\} \cap\left\{ \left|\Psi^{\prime}(t)\right|\lesssim t^{-\frac{3}{8}}\right\} $.
Therefore we choose the region $\Omega_{t}^{2}$ as follows: 
\begin{equation}
\begin{aligned}\Omega_{t}^{2}\left(\xi,\omega\right)= & \left\{ \begin{array}{cc}
\begin{array}{c}
\omega\neq0,\end{array} & \left|\frac{\left(\xi-\kappa-\eta\right)^{2}}{t^{2}\omega}-\frac{1}{2}\right|<2t^{-\frac{3}{8}}\end{array}\right\} \bigcap\left\{ \begin{array}{cc}
\begin{array}{c}
\omega\neq0,\end{array} & \left|\frac{\xi^{2}}{t^{2}\omega}-\frac{1}{2}\right|<2t^{-\frac{3}{8}}\end{array}\right\} .\end{aligned}
\end{equation}
The corresponding term in the energy is given by
\begin{equation}
\begin{aligned} & e_{4}\left(t,v,y\right)\\
:= & t^{-1}\mathcal{F}_{\xi}^{-1}\sum_{\mathbf{k}}\left[\sum_{\omega\neq0}\sum_{\Gamma_{\omega}\left(\mathbf{k}\right)}\iint\mathcal{X}_{2}e^{i\Psi(t)}\widehat{G}\left(t,\kappa,\mathbf{k}_{1}\right)\widehat{\overline{\mathcal{Z}}}\left(t,\xi-\eta-\kappa,\mathbf{k}_{2}\right)\widehat{G}\left(t,\eta,\mathbf{k}_{3}\right)d\kappa d\eta\right]e^{i\mathbf{k}\cdot y},
\end{aligned}
\end{equation}
where $\mathcal{X}_{2}$ is a cutoff function selecting this region.
Here we define the frequency cutoff function $\mathcal{X}_{2}$ by
\begin{equation}
\mathcal{X}_{2}:=\mathcal{X}_{1,\omega}\left(\xi-\kappa-\eta\right)\mathcal{X}_{2,\omega}\left(\xi\right),\label{defX_2}
\end{equation}
where
\begin{equation}
\mathcal{X}_{1,\omega}\left(\xi-\kappa-\eta\right):=\mathcal{X}\left(\frac{1}{2}t^{\frac{3}{8}}\left(\frac{\left(\xi-\kappa-\eta\right)^{2}}{t^{2}\omega}-\frac{1}{2}\right)\right),\quad\mathcal{X}_{2,\omega}\left(\xi\right):=\mathcal{X}\left(\frac{1}{2}t^{\frac{3}{8}}\left(\frac{\xi^{2}}{t^{2}\omega}-\frac{1}{2}\right)\right).
\end{equation}

\end{itemize}

Hence we have 
\[
t^{-1}S\left(-t\right)\left(w^{2}\overline{V}\right):=e_{1}+e_{2}+e_{3}+e_{4}.
\]

Assuming that the initial data satisfies (\ref{eq:ini}), by the local
well-posedness we know that $\nu(t)$ and $u(t)$ exist inside the
interval $[0,T]$. To advance from time 0 to time 1 we use the local
well-posedness results to obtain

\[
\left\Vert \nu\left(1\right)\right\Vert _{L_{x,y}^{2}}\lesssim\left\Vert \nu\left(0\right)\right\Vert _{L_{x,y}^{2}},
\]
and note that by the mass conservation law (\ref{eq:mass}) there
is the inequality 
\[
\left\Vert u(t)\right\Vert _{L_{x,y}^{2}}=\left\Vert u(0)\right\Vert _{L_{x,y}^{2}}\leq\epsilon
\]
for any $t\in[0,T]$. By (\ref{eq:LVequal}) there is the inequality
\begin{equation}
\frac{1}{2}\left\Vert \nu\left(T\right)\right\Vert _{L_{x,y}^{2}}^{2}\leq\frac{1}{2}\left\Vert \nu\left(1\right)\right\Vert _{L_{x,y}^{2}}^{2}+\left|\int_{1}^{T}\left\langle S(-t)V,e_{1}+e_{2}+e_{3}+e_{4}\right\rangle _{L_{x,y}^{2}}dt\right|.\label{eq:nuintegral}
\end{equation}

By (\ref{eq:bootsass1}), and the definition of $\gamma$, there is
the property 
\[
\left\Vert \gamma(t)\right\Vert _{L_{v}^{\infty}H_{y}^{1}}\lesssim\sqrt{t}\left\Vert u(t)\right\Vert _{L_{x}^{\infty}H_{y}^{1}}\lesssim D\epsilon.
\]

\subsection{The high frequency estimates. }

First we start with bounds for the high $v$-frequencies in $w$. 

\begin{lemma} Assume that $T\geq1.$ Then the following estimates
hold uniformly in $t$: 
\begin{equation}
\int_{1}^{T}\left|\left\langle S(-t)V,e_{1}\left(t\right)\right\rangle _{L_{v,y}^{2}}\right|dt\lesssim\int_{1}^{T}D^{\frac{11}{6}}\epsilon^{2}t^{-\frac{13}{12}}\left(1+t\right)^{\frac{11}{6}\delta}\left\Vert \nu\left(t\right)\right\Vert _{L_{x,y}^{2}}^{2}dt\label{eq:e1bound}
\end{equation}
\end{lemma} 

\begin{proof}

Since $S(-t)$ is a unitary operator, using (\ref{eq:alpha bound})
we have 
\begin{align*}
\left\Vert e_{1}\left(t\right)\right\Vert _{L_{v,y}^{2}} & \lesssim t^{-1}\left(\left\Vert w\right\Vert _{L_{v}^{\infty}H_{y}^{\alpha}}+\left\Vert \gamma\right\Vert _{L_{v}^{\infty}H_{y}^{\alpha}}\right)\left\Vert w-\gamma\right\Vert _{L_{v}^{\infty}H_{y}^{\alpha}}\left\Vert V\right\Vert _{L_{v,y}^{2}}\\
 & \lesssim t^{-\frac{13}{12}}\left(\left\Vert w\right\Vert _{L_{v}^{\infty}H_{y}^{\alpha}}+\left\Vert \gamma\right\Vert _{L_{v}^{\infty}H_{y}^{\alpha}}\right)\left\Vert L_{x}u\right\Vert _{L_{x,y}^{2}}^{\frac{2}{3}}\left\Vert D_{y}^{s}u\right\Vert _{L_{x,y}^{2}}^{\frac{1}{3}}\left\Vert V\right\Vert _{L_{v,y}^{2}}\\
 & \lesssim t^{-\frac{13}{12}}\left\Vert u\right\Vert _{L_{x,y}^{2}}^{\frac{1}{6}}\left\Vert u\right\Vert _{X^{+}}^{\frac{11}{6}}\left\Vert \nu\right\Vert _{L_{x,y}^{2}}.
\end{align*}
\end{proof}

\subsection{The $y$ frequencies resonant term.}

The growth of the energy mainly comes from the resonant term and will
be smaller than $t^{-1}$; hence we can apply Grownwall's inequality.

\begin{lemma} Assume that $T\geq1.$ we will have 
\begin{equation}
\int_{1}^{T}\left|\left\langle S(-t)V,e_{2}\left(t\right)\right\rangle _{L_{v,y}^{2}}\right|dt\lesssim\int_{1}^{T}D^{2}\epsilon^{2}t^{-1}\left\Vert \nu\left(t\right)\right\Vert _{L_{x,y}^{2}}^{2}dt.\label{eq:e2bound}
\end{equation}
\end{lemma} 

\begin{proof}

Here we use the inequality (\ref{eq:h1esti}), which provides a good
fixed-time estimate for the $y$-resonant interactions. The original
proof of (\ref{eq:h1esti}) is given in Lemma 7.1 in \cite{HPTV}.
By the fact that $S(-t)$ is unitary, and the inequality (\ref{eq:h1esti}),
the factor $e_{2}$ satisfies the following inequality: 
\begin{equation}
\begin{aligned}\left\Vert e_{2}\left(t\right)\right\Vert _{L_{v,y}^{2}} & \lesssim\frac{1}{t}\left\Vert \left\Vert \gamma\right\Vert _{H_{y}^{1}}^{2}\left\Vert V\right\Vert _{L_{y}^{2}}\right\Vert _{L_{v}^{2}}\lesssim\frac{1}{t}\left\Vert \gamma\right\Vert _{L_{v}^{\infty}H_{y}^{1}}^{2}\left\Vert V\right\Vert _{L_{v,y}^{2}}\lesssim\frac{1}{t}\left\Vert \gamma\right\Vert _{L_{v}^{\infty}H_{y}^{1}}^{2}\left\Vert \nu\right\Vert _{L_{x,y}^{2}}.\end{aligned}
\end{equation}

\end{proof}

\subsection{The fast-time oscillations.}

Here we use a normal form energy correction to cancel out the non-resonant
frequencies in $\frac{1}{t}\gamma^{2}\overline{L}$ , using a technique
developed in the papers \cite{IT2,shatah,GMS}. The idea is that we
may apply integration by parts in time to get a better decay where
the nonlinear term is non-resonant. 

From the equation (\ref{eq:nonlinV}) we may rewrite the remaining
terms with low frequency of $w$ and non-resonant $y$ frequencies
as 
\begin{equation}
\sum_{\omega\neq0}\sum_{\Gamma_{\omega}\left(\mathbf{k}\right)}\int\int e^{i\Psi(t)}\mathcal{X}_{1}\widehat{G}\left(t,\kappa,\mathbf{k}_{1}\right)\widehat{\overline{\mathcal{Z}}}\left(t,\xi-\eta-\kappa,\mathbf{k}_{2}\right)\widehat{G}\left(t,\eta,\mathbf{k}_{3}\right)d\kappa d\eta.
\end{equation}
 The resonance function $\Psi(t)$ is given in (\ref{eq:refunc}).
One can only apply the normal form correction in the region where
$\Psi^{\prime}\neq0$. By the following equation
\begin{equation}
\Psi^{\prime}\left(t\right)=-\frac{1}{2t^{2}}\left(\left(\xi-\eta-\kappa\right)^{2}+\xi^{2}\right)+\frac{1}{2}\omega,
\end{equation}
it is obvious that inside the area $\Omega_{t}^{1}$ we have $\left|\Psi^{\prime}\left(t\right)\right|\gtrsim t^{-\frac{3}{8}}\omega.$ 

\begin{lemma}Assume that $T\geq1.$Then we have the following estimate
\begin{equation}
\begin{aligned}\left|\int_{1}^{T}\left\langle S(-t)V,e_{3}\left(t\right)\right\rangle _{L_{v,y}^{2}}dt\right| & \lesssim D^{\frac{5}{3}}\epsilon^{2}T^{-\frac{5}{8}}\left(1+T\right)^{\frac{5}{3}\delta}\left\Vert \nu\left(T\right)\right\Vert _{L_{v,y}^{2}}^{2}+D^{\frac{5}{3}}\epsilon^{2}\left\Vert \nu(1)\right\Vert _{L_{v,y}^{2}}^{2}\\
 & \quad+\int_{1}^{T}\left[D^{\frac{5}{3}}\epsilon^{2}t^{-\frac{5}{4}}\left(1+t\right)^{\frac{5}{3}\delta}+D^{\frac{5}{3}}\epsilon^{2}t^{-\frac{13}{8}}\left(1+t\right)^{\frac{5}{3}\delta}\right]\left\Vert \nu\left(t\right)\right\Vert _{L_{x,y}^{2}}^{2}dt\\
 & \quad+\int_{1}^{T}D^{\frac{10}{3}}\epsilon^{4}t^{-\frac{13}{8}}\left(1+t\right)^{\frac{10}{3}\delta}\left\Vert \nu\left(t\right)\right\Vert _{L_{x,y}^{2}}^{2}dt.
\end{aligned}
\label{eq:e3bound}
\end{equation}

\end{lemma}

\begin{proof}

First observe that
\[
e^{i\Psi(t)}=\frac{1}{i\Psi^{\prime}(t)}\left(\partial_{t}e^{i\Psi(t)}\right)\text{ and }\partial_{t}\left(\frac{1}{i\Psi^{\prime}(t)}\right)=-\frac{\Psi^{\prime\prime}\left(t\right)}{i\left(\Psi^{\prime}(t)\right)^{2}},
\]
where $\Psi^{\prime\prime}\left(t\right)=\frac{1}{t^{3}}\left(\left(\xi-\kappa-\eta\right)^{2}+\xi^{2}\right)$
. Thus it is natural to define the trilinear form as follows:
\begin{align*}
 & \mathcal{O}_{1}^{t}\left[f_{1},f_{2},f_{3},f_{4}\right]\\
:= & t^{-1}\sum_{\mathbf{k}}\sum_{\omega\neq0}\sum_{\Gamma_{\omega}\left(\mathbf{k}\right)}\int\iint\mathcal{X}_{1}\frac{e^{i\Psi(t)}}{i\Psi^{\prime}}\widehat{f_{1}}\left(t,\kappa,\mathbf{k}_{1}\right)\widehat{\overline{f_{2}}}\left(t,\xi-\eta-\kappa,\mathbf{k}_{2}\right)\widehat{f_{3}}\left(t,\eta,\mathbf{k}_{3}\right)\widehat{\overline{f_{4}}}\left(t,-\xi,\mathbf{k}\right)d\kappa d\eta d\xi,
\end{align*}
\begin{align*}
\mathcal{O}_{2}^{t}\left[f_{1},f_{2},f_{3},f_{4}\right] & :=t^{-1}\sum_{\mathbf{k}}\sum_{\omega\neq0}\sum_{\Gamma_{\omega}\left(\mathbf{k}\right)}\int\iint\left[\mathcal{X}_{1}\frac{\Psi^{\prime\prime}}{i\left(\Psi^{\prime}\right)^{2}}+\frac{\partial_{t}\mathcal{X}_{1}}{i\Psi^{\prime}}\right]e^{i\Psi(t)}\widehat{f_{1}}\\
 & \quad\quad\left(t,\kappa,\mathbf{k}_{1}\right)\widehat{\overline{f_{2}}}\left(t,\xi-\eta-\kappa,\mathbf{k}_{2}\right)\widehat{f_{3}}\left(t,\eta,\mathbf{k}_{3}\right)\widehat{\overline{f_{4}}}\left(t,-\xi,\mathbf{k}\right)d\kappa d\eta d\xi.
\end{align*}

Then observe that
\begin{align*}
e_{3}\left(t\right)= & \partial_{t}\left(\mathcal{O}_{1}^{t}\left[G,\mathcal{Z},G,\mathcal{Z}\right]\right)+\mathcal{O}_{2}^{t}\left[G,\mathcal{Z},G,\mathcal{Z}\right]+t^{-1}\mathcal{O}_{1}^{t}\left[G,\mathcal{Z},G,\mathcal{Z}\right]\\
 & -\mathcal{O}_{1}^{t}\left[\partial_{t}G,\mathcal{Z},G,\mathcal{Z}\right]-\mathcal{O}_{1}^{t}\left[G,\partial_{t}\mathcal{Z},G,\mathcal{Z}\right]-\mathcal{O}_{1}^{t}\left[G,\mathcal{Z},\partial_{t}G,\mathcal{Z}\right]-\mathcal{O}_{1}^{t}\left[G,\mathcal{Z},G,\partial_{t}\mathcal{Z}\right].
\end{align*}

We start with the estimates associated with $\mathcal{O}_{1}^{t}$:

\begin{lemma}\label{o1estiamte}

Assume that $\mathcal{O}_{1}^{t}$ defined as above. Then for any
$t\geq1$
\begin{equation}
\left|\mathcal{O}_{1}^{t}\left[G,\mathcal{Z},G,\mathcal{Z}\right]\right|\lesssim\epsilon^{2}D^{\frac{5}{3}}t^{-\frac{5}{8}}\left(1+t\right)^{\frac{5}{3}\delta}\left\Vert \nu\left(t\right)\right\Vert _{L_{x,y}^{2}}^{2},
\end{equation}
\begin{equation}
\left|\mathcal{O}_{1}^{t}\left[\partial_{t}G,\mathcal{Z},G,\mathcal{Z}\right]\right|,\left|\mathcal{O}_{1}^{t}\left[G,\mathcal{Z},\partial_{t}G,\mathcal{Z}\right]\right|\lesssim\epsilon^{4}D^{\frac{10}{3}}t^{-\frac{13}{8}}\left(1+t\right)^{\frac{10}{3}\delta}\left\Vert \nu\right\Vert _{L_{x,y}^{2}}^{2},
\end{equation}
and 
\begin{equation}
\left|\mathcal{O}_{1}^{t}\left[G,\partial_{t}\mathcal{Z},G,\mathcal{Z}\right]\right|,\left|\mathcal{O}_{1}^{t}\left[G,\mathcal{Z},G,\partial_{t}\mathcal{Z}\right]\right|\lesssim\epsilon^{4}D^{\frac{10}{3}}t^{-\frac{13}{8}}\left(1+t\right)^{\frac{10}{3}\delta}\left\Vert \nu\right\Vert _{L_{x,y}^{2}}^{2}.
\end{equation}

\end{lemma}

\begin{proof}

Here we introduce the elementary inequality for $a>\frac{1}{2}$
\begin{equation}
\left\Vert f\right\Vert _{L_{x}^{1}\left(\mathbb{R}\right)}\lesssim\left\Vert f\right\Vert _{L_{x}^{2}\left(\mathbb{R}\right)}^{1-\frac{1}{2a}}\left\Vert \left|x\right|^{a}f\right\Vert _{L_{x}^{2}\left(\mathbb{R}\right)}^{\frac{1}{2a}}.\label{eq:ele}
\end{equation}
By (\ref{eq:ele}) and let $a=\frac{2}{3}$, we have the following
inequality by interpolation

\begin{equation}
\left\Vert \widehat{G}\left(t,\xi,\mathbf{k}\right)\right\Vert _{L_{\xi}^{1}h_{\mathbf{k}}^{\alpha}}\lesssim\left\Vert \widehat{G}(t,\xi,\mathbf{k})\right\Vert _{L_{\xi}^{2}h_{\mathbf{k}}^{\alpha}}^{\frac{1}{4}}\left\Vert \left|\xi\right|^{\frac{2}{3}}\widehat{G}\left(t,\xi,\mathbf{k}\right)\right\Vert _{L_{\xi}^{2}h_{\mathbf{k}}^{\alpha}}^{\frac{3}{4}}\lesssim\left\Vert \widehat{G}\right\Vert _{L_{\xi}^{2}l_{\mathbf{k}}^{2}}^{\frac{1}{6}}\left\Vert \widehat{G}\right\Vert _{L_{\xi}^{2}h_{\mathbf{k}}^{s}}^{\frac{1}{3}}\left\Vert \left|\xi\right|\widehat{G}\right\Vert _{L_{\xi}^{2}l_{\mathbf{k}}^{2}}^{\frac{1}{2}}.\label{eq:l1est}
\end{equation}
By (\ref{eq:l1est}), Lemma \ref{elementarylp}, (\ref{eq:gammabound}),(\ref{eq:mass})
and Minkowski's integral inequality
\begin{align*}
\left|\mathcal{O}_{1}^{t}\left[G,\mathcal{Z},G,\mathcal{Z}\right]\right| & \lesssim t^{-\frac{5}{8}}\left\Vert \widehat{G}\left(t,\xi,\mathbf{k}\right)\right\Vert _{L_{\xi}^{1}h_{\mathbf{k}}^{\alpha}}^{2}\left\Vert \widehat{\overline{\mathcal{Z}}}\left(t,\xi,\mathbf{k}\right)\right\Vert _{L_{\xi}^{2}l_{\mathbf{k}}^{2}}^{2}\lesssim t^{-\frac{5}{8}}\left\Vert \widehat{G}\right\Vert _{L_{\xi}^{2}l_{\mathbf{k}}^{2}}^{\frac{1}{3}}\left\Vert \widehat{G}\right\Vert _{L_{\xi}^{2}h_{\mathbf{k}}^{s}}^{\frac{2}{3}}\left\Vert \left|\xi\right|\widehat{G}\right\Vert _{L_{\xi}^{2}l_{\mathbf{k}}^{2}}\left\Vert \mathcal{Z}\right\Vert _{L_{v,y}^{2}}^{2}\\
 & \lesssim t^{-\frac{5}{8}}\left\Vert G\right\Vert _{L_{v,y}^{2}}^{\frac{1}{3}}\left\Vert G\right\Vert _{L_{\xi}^{2}H_{y}^{s}}^{\frac{2}{3}}\left\Vert \partial_{v}G\right\Vert _{L_{v,y}^{2}}\left\Vert \mathcal{Z}\right\Vert _{L_{v,y}^{2}}^{2}\lesssim t^{-\frac{5}{8}}\left\Vert D_{y}^{s}u\right\Vert _{L_{x,y}^{2}}^{\frac{2}{3}}\left\Vert L_{x}u\right\Vert _{L_{x,y}^{2}}\left\Vert \nu\right\Vert _{L_{x,y}^{2}}^{2}.
\end{align*}

From (\ref{eq:l1est}) and Bernstein's inequality, for $t>1$ there
is the bound 
\begin{align*}
\left\Vert \partial_{t}\widehat{G}\left(t,\xi,\mathbf{k}\right)\right\Vert _{L_{\xi}^{1}h_{\mathbf{k}}^{\alpha}} & \lesssim t^{-1}\left\Vert \widehat{w}(t,\xi,\mathbf{k})\right\Vert _{L_{\xi}^{1}h_{\mathbf{k}}^{\alpha}}^{3}+t^{-1}\left\Vert \widehat{w}\left(t,\xi\mathbf{k}\right)\right\Vert _{L_{\xi}^{1}h_{\mathbf{k}}^{\alpha}}\\
 & \lesssim\epsilon^{3}D^{\frac{5}{2}}t^{-1}\left(1+t\right)^{\frac{5}{2}\delta}+\epsilon D^{\frac{6}{5}}t^{-1}\left(1+t\right)^{\frac{5}{6}\delta}\lesssim\epsilon^{3}D^{\frac{5}{2}}t^{-1}\left(1+t\right)^{\frac{5}{2}\delta}.
\end{align*}
In the last inequality we consider the case when $t\gg1$; therefore
the first term will increase faster than the second term. 

Hence using the same procedure, there is the estimate
\begin{align*}
 & \left|\mathcal{O}_{1}^{t}\left[\partial_{t}G,\mathcal{Z},G,\mathcal{Z}\right]\right|\\
\lesssim & t^{-\frac{5}{8}}\left\Vert \partial_{t}\widehat{G}\left(t,\xi,\mathbf{k}\right)\right\Vert _{L_{\xi,}^{1}h_{\mathbf{k}}^{\alpha}}\left\Vert \widehat{G}\left(t,\xi,\mathbf{k}\right)\right\Vert _{L_{\xi}^{1}h_{\mathbf{k}}^{\alpha}}\left\Vert \mathcal{Z}\right\Vert _{L_{v,y}^{2}}^{2}\lesssim t^{-\frac{13}{8}}\left\Vert w\right\Vert _{L_{v,y}^{2}}^{\frac{2}{3}}\left\Vert w\right\Vert _{L_{v}^{2}H_{y}^{s}}^{\frac{4}{3}}\left\Vert \partial_{v}w\right\Vert _{L_{v,y}^{2}}^{2}\left\Vert \mathcal{Z}\right\Vert _{L_{v,y}^{2}}^{2}\\
\lesssim & D^{\frac{10}{3}}\epsilon^{4}t^{-\frac{13}{8}}\left(1+t\right)^{\frac{10}{3}\delta}\left\Vert \nu\right\Vert _{L_{x,y}^{2}}^{2}.
\end{align*}
Also there is the bound
\[
\left\Vert \partial_{t}\mathcal{Z}\right\Vert _{L_{v,y}^{2}}\lesssim t^{-1}\left\Vert w^{2}V\right\Vert _{L_{v,y}^{2}}\lesssim t^{-1}\left\Vert w\right\Vert _{L_{v}^{\infty}H_{y}^{\alpha}}^{2}\left\Vert V\right\Vert _{L_{v,y}^{2}}\lesssim D^{\frac{5}{3}}\epsilon^{2}t^{-1}\left(1+t\right)^{\frac{5}{3}\delta}\left\Vert \nu\right\Vert _{L_{x,y}^{2}},
\]
and therefore
\begin{align*}
\left|\mathcal{O}_{1}^{t}\left[G,\partial_{t}\mathcal{Z},G,\mathcal{Z}\right]\right| & \lesssim t^{-\frac{5}{8}}\left\Vert \widehat{G}\right\Vert _{L_{\xi}^{1}h_{\mathbf{k}}^{\alpha}}^{2}\left\Vert \partial_{t}\mathcal{Z}\right\Vert _{L_{v,y}^{2}}\left\Vert \mathcal{Z}\right\Vert _{L_{v,y}^{2}}\lesssim D^{\frac{10}{3}}\epsilon^{4}t^{-\frac{13}{8}}\left(1+t\right)^{\frac{10}{3}\delta}\left\Vert \nu\right\Vert _{L_{x,y}^{2}}^{2}.
\end{align*}

\end{proof}

It remains to establish the bound of $\mathcal{O}_{2}^{t}$. Note
that inside $\Omega_{t}^{1}$ we have the bounds 
\[
\left|\frac{\Psi^{\prime\prime}}{\left(\Psi^{\prime}\right)^{2}}\right|\lesssim t^{-\frac{1}{4}}\left|\omega\right|^{-1}\text{ for }\omega\neq0.
\]
By a direct computation we have that
\begin{align*}
\partial_{t}\mathcal{X}_{1} & =-\mathcal{X}_{1,\omega}^{\prime}\mathcal{X}_{2,\omega}\left[\frac{3}{16}t^{-\frac{5}{8}}\left(\frac{\left(\xi-\eta-\kappa\right)^{2}}{t^{2}\omega}-\frac{1}{2}\right)-t^{-\frac{5}{8}}\left(\frac{\left(\xi-\eta-\kappa\right)^{2}}{t^{2}\omega}\right)\right]\\
 & \quad\quad-\mathcal{X}_{1,\omega}\mathcal{X}_{2,\omega}^{\prime}\left[\frac{3}{16}t^{-\frac{5}{8}}\left(\frac{\xi^{2}}{t^{2}\omega}-\frac{1}{2}\right)-t^{-\frac{5}{8}}\left(\frac{\xi^{2}}{t^{2}\omega}\right)\right].
\end{align*}
Hence we have 
\begin{equation}
\left|\partial_{t}\mathcal{X}_{1}\right|\lesssim t^{-\frac{5}{8}},\mbox{ and }\left|\frac{\partial_{t}\mathcal{X}_{1}}{\Psi^{\prime}}\right|\lesssim t^{-\frac{1}{4}}\left|\omega\right|^{-1}.\label{eq:Xi_esti}
\end{equation}

\begin{lemma}\label{LemmaO_2}Assuming $t\geq1$, we have

\begin{equation}
\left|\mathcal{O}_{2}^{t}\left[G,\mathcal{Z},G,\mathcal{Z}\right]\right|\lesssim D^{\frac{5}{3}}\epsilon^{2}t^{-\frac{5}{4}}\left(1+t\right)^{\frac{5}{3}\delta}\left\Vert \nu\left(t\right)\right\Vert _{L_{x,y}^{2}}^{2}.
\end{equation}

\end{lemma}

Applying the same estimate as Lemma (\ref{o1estiamte}) one can obtain
the bound. Hence we have

\begin{align*}
\left|\int_{1}^{T}\left\langle S(-t)V,e_{3}\left(t\right)\right\rangle _{L_{v,y}^{2}}dt\right|\lesssim & \left.\left|\mathcal{O}_{1}^{t}\left[G,\mathcal{Z},G,\mathcal{Z}\right]\right|\right|_{t=1}^{T}+\int_{1}^{T}\left|\mathcal{O}_{2}^{t}\left[G,\mathcal{Z},G,\mathcal{Z}\right]\right|dt+\left|t^{-1}\mathcal{O}_{1}^{t}\left[G,\mathcal{Z},G,\mathcal{Z}\right]\right|dt\\
+\int_{1}^{T} & \left|\mathcal{O}_{1}^{t}\left[G_{t},\mathcal{Z},G,\mathcal{Z}\right]\right|+\left|\mathcal{O}_{1}^{t}\left[G,\mathcal{Z}_{t},G,\mathcal{Z}\right]\right|+\left|\mathcal{O}_{1}^{t}\left[G,\mathcal{Z},G_{t},\mathcal{Z}\right]\right|+\left|\mathcal{O}_{1}^{t}\left[G,\mathcal{Z},G,\mathcal{Z}_{t}\right]\right|dt.
\end{align*}
Using the estimates from $\mathcal{O}_{1}^{t}$ and $\mathcal{O}_{2}^{t}$,
the proof of Lemma \ref{LemmaO_2} is complete.

\end{proof}

\subsection{Almost resonant interactions.}

The remaining case corresponds to frequency interactions localized
in the region $\Omega_{t}^{2}$, where the inequality $\left|\Psi^{\prime}\left(t\right)\right|\lesssim t^{-\frac{3}{8}}$
holds. This case corresponds to almost in resonance both in $y$ and
$v$ frequency. Due to the smallness of $\left|\Psi^{\prime}\left(t\right)\right|$,
we are not able to perform normal form correction and obtain better
decay properties. To prove (\ref{eq:nugrowth}), it is crucial to
have the $t^{-1}$ decay of $\left|\left\langle S(-t)V,e_{4}\right\rangle _{L_{x,y}^{2}}\right|$.
Since on $\Omega_{t}^{2}$, the phase function $\Psi(t)$ is almost
constant in unit time scale, it is natural to divide the time interval
$[1,T]$ into unit time intervals. Inside each interval, we can replace
the phase function by a constant with small errors. After that we
apply a frequency localization in $v$. Then the quantity $e_{4}$
can be described as the output of interaction of linear waves, with
very small errors. Hence we can bound $e_{4}$ by Strichartz estimates
on torus, which gives us the desired bound $t^{-1}\left\Vert G\right\Vert _{L_{v}^{\infty}H_{y}^{1}}^{2}\left\Vert \nu\right\Vert _{L_{v,y}^{2}}^{2}$
with integrable errors.

Here we let $T_{1}=1$, $T_{n+1}=T_{n}+2\pi$. For $t\in\left[T_{n},T_{n+1}\right)$,
we need to separate $\left\langle S(-t)V,e_{4}\right\rangle _{L_{x,y}^{2}}$
into two parts:

\begin{align*}
\left\langle S(-t)V,e_{4}\right\rangle _{L_{x,y}^{2}} & :=\left\langle S(-t)V,M(t)\right\rangle _{L_{x,y}^{2}}+\left\langle S(-t)V,\text{err}\left(t\right)\right\rangle _{L_{x,y}^{2}},
\end{align*}
where $M(t)$ is a product of linear flow. For $M(t)$ we have the
bound 
\begin{equation}
\left\Vert M\left(t\right)\right\Vert _{L_{t}^{2}\left(T_{n},T_{n+1};L_{v,y}^{2}\right)}\lesssim t^{-1}\left\Vert G(T_{n},v,y)\right\Vert _{L_{v}^{\infty}H_{y}^{1}}^{2}\left\Vert V\left(T_{n},v,y\right)\right\Vert _{L_{v,y}^{2}}.
\end{equation}
For the error term $\text{err}(t)$, the decay rate is faster than
$t^{-1}$, by using the bootstrap assumptions (\ref{eq:bootsass1}),
(\ref{eq:bootsass2}), we can obtain the bound:
\[
\int_{1}^{T}\left\Vert \text{err}\left(t\right)\right\Vert _{L_{v,y}^{2}}\left\Vert \nu\right\Vert _{L_{x,y}^{2}}dt\lesssim\epsilon.
\]

\begin{lemma}

For $T\geq1$, we will have
\begin{equation}
\begin{aligned}\int_{1}^{T}\left|\left\langle S(-t)V,e_{4}\left(t\right)\right\rangle _{L_{v,y}^{2}}\right|dt & \lesssim\int_{1}^{T}D^{2}\epsilon^{2}t^{-1}\left\Vert \nu\left(t\right)\right\Vert _{L_{x,y}^{2}}^{2}dt\\
+ & \int_{1}^{T}\left[D^{\frac{11}{6}}\epsilon^{2}t^{-\frac{65}{64}}\left(1+t\right)^{\frac{11}{6}\delta}+D^{\frac{5}{3}}\epsilon^{2}t^{-\frac{5}{4}}\left(1+t\right)^{\frac{5}{3}\delta}+D^{\frac{73}{42}}\epsilon^{2}t^{-\frac{25}{24}}\left(1+t\right)^{\frac{73}{42}\delta}\right]\left\Vert \nu(t)\right\Vert _{L_{x,y}^{2}}^{2}dt\\
+ & \int_{1}^{T}\left[D^{\frac{11}{3}}\epsilon^{4}t^{-2}\left(1+t\right)^{\frac{5}{3}\delta}+D^{\frac{10}{3}}\epsilon^{4}t^{-2}\left(1+t\right)^{\frac{10}{3}\delta}\right]\left\Vert \nu\left(t\right)\right\Vert _{L_{x,y}^{2}}^{2}dt.
\end{aligned}
\label{eq:e4bound}
\end{equation}

\end{lemma}

Recall that $P^{y}$ denotes the frequency projection of $y$. We
may restrict the case to where $\max\left\{ \left|k_{1}\right|,\left|k_{3}\right|\right\} \leq t^{\frac{1}{16}}$,
and since $s>\frac{3d}{2}\geq\frac{3}{2}$, we peel off the high $y$-frequency
in $\gamma$:

\begin{lemma}Define 
\begin{align*}
\text{err}_{1}\left(t\right) & :=e_{4}\left(t,v,y\right)\\
 & \quad-t^{-1}\mathcal{F}_{\xi}^{-1}\sum_{\mathbf{k}}\left[\sum_{\omega\neq0}\sum_{\Gamma_{\omega}\left(\mathbf{k}\right)}\iint\mathcal{X}_{2}e^{i\Psi(t)}\widehat{\tilde{G}}\left(t,\kappa,\mathbf{k}_{1}\right)\widehat{\overline{\mathcal{Z}}}\left(t,\xi-\eta-\kappa,\mathbf{k}_{2}\right)\widehat{\tilde{G}}\left(t,\eta,\mathbf{k}_{3}\right)d\kappa d\eta\right]e^{i\mathbf{k}\cdot y},
\end{align*}
where 
\[
\tilde{G}=P_{\leq t^{\frac{1}{16}}}^{y}G.
\]
Then we have 
\begin{equation}
\left\Vert \text{err}_{1}\left(t\right)\right\Vert _{L_{v,y}^{2}}\lesssim D^{\frac{11}{6}}\epsilon^{2}t^{-\frac{65}{64}}\left(1+t\right)^{\frac{11}{6}\delta}\left\Vert \nu\right\Vert _{L_{x,y}^{2}}.\label{eq:err_1}
\end{equation}

\end{lemma}

\begin{proof} Here we follow the computations in (\ref{eq:alpha bound}),
and use the $y$ frequency is larger than $t^{\frac{1}{16}}$ to obtain
extra decay.

\begin{align*}
\left\Vert \text{err}_{1}\left(t\right)\right\Vert _{L_{v,y}^{2}}= & \frac{1}{t}\left\Vert P_{\geq t^{\frac{1}{16}}}^{y}\gamma\right\Vert _{L_{v}^{\infty}H_{y}^{\alpha}}\left\Vert \gamma\right\Vert _{L_{v}^{\infty}H_{y}^{\alpha}}\left\Vert V\right\Vert _{L_{v,y}^{2}}\\
\lesssim & \frac{1}{t}\left\Vert P_{\geq t^{\frac{1}{16}}}^{y}\gamma\right\Vert _{L_{v,y}^{2}}^{\frac{1}{6}}\left\Vert \gamma\right\Vert _{L_{v,y}^{2}}^{\frac{1}{6}}\left\Vert \partial_{v}\gamma\right\Vert _{L_{v,y}^{2}}\left\Vert D_{y}^{3\alpha}\gamma\right\Vert _{L_{v,y}^{2}}^{\frac{2}{3}}\left\Vert V\right\Vert _{L_{v,y}^{2}}\\
\lesssim & t^{-\frac{65}{64}}\left\Vert \gamma\right\Vert _{L_{v,y}^{2}}^{\frac{1}{6}}\left\Vert D_{y}^{s}\gamma\right\Vert _{L_{v,y}^{2}}^{\frac{5}{6}}\left\Vert \partial_{v}w\right\Vert _{L_{v,y}^{2}}\left\Vert V\right\Vert _{L_{v,y}^{2}}.
\end{align*}

\end{proof}

Thus we can reduce the problem to the case $0\leq\frac{1}{2}\omega\leq t^{\frac{1}{8}}$
since display $\frac{1}{2}\omega=\frac{1}{2}\left(\left|\mathbf{k}_{1}\right|^{2}+\left|\mathbf{k}_{3}\right|^{2}-\left|\mathbf{k}_{2}\right|^{2}-\left|\mathbf{k}_{4}\right|^{2}\right)\leq t^{\frac{1}{8}}$.

Next we consider unit time intervals, and show that we can freeze
$\widehat{\tilde{G}}$ and $\widehat{\mathcal{Z}}$ at end points.
First we show that on unit time intervals there are uniform bounds
for the linearized equation.

\begin{lemma}\label{lemmaZint}For any $s,t\in\left[T_{n},T_{n+1}\right)$,
and $s\leq t$, there are the bounds 
\begin{equation}
\left\Vert \mathcal{Z}\left(t\right)-\mathcal{Z}\left(s\right)\right\Vert _{L_{v,y}^{2}}\lesssim D^{\frac{5}{3}}\epsilon^{2}t^{-1}\left(1+t\right)^{\frac{5}{3}\delta}\left\Vert \mathcal{Z}\left(t\right)\right\Vert _{L_{v,y}^{2}},\label{eq:tdiffbound}
\end{equation}
\begin{equation}
\int_{s}^{t}\left\Vert \mathcal{Z}\left(\sigma\right)\right\Vert _{L_{v,y}^{2}}d\sigma\lesssim\left\Vert \mathcal{Z}\left(t\right)\right\Vert _{L_{v,y}^{2}}.\label{eq:tintebound}
\end{equation}

\end{lemma}

\begin{proof}Since 
\[
\left\Vert \mathcal{Z}\left(t\right)-\mathcal{Z}\left(s\right)\right\Vert _{L_{v,y}^{2}}\lesssim\int_{s}^{t}\left\Vert \partial_{t}\mathcal{Z}(\sigma)\right\Vert d\sigma\lesssim\int_{s}^{t}\sigma^{-1}\left\Vert w\left(\sigma\right)\right\Vert _{L_{v}^{\infty}H_{y}^{\alpha}}^{2}\left\Vert \mathcal{Z}\left(\sigma\right)\right\Vert _{L_{v,y}^{2}}d\sigma,
\]
applying (\ref{eq:bootsass1}), (\ref{eq:bootsass2}), there is the
bound (\ref{eq:tdiffbound}). By applying (\ref{eq:tdiffbound})

\begin{align*}
\int_{s}^{t}\left\Vert \mathcal{Z}\left(\sigma\right)\right\Vert _{L_{v,y}^{2}}d\sigma & \lesssim\left\Vert \mathcal{Z}\left(t\right)\right\Vert _{L_{x,y}^{2}}+\int_{s}^{t}\left\Vert \mathcal{Z}\left(\sigma\right)-\mathcal{Z}(t)\right\Vert _{L_{v,y}^{2}}d\sigma\\
 & \lesssim\left\Vert \mathcal{Z}\left(t\right)\right\Vert _{L_{x,y}^{2}}+D^{\frac{5}{3}}\epsilon^{2}t^{-1}\left(1+t\right)^{\frac{5}{3}\delta}\int_{s}^{t}\left\Vert \mathcal{Z}\left(\sigma\right)\right\Vert _{L_{v,y}^{2}}d\sigma.
\end{align*}
By recursion, when $T_{n}$ is large hence $t$ is large, the inequality
(\ref{eq:tintebound}) holds.

\end{proof}

\begin{lemma}For $t\in\left[T_{n},T_{n+1}\right)$, define 
\begin{align*}
 & \tilde{M}\left(t,v,y\right)\\
:= & t^{-1}\mathcal{F}_{\xi}^{-1}\sum_{\mathbf{k}}\left[\sum_{\omega\neq0}\sum_{\Gamma_{\omega}\left(\mathbf{k}\right)}\iint\mathcal{X}_{2}\left(T_{n}\right)e^{i\Psi(T_{n})}\widehat{\tilde{G}}\left(T_{n},\kappa,\mathbf{k}_{1}\right)\widehat{\overline{\mathcal{Z}}}\left(T_{n},\xi-\eta-\kappa,\mathbf{k}_{2}\right)\widehat{\tilde{G}}\left(T_{n},\eta,\mathbf{k}_{3}\right)d\kappa d\eta\right]e^{i\mathbf{k}\cdot y},
\end{align*}
and
\[
\text{err}_{2}(t)=e_{4}\left(t,v,y\right)-\text{err}_{1}(t)-\tilde{M}(t,v,y).
\]
Then there is the bound
\begin{equation}
\begin{aligned}\left\Vert \text{err}_{2}(t)\right\Vert _{L_{v,y}^{2}} & \lesssim\left[D^{\frac{5}{3}}\epsilon^{2}t^{-\frac{5}{4}}\left(1+t\right)^{\frac{5}{3}\delta}+D^{\frac{10}{3}}\epsilon^{4}t^{-2}\left(1+t\right)^{\frac{10}{3}\delta}\right]\left\Vert \nu\right\Vert _{L_{x,y}^{2}}.\end{aligned}
\label{eq:err_2}
\end{equation}

\end{lemma}

\begin{proof}

Define
\begin{align*}
\mathcal{\tilde{R}}^{t}\left[f_{1},f_{2},f_{3}\right]:= & \mathcal{F}_{\xi}^{-1}\sum_{\mathbf{k}}\left[\sum_{\omega\neq0}\sum_{\Gamma_{\omega}\left(\mathbf{k}\right)}\iint\mathcal{X}_{2}\left(t\right)e^{i\Psi(t)}\widehat{f_{1}}\left(\kappa,\mathbf{k}_{1}\right)\widehat{\overline{f_{2}}}\left(\xi-\eta-\kappa,\mathbf{k}_{2}\right)\widehat{f_{3}}\left(\eta,\mathbf{k}_{3}\right)d\kappa d\eta\right]e^{i\mathbf{k}\cdot y}.
\end{align*}

Then we have
\[
\text{err}_{2}(t)=t^{-1}\int_{T_{n}}^{t}\left(\left(\partial_{\sigma}\mathcal{\tilde{R}}^{\sigma}\right)\left[\tilde{G},\mathcal{Z},\tilde{G}\right]+\mathcal{\tilde{R}}^{\sigma}\left[\tilde{G}_{t},\mathcal{Z},\tilde{G}\right]+\mathcal{\tilde{R}}^{\sigma}\left[\tilde{G},\mathcal{Z}_{t},\tilde{G}\right]+\mathcal{\tilde{R}}^{\sigma}\left[\tilde{G},\mathcal{Z},\tilde{G}_{t}\right]\right)d\sigma.
\]
By a direct computation,
\[
\begin{aligned} & \left(\partial_{t}\mathcal{\tilde{R}}^{t}\right)\left[G,\mathcal{Z},G\right]\\
:= & \mathcal{F}_{\xi}^{-1}\sum_{\mathbf{k}}\left[\sum_{\omega\neq0}\sum_{\Gamma_{\omega}\left(\mathbf{k}\right)}\iint e^{i\Psi(t)}\left[i\Psi^{\prime}\mathcal{X}_{2}+\partial_{t}\mathcal{X}_{2}\right](t)\widehat{\tilde{G}}\left(t,\kappa,\mathbf{k}_{1}\right)\widehat{\overline{\mathcal{Z}}}\left(t,\xi-\eta-\kappa,\mathbf{k}_{2}\right)\widehat{\tilde{G}}\left(t,\eta,\mathbf{k}_{3}\right)d\kappa d\eta\right]e^{i\mathbf{k}\cdot y}.
\end{aligned}
\]
Using the bound on $\omega$, we have
\[
\left|\Psi^{\prime}(t)\mathcal{X}_{2}\right|\lesssim\omega t^{-\frac{3}{8}}\lesssim t^{-\frac{1}{4}},
\]
and by (\ref{eq:Xi_esti})
\[
\left|\partial_{t}\mathcal{X}_{2}\right|=\left|-\partial_{t}\mathcal{X}_{1}\right|\lesssim t^{-\frac{5}{8}}.
\]
Therefore by Young's inequality and Parseval's identity we have 
\begin{align*}
 & \left\Vert \left(\partial_{t}\mathcal{\tilde{R}}^{t}\right)\left[\tilde{G},\mathcal{Z},\tilde{G}\right]\right\Vert _{L_{v,y}^{2}}\\
\lesssim & \left\Vert \sum_{\omega\neq0}\sum_{\Gamma_{\omega}\left(\mathbf{k}\right)}\iint\left|i\Psi^{\prime}\mathcal{X}_{2}+\partial_{t}\mathcal{X}_{2}\right|\left|\widehat{\tilde{G}}\left(t,\kappa,\mathbf{k}_{1}\right)\right|\left|\widehat{\overline{\mathcal{Z}}}\left(t,\xi-\eta-\kappa,\mathbf{k}_{2}\right)\right|\left|\widehat{\tilde{G}}\left(t,\eta,\mathbf{k}_{3}\right)\right|d\kappa d\eta\right\Vert _{L_{\xi}^{2}l_{\mathbf{k}}^{2}}\\
\lesssim & \left\Vert i\Psi^{\prime}\mathcal{X}_{2}+\partial_{t}\mathcal{X}_{2}\right\Vert _{L_{\xi}^{\infty}l_{\mathbf{k}}^{\infty}}\left\Vert \widehat{\tilde{G}}(t)\right\Vert _{L_{\xi}^{1}l_{\mathbf{k}}^{1}}^{2}\left\Vert \widehat{\overline{\mathcal{Z}}}\left(t\right)\right\Vert _{L_{\xi}^{2}l_{\mathbf{k}}^{2}}.
\end{align*}
Hence by the same estimation in Lemma \ref{o1estiamte}, we have the
following bound
\begin{alignat*}{1}
\left\Vert \left(\partial_{t}\mathcal{\tilde{R}}^{t}\right)\left[\tilde{G},\mathcal{Z},\tilde{G}\right]\right\Vert _{L_{v,y}^{2}} & \lesssim t^{-\frac{1}{4}}\left\Vert \widehat{G}\right\Vert _{L_{\xi}^{1}h_{\mathbf{k}}^{\alpha}}^{2}\left\Vert \widehat{\overline{\mathcal{Z}}}\right\Vert _{L_{\xi}^{2}l_{\mathbf{k}}^{2}}\lesssim t^{-\frac{1}{4}}\left\Vert \widehat{G}\right\Vert _{L_{\xi}^{2}l_{\mathbf{k}}^{2}}^{\frac{1}{3}}\left\Vert \widehat{G}\right\Vert _{L_{\xi}^{2}h_{\mathbf{k}}^{s}}^{\frac{2}{3}}\left\Vert \left|\xi\right|\widehat{G}\right\Vert _{L_{\xi}^{2}l_{\mathbf{k}}^{2}}\left\Vert \mathcal{Z}\right\Vert _{L_{v,y}^{2}}\\
 & \lesssim t^{-\frac{1}{4}}D^{\frac{5}{3}}\epsilon^{2}\left(1+t\right)^{\frac{5}{3}\delta}\left\Vert \nu\right\Vert _{L_{x,y}^{2}}.
\end{alignat*}
Similarly, we have 
\begin{align*}
\left\Vert \mathcal{\tilde{R}}^{t}\left[\tilde{G_{t}},\mathcal{Z},\tilde{G}\right]\right\Vert _{L_{v,y}^{2}} & \lesssim\left\Vert \mathcal{X}_{2}\right\Vert _{L_{\xi}^{\infty}l_{\mathbf{k}}^{\infty}}\left\Vert \widehat{\tilde{G}_{t}}(t)\right\Vert _{L_{\xi}^{1}l_{\mathbf{k}}^{1}}\left\Vert \widehat{\tilde{G}}(t)\right\Vert _{L_{\xi}^{1}l_{\mathbf{k}}^{1}}\left\Vert \widehat{\overline{\mathcal{Z}}}\left(t\right)\right\Vert _{L_{\xi}^{2}l_{\mathbf{k}}^{2}}.
\end{align*}
 By the estimate in Lemma \ref{o1estiamte} again, we have
\begin{alignat*}{1}
\left\Vert \mathcal{\tilde{R}}^{t}\left[\tilde{G_{t}},\mathcal{Z},\tilde{G}\right]\right\Vert _{L_{v,y}^{2}} & \lesssim\left\Vert \widehat{\partial_{t}w}\right\Vert _{L_{\xi}^{1}h_{\mathbf{k}}^{\alpha}}\left\Vert \widehat{w}\right\Vert _{L_{\xi}^{1}h_{\mathbf{k}}^{\alpha}}\left\Vert \widehat{\overline{\mathcal{Z}}}\right\Vert _{L_{\xi}^{2}l_{\mathbf{k}}^{2}}\lesssim t^{-1}\left\Vert \widehat{w}\left(t,\xi,\mathbf{k}\right)\right\Vert _{L_{\xi}^{1}h_{\mathbf{k}}^{\alpha}}^{2}\left(1+\left\Vert \widehat{w}\left(t,\xi,\mathbf{k}\right)\right\Vert _{L_{\xi}^{1}h_{\mathbf{k}}^{\alpha}}^{2}\right)\left\Vert \mathcal{Z}\right\Vert _{L_{v,y}^{2}}\\
 & \lesssim t^{-1}D^{\frac{10}{3}}\epsilon^{4}\left(1+t\right)^{\frac{10}{3}\delta}\left\Vert \nu\right\Vert _{L_{x,y}^{2}},
\end{alignat*}
The same estimate yields
\[
\left\Vert \mathcal{\tilde{R}}^{t}\left[\tilde{G},\mathcal{Z}_{t},\tilde{G}\right]\right\Vert _{L_{v,y}^{2}}\lesssim t^{-1}D^{\frac{10}{3}}\epsilon^{4}\left(1+t\right)^{\frac{10}{3}\delta}\left\Vert \nu\right\Vert _{L_{x,y}^{2}}.
\]
Therefore we have 
\begin{equation}
\left\Vert \text{err}_{2}(t)\right\Vert _{L_{v,y}^{2}}\lesssim\int_{T_{n}}^{t}\left[D^{\frac{5}{3}}\epsilon^{2}\sigma^{-\frac{5}{4}}\left(1+\sigma\right)^{\frac{5}{3}\delta}+D^{\frac{10}{3}}\epsilon^{4}\sigma^{-2}\left(1+\sigma\right)^{\frac{10}{3}\delta}\right]\left\Vert \nu(\sigma)\right\Vert _{L_{x,y}^{2}}d\sigma.\label{eq:err2}
\end{equation}

In order to switch this quantity into a form where Gronwall's inequality
can be applied, we use Lemma \ref{lemmaZint}.

Using that $\left\Vert \mathcal{Z}\left(t\right)\right\Vert _{L_{v,y}^{2}}=\left\Vert \nu(\sigma)\right\Vert _{L_{x,y}^{2}}$
and Lemma \ref{lemmaZint}, we can switch the bound (\ref{eq:err2})
to
\begin{align*}
\left\Vert \text{err}_{2}(t)\right\Vert _{L_{v,y}^{2}} & \lesssim\left[D^{\frac{5}{3}}\epsilon^{2}t^{-\frac{5}{4}}\left(1+t\right)^{\frac{5}{3}\delta}+D^{\frac{10}{3}}\epsilon^{4}t^{-2}\left(1+t\right)^{\frac{10}{3}\delta}\right]\left\Vert \nu(t)\right\Vert _{L_{x,y}^{2}},
\end{align*}
therefore proving the lemma.

\end{proof}Since from (\ref{defX_2}), we can define the interval
$A(\omega)$ corresponding to the cutoff function $\mathcal{X}_{1,\omega}\left(T_{n}\right)$
by
\begin{equation}
\begin{aligned}A\left(\omega\right):= & \left\{ \xi:\left|\frac{\left(\xi-\kappa-\eta\right)^{2}}{T_{n}^{2}\omega}-\frac{1}{2}\right|<2T_{n}^{-\frac{3}{8}}\right\} ,\end{aligned}
\end{equation}
to be the $v-$frequencies contributing to the $\Omega_{t}^{2}$ region
for a fixed $\omega$. Then we have the following equality for fixed
$(\mathbf{k}_{1},\mathbf{k}_{2},\mathbf{k}_{3})$:
\begin{equation}
\begin{aligned} & \iint\mathcal{X}_{2}\left(T_{n}\right)e^{i\Psi(T_{n})}\widehat{\tilde{G}}\left(T_{n},\kappa,\mathbf{k}_{1}\right)\widehat{\overline{\mathcal{Z}}}\left(T_{n},\xi-\eta-\kappa,\mathbf{k}_{2}\right)\widehat{\tilde{G}}\left(T_{n},\eta,\mathbf{k}_{3}\right)d\kappa d\eta\\
= & \mathcal{X}_{2,\omega}\left(T_{n},\xi\right)\iint e^{i\Psi(T_{n})}\widehat{\tilde{G}}\left(T_{n},\kappa,\mathbf{k}_{1}\right)\widehat{\overline{\mathcal{Z}_{A\left(\omega\right)}}}\left(T_{n},\xi-\eta-\kappa,\mathbf{k}_{2}\right)\widehat{\tilde{G}}\left(T_{n},\eta,\mathbf{k}_{3}\right)d\kappa d\eta.
\end{aligned}
\label{eq:XtoA}
\end{equation}
Since the $\mathcal{X}_{2,\omega}\left(T_{n},\xi\right)$ factor does
not affect the computation of $L_{v,y}^{2}$ norm, we can omit it.
Inside this interval $A(\omega)$,
\begin{align*}
\left(\xi-\left(\eta+\kappa\right)\right)^{2} & =\frac{T_{n}^{2}}{2}\omega+\mathcal{O}\left(\omega T_{n}^{\frac{13}{8}}\right),
\end{align*}
hence
\[
\xi=\pm T_{n}\sqrt{\frac{\omega}{2}}\sqrt{1+\mathcal{O}\left(T_{n}^{-\frac{3}{8}}\right)}+\left(\eta+\kappa\right).
\]
By Taylor expansion $\sqrt{1+x}\approx1+\frac{1}{2}x$ for $x$ very
small and $\sqrt{\frac{\omega}{2}}\leq T_{n}^{\frac{1}{16}}$,
\[
\left|\xi\mp T_{n}\sqrt{\frac{\omega}{2}}\right|\lesssim\left|\eta+\kappa\right|+\mathcal{O}\left(T_{n}^{\frac{11}{16}}\right).
\]
For each $A(k)$ we compute the distance of sets by the distance from
their ``center'':
\[
\left|T_{n}\sqrt{\frac{k}{2}}-T_{n}\sqrt{\frac{k+1}{2}}\right|\approx\frac{T_{n}}{\sqrt{k}\sqrt{k+1}}.
\]
Since $\omega<T_{n}^{\frac{1}{16}}$, the centers are separated by
a distance at least $T_{n}^{\frac{15}{16}}$, and the width of each
set is at most $T_{n}^{\frac{11}{16}}$. Hence the sets are disjoint
when $T_{n}$ is sufficiently large. 

Since we are not able to directly localize to this interval, we instead
divide $A(\omega)$ into several disjoint intervals, where each interval
has length $\sqrt{T_{n}}$. Denote these intervals by $A(\omega,k),$
where $k\in\mathbb{Z}$,
\[
A\left(\omega,k\right):=\left\{ \xi:\left|\xi\mp T_{n}\sqrt{\frac{\omega}{2}}-2k\sqrt{T_{n}}\right|\leq\sqrt{T_{n}}\right\} .
\]
Each $A(\omega)$ has length $T_{n}^{\frac{11}{16}}$ and each $A(\omega,k)$
has length $\sqrt{T_{n}}$ , therefore $\left|k\right|\lesssim T_{n}^{\frac{11}{16}}/T_{n}^{\frac{1}{2}}\lesssim T_{n}^{\frac{3}{16}}.$
By (\ref{eq:XtoA}), the quantity $\tilde{M}$ can be written as 
\begin{align*}
\tilde{M}\left(t,v,y\right) & :=t^{-1}\mathcal{F}_{\xi}^{-1}\sum_{\mathbf{k}}\sum_{\omega\neq0}\sum_{\Gamma_{\omega}\left(\mathbf{k}\right)}\sum_{k}\\
 & \quad\quad\left[\iint e^{i\Psi(T_{n})}\widehat{\tilde{G}}\left(T_{n},\kappa,\mathbf{k}_{1}\right)\widehat{\overline{\mathcal{Z}_{A\left(\omega,k\right)}}}\left(T_{n},\xi-\eta-\kappa,\mathbf{k}_{2}\right)\widehat{\tilde{G}}\left(T_{n},\eta,\mathbf{k}_{3}\right)d\kappa d\eta\right]e^{i\mathbf{k}\cdot y}.
\end{align*}

After the frequency localization, we can transform the effect of phase
function into translation of the $v$ variable with small errors.
In the following computation, let $\xi_{0}$ be either $T_{n}\sqrt{\frac{\omega}{2}}+2k\sqrt{T_{n}}$
or $-T_{n}\sqrt{\frac{\omega}{2}}+2k\sqrt{T_{n}}$,
\begin{align*}
\frac{1}{2T_{n}}\left(\left(\xi-\eta-\kappa\right)^{2}+\xi^{2}\right) & =\frac{\xi^{2}}{T_{n}}-\frac{\xi_{0}\left(\eta+\kappa\right)}{T_{n}}-\frac{\left(\xi-\xi_{0}\right)\left(\eta+\kappa\right)}{T_{n}}+\frac{\left(\eta+\kappa\right)^{2}}{2T_{n}}.
\end{align*}
The last factor $\frac{\left(\eta+\kappa\right)^{2}}{2T_{n}}\lesssim1$;
hence it is negligible. By a direct computation of Fourier transformation,
we have 
\begin{align*}
 & \mathcal{F}_{\xi}^{-1}\iint\exp\left(i\left(\frac{\xi^{2}}{T_{n}}-\frac{\xi_{0}\left(\eta+\kappa\right)}{T_{n}}-\frac{\left(\xi-\xi_{0}\right)\left(\eta+\kappa\right)}{T_{n}}\right)\right)\widehat{f_{1}}\left(\kappa\right)\widehat{\overline{f_{2}}}\left(\xi-\eta-\kappa\right)\widehat{f_{3}}\left(\eta\right)d\kappa d\eta\\
= & e^{-i\frac{\partial_{v}^{2}}{T_{n}}}\left[f_{1}\left(v-\frac{\xi_{0}}{T_{n}}\right)\overline{f_{2}(v)}f_{3}\left(v-\frac{\xi_{0}}{T_{n}}\right)\right].
\end{align*}
Hence we can split $\tilde{M}$ into three parts:
\[
\tilde{M}\left(t,v,y\right):=M_{+}\left(t,v,y\right)+M_{-}(t,v,y)+\text{err}_{3}\left(t,v,y\right).
\]
Where $M_{+}$ and $M_{-}$ correspond to positive frequency localization
$A^{+}(\omega,k):=A(\omega,k)\cap\{\xi>0\}$ and negative freqency
localization $A^{-}(\omega,k):=A(\omega,k)\cap\{\xi<0\}$.
\begin{align*}
M_{+}\left(t,v,y\right) & :=t^{-1}\mathcal{F}_{\xi}^{-1}\sum_{\mathbf{k}}\sum_{\omega\neq0}\sum_{\Gamma_{\omega}\left(\mathbf{k}\right)}\sum_{k}\\
 & \quad\quad e^{-i\frac{\partial_{v}^{2}}{T_{n}}+i\frac{\omega}{2}T_{n}}\sum_{k}\tilde{G}\left(T_{n},v-\sqrt{\frac{\omega}{2}}-\frac{2k}{\sqrt{T_{n}}},\mathbf{k}_{1}\right)\overline{\mathcal{Z}_{A^{+}(\omega,k)}}\left(T_{n},v,\mathbf{k}_{2}\right)\tilde{G}\left(T_{n},v-\sqrt{\frac{\omega}{2}}-\frac{2k}{\sqrt{T_{n}}},\mathbf{k}_{3}\right),
\end{align*}
\begin{align*}
M_{-}\left(t,v,y\right) & :=t^{-1}\mathcal{F}_{\xi}^{-1}\sum_{\mathbf{k}}\sum_{\omega\neq0}\sum_{\Gamma_{\omega}\left(\mathbf{k}\right)}\sum_{k}\\
 & \quad\quad e^{-i\frac{\partial_{v}^{2}}{T_{n}}+i\frac{\omega}{2}T_{n}}\sum_{k}\tilde{G}\left(T_{n},v+\sqrt{\frac{\omega}{2}}-\frac{2k}{\sqrt{T_{n}}},\mathbf{k}_{1}\right)\overline{\mathcal{Z}_{A^{-}(\omega,k)}}\left(T_{n},v,\mathbf{k}_{2}\right)\tilde{G}\left(T_{n},v+\sqrt{\frac{\omega}{2}}-\frac{2k}{\sqrt{T_{n}}},\mathbf{k}_{3}\right).
\end{align*}
The error from transforming the effect of phase function into spatial
translation is given by 
\[
\text{err}_{3}(t,v,y):=\tilde{M}(t,v,y)-M_{+}(t,v,y)-M_{-}(t,v,y).
\]

\begin{lemma}For $t\geq1$, the error from approximating the phase
function by spatial translation of $v$ has the bound
\begin{equation}
\left\Vert \text{err}_{3}\left(t,v,y\right)\right\Vert _{L_{v,y}^{2}}\lesssim D^{\frac{73}{42}}\epsilon^{2}t^{-\frac{25}{24}}\left(1+t\right)^{\frac{73}{42}\delta}\left\Vert \nu\left(t\right)\right\Vert _{L_{x,y}^{2}}.\label{eq:err_3}
\end{equation}

\end{lemma}

\begin{proof}

From the inequality 
\[
\left|-\frac{\left(\xi-\xi_{0}\right)\left(\eta+\kappa\right)}{T_{n}}+\frac{\left(\eta+\kappa\right)^{2}}{2T_{n}}\right|\lesssim\frac{\left|\eta+\kappa\right|}{\sqrt{T_{n}}}+\frac{\left(\eta+\kappa\right)^{2}}{2T_{n}},
\]
thus we know that 

\begin{align*}
 & \left\Vert \iint\left|\frac{\eta+\kappa}{\sqrt{T_{n}}}\right|^{\frac{1}{12}}\left|\widehat{\tilde{G}}\left(T_{n},\kappa,\mathbf{k}_{1}\right)\widehat{\overline{\mathcal{Z}}_{A(\omega)}}\left(T_{n},\xi-\eta-\kappa,\mathbf{k}_{2}\right)\widehat{\tilde{G}}\left(T_{n},\eta,\mathbf{k}_{3}\right)\right|d\kappa d\eta\right\Vert _{L_{\xi}^{2}}\\
\lesssim & T_{n}^{-\frac{1}{24}}\left\Vert |\xi|^{\frac{1}{12}}\widehat{\tilde{G}}\left(T_{n},\xi,\mathbf{k}_{1}\right)\right\Vert _{L_{\xi}^{1}}\left\Vert \widehat{\mathcal{Z}_{A(\omega)}}\left(T_{n},\xi,\mathbf{k}_{2}\right)\right\Vert _{L_{\xi}^{2}}\left\Vert \widehat{\tilde{G}}\left(T_{n},\xi,\mathbf{k}_{3}\right)\right\Vert _{L_{\xi}^{1}}\\
 & \quad+T_{n}^{-\frac{1}{24}}\left\Vert \widehat{\tilde{G}}\left(t,\xi,\mathbf{k}_{1}\right)\right\Vert _{L_{\xi}^{1}}\left\Vert \widehat{\mathcal{Z}_{A(\omega)}}\left(T_{n},\xi,\mathbf{k}_{2}\right)\right\Vert _{L_{\xi}^{2}}\left\Vert |\xi|^{\frac{1}{12}}\widehat{\tilde{G}}\left(T_{n},\xi,\mathbf{k}_{3}\right)\right\Vert _{L_{\xi}^{1}}.
\end{align*}
From (\ref{eq:ele}) and let $a=\frac{7}{12}$, we have the following
inequality 
\begin{equation}
\left\Vert |\xi|^{\frac{1}{12}}\widehat{\tilde{G}}\left(T_{n},\xi,\mathbf{k}\right)\right\Vert _{L_{\xi}^{1}h_{\mathbf{k}}^{\alpha}}\lesssim\left\Vert |\xi|^{\frac{1}{12}}\widehat{\tilde{G}}\right\Vert _{L_{\xi}^{2}h_{\mathbf{k}}^{\alpha}}^{\frac{1}{7}}\left\Vert |\xi|^{\frac{2}{3}}\widehat{\tilde{G}}\right\Vert _{L_{\xi}^{2}h_{\mathbf{k}}^{\alpha}}^{\frac{6}{7}}\lesssim\left\Vert \widehat{\tilde{G}}\right\Vert _{L_{\xi}^{2}l_{\mathbf{k}}^{2}}^{\frac{2}{21}}\left\Vert \widehat{\tilde{G}}\right\Vert _{L_{\xi}^{2}h_{\mathbf{k}}^{s}}^{\frac{1}{3}}\left\Vert \left|\xi\right|\widehat{\tilde{G}}\right\Vert _{L_{\xi}^{2}l_{\mathbf{k}}^{2}}^{\frac{4}{7}}.\label{eq:tild_g_est}
\end{equation}
Therefore when summing over $\mathcal{M}\left(\mathbf{k}\right)$
by using Lemma \ref{elementarylp}, (\ref{eq:l1est}), (\ref{eq:tild_g_est}),
and the fact that the sets $A(\omega)$ are disjoint for different
$\omega$, we have
\begin{align*}
 & \left\Vert \text{err}_{3}\left(t,v,y\right)\right\Vert _{L_{v,y}^{2}}\lesssim t^{-\frac{25}{24}}\left\Vert \left|\xi\right|^{\frac{1}{12}}\widehat{\tilde{G}}\left(T_{n},\xi,\mathbf{k}\right)\right\Vert _{L_{\xi}^{1}h_{\mathbf{k}}^{\alpha}}\left\Vert \widehat{\tilde{G}}\left(T_{n},\xi,\mathbf{k}\right)\right\Vert _{L_{\xi}^{1}h_{\mathbf{k}}^{\alpha}}\left(\sum_{j}\left\Vert \widehat{\mathcal{Z}_{A(j)}}\left(T_{n},\xi,\mathbf{k}\right)\right\Vert _{l_{\mathbf{k}}^{2}L_{\xi}^{2}}^{2}\right)^{\frac{1}{2}}\\
\lesssim & t^{-\frac{25}{24}}\left\Vert \widehat{\tilde{G}}\left(T_{n}\right)\right\Vert _{L_{\xi}^{2}l_{\mathbf{k}}^{2}}^{\frac{11}{42}}\left\Vert \widehat{\tilde{G}}\left(T_{n}\right)\right\Vert _{L_{\xi}^{2}h_{\mathbf{k}}^{s}}^{\frac{2}{3}}\left\Vert \xi\widehat{\tilde{G}}\left(T_{n}\right)\right\Vert _{l_{\mathbf{k}}^{2}L_{\xi}^{2}}^{\frac{15}{14}}\left\Vert \mathcal{Z}\left(T_{n}\right)\right\Vert _{L_{v,y}^{2}}\lesssim D^{\frac{73}{42}}\epsilon^{2}t^{-\frac{25}{24}}\left(1+t\right)^{\frac{73}{42}\delta}\left\Vert \nu\left(t\right)\right\Vert _{L_{x,y}^{2}}.
\end{align*}
\end{proof}

To simplify the next computations, we consider the case where $\xi>0$;
the case $\xi<0$ is similar by changing the sign. First notice that
while doing summation over all the possible triples, due to the fact
that the $A(\omega)$'s are disjoint, the $e^{i\frac{\omega}{2}T_{n}}$
does not affect the computation of $L_{v,y}^{2}$ norm 
\[
\begin{aligned} & \left\Vert \sum_{\omega\neq0}\sum_{\Gamma_{\omega}\left(\mathbf{k}\right)}\sum_{k}e^{-i\frac{\partial_{v}^{2}}{T_{n}}+i\frac{\omega}{2}T_{n}}\tilde{G}\left(T_{n},v-\sqrt{\frac{\omega}{2}}-\frac{2k}{\sqrt{T_{n}}},\mathbf{k}_{1}\right)\overline{\mathcal{Z}_{A^{+}(\omega,k)}}\left(T_{n},v,\mathbf{k}_{2}\right)\tilde{G}\left(T_{n},v-\sqrt{\frac{\omega}{2}}-\frac{2k}{\sqrt{T_{n}}},\mathbf{k}_{3}\right)\right\Vert _{L_{v}^{2}}\\
= & \left\Vert \sum_{\omega\neq0}\sum_{\Gamma_{\omega}\left(\mathbf{k}\right)}\sum_{k}\tilde{G}\left(T_{n},v-\sqrt{\frac{\omega}{2}}-\frac{2k}{\sqrt{T_{n}}},\mathbf{k}_{1}\right)\overline{\mathcal{Z}_{A^{+}(\omega,k)}}\left(T_{n},v,\mathbf{k}_{2}\right)\tilde{G}\left(T_{n},v-\sqrt{\frac{\omega}{2}}-\frac{2k}{\sqrt{T_{n}}},\mathbf{k}_{3}\right)\right\Vert _{L_{v}^{2}}.
\end{aligned}
\]
Hence it suffices only to consider the $L^{2}$ norm of the following
function 
\begin{align*}
M\left(t,v,y\right) & =t^{-1}\sum_{\mathbf{k}}\sum_{\omega\neq0}\sum_{\Gamma_{\omega}\left(\mathbf{k}\right)}\sum_{k}\left[\tilde{G}\left(T_{n},v-\sqrt{\frac{\omega}{2}}-\frac{2k}{\sqrt{T_{n}}},\mathbf{k}_{1}\right)\overline{\mathcal{Z}_{A^{+}(\omega,k)}}\left(T_{n},v,\mathbf{k}_{2}\right)\tilde{G}\left(T_{n},v-\sqrt{\frac{\omega}{2}}-\frac{2k}{\sqrt{T_{n}}},\mathbf{k}_{3}\right)\right.\\
 & \quad+\left.\tilde{G}\left(T_{n},v+\sqrt{\frac{\omega}{2}}-\frac{2k}{\sqrt{T_{n}}},\mathbf{k}_{1}\right)\overline{\mathcal{Z}_{A^{-}(\omega,k)}}\left(T_{n},v,\mathbf{k}_{2}\right)\tilde{G}\left(T_{n},v+\sqrt{\frac{\omega}{2}}-\frac{2k}{\sqrt{T_{n}}},\mathbf{k}_{3}\right)\right]e^{i\mathbf{k}\cdot y}.
\end{align*}
When computing the $L_{v}^{2}$ norm, we can apply a shift to $v$
and obtain the same value,
\begin{align*}
 & \left\Vert \tilde{G}\left(T_{n},v-\sqrt{\frac{\omega}{2}}-\frac{2k}{\sqrt{T_{n}}},\mathbf{k}_{1}\right)\overline{\mathcal{Z}_{A^{+}(\omega,k)}}\left(T_{n},v,\mathbf{k}_{2}\right)\tilde{G}\left(T_{n},v-\sqrt{\frac{\omega}{2}}-\frac{2k}{\sqrt{T_{n}}},\mathbf{k}_{3}\right)\right\Vert _{L_{v}^{2}}\\
= & \left\Vert \tilde{G}\left(T_{n},v,\mathbf{k}_{1}\right)\overline{\mathcal{Z}_{A^{+}(\omega,k)}}\left(t,v+\sqrt{\frac{\omega}{2}}+\frac{2k}{\sqrt{T_{n}}},\mathbf{k}_{2}\right)\tilde{G}\left(T_{n},v,\mathbf{k}_{3}\right)\right\Vert _{L_{v}^{2}}.
\end{align*}
Since $\text{dist}\left(A^{+}\left(\omega,k\right),A^{+}\left(\omega,k+5\right)\right)\geq8\sqrt{t}$,
we sort the $k$ into five groups. There is the equality
\[
\begin{aligned} & \left\Vert \sum_{\omega\neq0}\sum_{\Gamma_{\omega}\left(\mathbf{k}\right)}\sum_{m}\tilde{G}\left(T_{n},v-\sqrt{\frac{\omega}{2}}-\frac{2(5m+i)}{\sqrt{T_{n}}},\mathbf{k}_{1}\right)\overline{\mathcal{Z}_{A^{+}(\omega,5m+i)}}\left(T_{n},v,\mathbf{k}_{2}\right)\tilde{G}\left(T_{n},v-\sqrt{\frac{\omega}{2}}-\frac{2(5m+i)}{\sqrt{T_{n}}},\mathbf{k}_{3}\right)\right\Vert _{L_{v}^{2}}\\
= & \left\Vert \sum_{\omega\neq0}\sum_{\Gamma_{\omega}\left(\mathbf{k}\right)}\sum_{m}\tilde{G}\left(T_{n},v,\mathbf{k}_{1}\right)\overline{\mathcal{Z}_{A^{+}(\omega,5m+i)}}\left(T_{n},v+\sqrt{\frac{\omega}{2}}+\frac{2(5m+i)}{\sqrt{T_{n}}},\mathbf{k}_{2}\right)\tilde{G}\left(T_{n},v,\mathbf{k}_{3}\right)\right\Vert _{L_{v}^{2}},
\end{aligned}
\]
where $i=0,1,2,3,4$. Hence we have the following bound
\[
\begin{aligned} & \left\Vert \sum_{\omega\neq0}\sum_{\Gamma_{\omega}\left(\mathbf{k}\right)}\sum_{k}\tilde{G}\left(T_{n},v-\sqrt{\frac{\omega}{2}}-\frac{2k}{\sqrt{T_{n}}},\mathbf{k}_{1}\right)\overline{\mathcal{Z}_{A^{+}(\omega,k)}}\left(T_{n},v,\mathbf{k}_{2}\right)\tilde{G}\left(T_{n},v-\sqrt{\frac{\omega}{2}}-\frac{2k}{\sqrt{T_{n}}},\mathbf{k}_{3}\right)\right\Vert _{L_{v}^{2}}\\
\lesssim & \left\Vert \sum_{\omega\neq0}\sum_{\Gamma_{\omega}\left(\mathbf{k}\right)}\sum_{k}\tilde{G}\left(T_{n},v,\mathbf{k}_{1}\right)\overline{\mathcal{Z}_{A^{+}(\omega,k)}}\left(T_{n},v+\sqrt{\frac{\omega}{2}}+\frac{2k}{\sqrt{T_{n}}},\mathbf{k}_{2}\right)\tilde{G}\left(T_{n},v,\mathbf{k}_{3}\right)\right\Vert _{L_{v}^{2}}.
\end{aligned}
\]

Moreover we define the function
\[
\mathcal{Z}_{+}\left(\sigma,v,\mathbf{k}_{2}\right)\left(T\right):=\sum_{1\leq j\leq T_{n}^{\frac{1}{16}}}\sum_{k}e^{ij\sigma/2}\mathcal{Z}_{A^{+}\left(j,k\right)}\left(T_{n},v+\sqrt{\frac{j}{2}}+\frac{2k}{\sqrt{T_{n}}},\mathbf{k}_{2}\right),
\]
\[
\mathcal{Z}_{-}\left(\sigma,v,\mathbf{k}_{2}\right)\left(T\right):=\sum_{1\leq j\leq T_{n}^{\frac{1}{16}}}\sum_{k}e^{ij\sigma/2}\mathcal{Z}_{A^{-}\left(j,k\right)}\left(T_{n},v-\sqrt{\frac{j}{2}}+\frac{2k}{\sqrt{T_{n}}},\mathbf{k}_{2}\right),
\]
where $\sigma\in\left[0,2\pi\right)$ which is akin to $\sigma=t-T_{n}$.
Notice that when $T_{n}$ is sufficiently large we have each $A^{+}(j)$
are disjoint hence $\left\Vert \mathcal{Z}_{+}\right\Vert _{L_{v,y}^{2}}\leq\left\Vert V\right\Vert _{L_{v,y}^{2}}.$
Using that the distance between $A^{+}(j)$ and $A^{+}(j+1)$ is greater
than $\sqrt{t}$ again, it follows that $\mathcal{F}_{v}G\overline{\mathcal{Z}_{A_{j}}}$
also have disjoint support for different $j$. 

\begin{lemma}For $t\geq1$, there is the bound

\[
\left\Vert M\left(t\right)\right\Vert _{L_{v,y}^{2}}\lesssim t^{-1}\left\Vert e^{-i\sigma\triangle_{y}/2}\left[\left(e^{i\sigma\triangle_{y}/2}\tilde{G}\left(T_{n},v,y\right)\right)^{2}\overline{e^{i\sigma\triangle_{y}/2}\left[\mathcal{Z}_{+}+\mathcal{Z}_{-}\right]\left(\sigma,v,y\right)}\right]\right\Vert _{L_{\sigma}^{2}\left(\mathbb{T};L_{v,y}^{2}\right)}.
\]

\end{lemma}

\begin{proof} 

Given any $\varphi\in L_{v,y}^{2}$,

\begin{align*}
 & \left\langle \varphi\left(v,y\right),e^{-i\sigma\triangle_{y}/2}\left[\left(e^{i\sigma\triangle_{y}/2}\tilde{G}\right)^{2}\left(T_{n},v,y\right)\overline{e^{i\sigma\triangle_{y}/2}\mathcal{Z}_{+}}\left(\sigma,v,y\right)\right]\right\rangle _{L_{v,y}^{2}}\\
= & \left\langle e^{i\sigma\triangle_{y}/2}\varphi\left(v,\mathbf{k}\right),\left(e^{i\sigma\triangle_{y}/2}\tilde{G}\right)^{2}\left(T,v,y\right)\overline{e^{i\sigma\triangle_{y}/2}\mathcal{Z}_{+}}\left(\sigma,v,y\right)\right\rangle _{L_{v}^{2}l_{k}^{2}}\\
= & \sum_{\omega\neq0}\sum_{\Gamma_{\omega}}\sum_{1\leq j\leq T_{n}^{\frac{1}{16}}}e^{i\left(\omega-j\right)\sigma/2}\left\langle \varphi\left(v,\mathbf{k}\right),\tilde{G}\left(T_{n},v,\mathbf{k}_{1}\right)\sum_{k}\overline{\mathcal{Z}_{A^{+}\left(j,k\right)}}\left(T,v+\sqrt{\frac{j}{2}}+\frac{2k}{\sqrt{T_{n}}},\mathbf{k}_{4}\right)\tilde{G}\left(T,v,\mathbf{k}_{3}\right)\right\rangle _{L_{\xi}^{2}L_{y}^{2}}.
\end{align*}
Only when $j=\omega=\left|k_{1}\right|^{2}-\left|k_{2}\right|^{2}+\left|k_{3}\right|^{2}-\left|k_{4}\right|^{2}$
can the integral in time be nonzero. In other cases there will be
$e^{it\left(j-\omega\right)/2}$ in front of the inner product and
integrating in $t$ over the unit time interval will be $0$. Hence
\begin{align*}
 & t^{-1}\int_{0}^{2\pi}\left\langle \varphi\left(v,y\right),e^{-i\sigma\triangle_{y}/2}\left[\left(e^{i\sigma\triangle_{y}/2}\tilde{G}\right)^{2}\left(T_{n},v,y\right)\overline{e^{i\sigma\triangle_{y}/2}\mathcal{Z}_{+}}\left(\sigma,v,y\right)\right]\right\rangle _{L_{v,y}^{2}}d\sigma\\
= & \left\langle \varphi(v,y),t^{-1}\sum_{\omega\neq0}\sum_{\Gamma_{\omega}\left(\mathbf{k}\right)}\sum_{k}\tilde{G}\left(T_{n},v,\mathbf{k}_{1}\right)\overline{\mathcal{Z}_{A^{+}(\omega,k)}}\left(t,v+\sqrt{\frac{\omega}{2}}+\frac{2k}{\sqrt{T_{n}}},\mathbf{k}_{2}\right)\tilde{G}\left(T_{n},v,\mathbf{k}_{3}\right)\right\rangle _{L_{v,y}^{2}}.
\end{align*}
Since $\varphi$ is an arbitrary function, by the definition of $L^{2}-$norm
, we know that the lemma holds.

\end{proof}

\begin{lemma}For any $t\in\left[T_{n},T_{n+1}\right),$ there is
the bound
\begin{equation}
\left\Vert M\left(t\right)\right\Vert _{L_{v,y}^{2}}\lesssim t^{-1}\left\Vert G(T_{n},v,y)\right\Vert _{L_{v}^{\infty}H_{y}^{1}}^{2}\left\Vert V\left(T_{n},v,y\right)\right\Vert _{L_{v,y}^{2}}.\label{eq:Mbound}
\end{equation}

\end{lemma}

\begin{proof}

From the previous lemma, for any $\varphi\in L_{v,y}^{2}$ there is
the inequality

\begin{align*}
 & t^{-1}\left\langle \varphi\left(v,y\right),e^{-i\sigma\triangle_{y}/2}\left[\left(e^{i\sigma\triangle_{y}/2}\tilde{G}\right)^{2}\left(T_{n},v,y\right)\overline{e^{i\sigma\triangle_{y}/2}\mathcal{Z}_{+}}\left(\sigma,v,y\right)\right]\right\rangle _{L_{\sigma,v,y}^{2}}\\
\lesssim & t^{-1}\left\Vert \overline{e^{i\sigma\triangle_{y}/2}\varphi}e^{i\sigma\triangle_{y}/2}\tilde{G}\left(T_{n},v,y\right)\right\Vert _{L_{\sigma,v,y}^{2}}\left\Vert e^{i\sigma\triangle_{y}/2}\tilde{G}\left(T_{n},v,y\right)\overline{e^{i\sigma\triangle_{y}/2}\mathcal{Z}_{+}}\left(\sigma,v,y\right)\right\Vert _{L_{\sigma,v,y}^{2}}.
\end{align*}
Thus by Fourier transformation and disjointness of domains for different
$j$, applying the bilinear Strichartz estimate we will have the $L^{2}$
bound: 
\begin{align*}
 & \left\Vert \left[e^{i\sigma\triangle_{y}/2}\tilde{G}\left(T_{n},v,y\right)\right]\overline{e^{i\sigma\triangle_{y}/2}\mathcal{Z}_{+}}\left(\sigma,v,y\right)\right\Vert _{L_{\sigma}^{2}\left(\mathbb{T};L_{v,y}^{2}\right)}\\
= & \sum_{1\leq j\leq T^{\frac{1}{16}}}\left\Vert e^{ij\sigma/4}\mathcal{F}_{v}\left(e^{i\sigma\triangle_{y}/2}\tilde{G}\left(T_{n},v,y\right)e^{-i\sigma\triangle_{y}/2}\sum_{k}\overline{\mathcal{Z}_{A^{+}\left(j,k\right)}}\left(T_{n},v+\sqrt{\frac{j}{2}}+\frac{2k}{\sqrt{T_{n}}},y\right)\right)\right\Vert _{L_{\sigma}^{2}\left(\mathbb{T};L_{\xi,y}^{2}\right)}\\
= & \left\Vert e^{i\sigma\triangle_{y}/2}\tilde{G}\left(T_{n},v,y\right)e^{-i\sigma\triangle_{y}/2}\left[\sum_{1\leq j\leq T^{\frac{1}{16}}}\sum_{k}\overline{\mathcal{Z}_{A^{+}\left(j,k\right)}}\left(T_{n},v+\sqrt{\frac{j}{2}}+\frac{2k}{\sqrt{T_{n}}},y\right)\right]\right\Vert _{L_{\sigma}^{2}\left(\mathbb{T};L_{v,y}^{2}\right)}\\
\lesssim & \left\Vert \tilde{G}(T,v,y)\right\Vert _{L_{v}^{\infty}H_{y}^{1}}\left\Vert \sum_{1\leq j\leq T^{\frac{1}{16}}}\sum_{k}\overline{\mathcal{Z}_{A^{+}\left(j,k\right)}}\left(T_{n},v+\sqrt{\frac{j}{2}}+\frac{2k}{\sqrt{T_{n}}},y\right)\right\Vert _{L_{v,y}^{2}}\\
\lesssim & \left\Vert G(T_{n},v,y)\right\Vert _{L_{v}^{\infty}H_{y}^{1}}\left\Vert V\left(T_{n},v,y\right)\right\Vert _{L_{v,y}^{2}}.
\end{align*}
The proof of bilinear Strichartz estimate on $L_{\sigma,y}^{2}$ is
given in \cite{HTT2}. Similarly, 
\begin{align*}
\left\Vert \overline{e^{i\sigma\triangle_{y}/2}\varphi}e^{i\sigma\triangle_{y}/2}\tilde{G}\left(T_{n},v,y\right)\right\Vert _{L_{\sigma}^{2}\left(\mathbb{T};L_{v,y}^{2}\right)} & \lesssim\left\Vert \tilde{G}(T_{n},v,y)\right\Vert _{L_{v}^{\infty}H_{y}^{1}}\left\Vert \varphi\right\Vert _{L_{v,y}^{2}}\lesssim\left\Vert G(T_{n},v,y)\right\Vert _{L_{v}^{\infty}H_{y}^{1}}\left\Vert \varphi\right\Vert _{L_{v,y}^{2}}.
\end{align*}
Since $\varphi$ is an arbitrary $L_{v,y}^{2}$ function, by the property
of $L^{2}-$norm we know that for $t\in[T_{n},T_{n+1}),$
\[
\left\Vert M\left(t\right)\right\Vert _{L_{v,y}^{2}}\lesssim t^{-1}\left\Vert \gamma(T_{n},v,y)\right\Vert _{L_{v}^{\infty}H_{y}^{1}}^{2}\left\Vert \nu\left(T_{n},v,y\right)\right\Vert _{L_{x,y}^{2}}.
\]
\end{proof}

In order to switch back to estimate of integration in time, by (\ref{eq:tdiffbound})
for any $t\in\left[T_{n},T_{n+1}\right)$ we have
\begin{align*}
\left\Vert M\left(t\right)\right\Vert _{L_{v,y}^{2}} & \lesssim D^{2}\epsilon^{2}t^{-1}\left\Vert \mathcal{Z}\left(T_{n},v,y\right)\right\Vert _{L_{v,y}^{2}}\\
 & \lesssim D^{2}\epsilon^{2}t{}^{-1}\left\Vert \nu\left(t\right)\right\Vert _{L_{x,y}^{2}}+D^{\frac{11}{3}}\epsilon^{4}t^{-2}\left(1+t\right)^{\frac{5}{3}\delta}\left\Vert \nu\left(t\right)\right\Vert _{L_{x,y}^{2}}.
\end{align*}

Since 
\[
e_{4}(t,v,y)=M_{+}(t,v,y)+M_{-}(t,v,y)+\text{err}_{1}(t,v,y)+\text{err}_{2}(t,v,y)+\text{err}_{3}(t,v,y),
\]
combining all the estimates (\ref{eq:err_1}), (\ref{eq:err_2}),
(\ref{eq:err_3}) together, the estimate (\ref{eq:e4bound}) is proved.
\begin{align*}
\left\Vert e_{4}\left(t\right)\right\Vert _{L_{v,y}^{2}} & \lesssim D^{2}\epsilon^{2}t^{-1}\left\Vert \nu\left(t\right)\right\Vert _{L_{x,y}^{2}}+D^{\frac{11}{3}}\epsilon^{4}t^{-2}\left(1+t\right)^{\frac{5}{3}\delta}\left\Vert \nu\left(t\right)\right\Vert _{L_{x,y}^{2}}\\
 & \quad+\left\Vert \text{err}_{1}\left(t\right)\right\Vert _{L_{v,y}^{2}}+\left\Vert \text{err}_{2}\left(t\right)\right\Vert _{L_{v,y}^{2}}+\left\Vert \text{err}_{3}\left(t\right)\right\Vert _{L_{v,y}^{2}}.
\end{align*}

\subsection{Proof of Proposition \ref{prop. nugrowth}}

By (\ref{eq:nuintegral}), we have the bound

\begin{align*}
\left\Vert \nu\left(T\right)\right\Vert _{L_{x,y}^{2}}^{2} & \leq\left\Vert \nu\left(1\right)\right\Vert _{L_{x,y}^{2}}^{2}+2\int_{1}^{T}\left(\left\Vert e_{1}\right\Vert _{L_{v,y}^{2}}+\left\Vert e_{2}\right\Vert _{L_{v,y}^{2}}+\left\Vert e_{4}\right\Vert _{L_{v,y}^{2}}\right)\left\Vert \nu(t)\right\Vert _{L_{x,y}}dt+\left|\int_{1}^{T}\left\langle \nu(t),e_{3}\right\rangle _{L_{x,y}^{2}}dt\right|.
\end{align*}

Combining (\ref{eq:e1bound}),(\ref{eq:e2bound}),(\ref{eq:e3bound})
and (\ref{eq:e4bound}) together, we assume that $D$ is a sufficiently
large positive constant which only depends on $d$ and $s$. Notice
that the choice of $D$ does not depend on $\epsilon$ and $T$. Therefore
the following inequality holds:

\begin{align*}
 & \left[1-D^{\frac{8}{3}}\epsilon^{2}T^{-\frac{5}{8}}\left(1+T\right)^{\frac{5}{3}\delta}\right]\left\Vert \nu\left(T\right)\right\Vert _{L_{x,y}^{2}}^{2}\\
\leq & \left\Vert \nu\left(1\right)\right\Vert _{L_{x,y}^{2}}^{2}+D^{\frac{8}{3}}\epsilon^{2}\left\Vert \nu\left(1\right)\right\Vert _{L_{x,y}^{2}}^{2}+\int_{1}^{T}D^{3}\epsilon^{2}t^{-1}\left\Vert \nu(t)\right\Vert _{L_{x,y}^{2}}^{2}dt\\
 & +D\int_{1}^{T}\left(D^{\frac{11}{6}}\epsilon^{2}+D^{\frac{11}{3}}\epsilon^{4}+D^{\frac{10}{3}}\epsilon^{4}+D^{\frac{73}{42}}\epsilon^{2}+D^{\frac{5}{3}}\epsilon^{2}\right)\left(1+t\right)^{-\frac{65}{64}+\frac{11}{6}\delta}\left\Vert \nu(t)\right\Vert _{L_{x,y}^{2}}^{2}dt.
\end{align*}
By Gronwall's inequality for any $t\in[1,T]$ there is the inequality
\begin{align*}
\left\Vert \nu\left(t\right)\right\Vert _{L_{x,y}^{2}}^{2}\leq & \frac{\left(1+D^{\frac{8}{3}}\epsilon^{2}\right)}{1-D^{\frac{8}{3}}\epsilon^{2}}\left\Vert \nu\left(1\right)\right\Vert _{L_{x,y}^{2}}^{2}\exp\left(\frac{D^{3}\epsilon^{2}}{1-D^{\frac{8}{3}}\epsilon^{2}}\log(1+t)\right)\\
 & \quad\exp\left(\frac{D}{1-D^{\frac{8}{3}}\epsilon^{2}}\left(D^{\frac{11}{6}}\epsilon^{2}+D^{\frac{11}{3}}\epsilon^{4}+D^{\frac{10}{3}}\epsilon^{4}+D^{\frac{73}{42}}\epsilon^{2}+D^{\frac{5}{3}}\epsilon^{2}\right)\right).
\end{align*}
Thus if we take $\epsilon\ll D^{-\frac{3}{2}}$ such that $0<D^{3}\epsilon^{2}<\delta<\frac{1}{2},$
we have the desired inequality 
\begin{equation}
\left\Vert \nu\left(t\right)\right\Vert _{L_{x,y}^{2}}^{2}\lesssim\left\Vert \nu\left(1\right)\right\Vert _{L_{x,y}^{2}}^{2}\left(1+t\right)^{2D^{3}\epsilon^{2}}\lesssim\left\Vert \nu\left(0\right)\right\Vert _{L_{x,y}^{2}}^{2}\left(1+t\right)^{2D^{3}\epsilon^{2}}.\label{eq:energybd}
\end{equation}

\begin{remark}

In the special case $d=1$, since we have the fact $\left|\gamma\right|$
is bounded, we can get a much simpler estimate.

\end{remark}

\subsection{Correction to the Leibnitz derivative rule}

In order to finish proof of the estimate for $D_{y}^{s}u$, we need
to prove the following lemma. Then the same estimate of $\left\Vert \nu\right\Vert _{L_{x,y}^{2}}$
for the solution $\nu$ to the linearized equation can be directly
applied to $\left\Vert D_{y}^{s}u\right\Vert _{L_{x,y}^{2}}.$

\begin{lemma}\label{LemmaD_ycorrect}

Assume that $u$ satisfies Hypothesis \ref{hyp}. Then there is the
bound
\begin{equation}
\begin{aligned}\left|\int_{1}^{T}\left\langle D_{y}^{s}u,\text{cor}\left(t\right)\right\rangle _{L_{x,y}^{2}}dt\right|\lesssim & D^{\frac{5}{3}}\epsilon^{2}\left\Vert D_{y}^{s}u\left(1\right)\right\Vert _{L_{x,y}^{2}}^{2}+D^{\frac{5}{3}}\epsilon^{2}T^{-1}\left(1+T\right)^{\frac{5}{3}\delta}\left\Vert D_{y}^{s}u\left(T\right)\right\Vert _{L_{x,y}^{2}}^{2}\\
 & +\int_{1}^{T}t^{-1}D^{2}\epsilon^{2}\left\Vert D_{y}^{s}u\right\Vert _{L_{x,y}^{2}}^{2}dt+\int_{1}^{T}D^{\frac{d}{s}+\tilde{\beta}}\epsilon^{2}t^{-1-\tilde{\beta}}\left(1+t\right)^{\left(\frac{d}{s}+\tilde{\beta}\right)\delta}\left\Vert D_{y}^{s}u(t)\right\Vert _{L_{x,y}^{2}}^{2}dt\\
 & +\int_{1}^{T}\left[D^{\frac{5}{3}}\epsilon^{2}t^{-2}\left(1+t\right)^{\frac{5}{3}\delta}+D^{\frac{10}{3}}\epsilon^{4}t^{-2}\left(1+t\right)^{\frac{10}{3}\delta}\right]\left\Vert D_{y}^{s}u\left(t\right)\right\Vert _{L_{x,y}^{2}}^{2}dt,
\end{aligned}
\label{eq:Ducorerro}
\end{equation}
for some positive number $\tilde{\beta}.$

\end{lemma}

\begin{proof}

We first split the nonlinear term into two parts: the low $v$-frequency
part and the high frequency part,
\[
t^{-1}\left|w\right|^{2}w=t^{-1}\left|\gamma\right|^{2}\gamma+t^{-1}\left(\left|w\right|^{2}w-\left|\gamma\right|^{2}\gamma\right).
\]
The first factor $t^{-1}\left|\gamma\right|^{2}\gamma$ is easy to
estimate. The inner product can be written as
\begin{align*}
\left\langle D_{y}^{s}w,D_{y}^{s}\left(t^{-1}\left|\gamma\right|^{2}\gamma\right)\right\rangle _{L_{v,y}^{2}}= & \left\langle P_{\leq5\sqrt{t}}D_{y}^{s}w,D_{y}^{s}\left(t^{-1}\left|\gamma\right|^{2}\gamma\right)\right\rangle _{L_{v,y}^{2}}.
\end{align*}
Then separate the inner product into a $y$ resonant factor and a
$y$ nonresonant factor. The resonant factor can be estimated by (\ref{eq:h1esti}),
and the elementary inequality
\[
\left\Vert D^{s}\left(fg\right)\right\Vert _{L^{2}}\lesssim\left\Vert \left(D^{s}f\right)g\right\Vert _{L^{2}}+\left\Vert f\left(D^{s}g\right)\right\Vert _{L^{2}},
\]
hence we have
\begin{align*}
t^{-1}\left\Vert R\left[\gamma,\gamma,\gamma\right]\right\Vert _{L_{v}^{2}H_{y}^{s}}\lesssim t^{-1}\left\Vert \gamma\right\Vert _{L_{v}^{\infty}H_{y}^{1}}^{2}\left\Vert D_{y}^{s}w\right\Vert _{L_{v,y}^{2}}.
\end{align*}
The nonresonant factor can be estimated by integration by parts in
time, which is almost the same as the computations in the proof of
Proposition \ref{prop:wproperty} (b), hence we omit it. Therefore
we have 
\begin{equation}
\begin{aligned}\left|\int_{1}^{T}\left\langle D_{y}^{s}u,t^{-1}\left|\gamma\right|^{2}\gamma\right\rangle _{L_{x,y}^{2}}dt\right| & \lesssim D^{\frac{5}{3}}\epsilon^{2}\left\Vert D_{y}^{s}u\left(1\right)\right\Vert _{L_{x,y}^{2}}^{2}+D^{\frac{5}{3}}\epsilon^{2}T^{-1}\left(1+T\right)^{\frac{5}{3}\delta}\left\Vert D_{y}^{s}u\left(T\right)\right\Vert _{L_{x,y}^{2}}^{2}\\
 & \quad+\int_{1}^{T}\left[D^{\frac{5}{3}}\epsilon^{2}t^{-2}\left(1+t\right)^{\frac{5}{3}\delta}+D^{\frac{10}{3}}\epsilon^{4}t^{-2}\left(1+t\right)^{\frac{10}{3}\delta}\right]\left\Vert D_{y}^{s}u\left(t\right)\right\Vert _{L_{x,y}^{2}}^{2}dt.
\end{aligned}
\label{eq:Dylow}
\end{equation}

In order to estimate the second term $t^{-1}\left(\left|w\right|^{2}w-\left|\gamma\right|^{2}\gamma\right)$,
we separate it into three parts 
\[
t^{-1}\left(P_{\geq\sqrt{t}}w|w|^{2}+\gamma\overline{P_{\geq\sqrt{t}}w}w+|\gamma|^{2}P_{\geq\sqrt{t}}w\right).
\]
We move the bound for the first term, and the others follow in the
same manner.

By Lemma \ref{lemma:correction}, the Leibnitz rule correction form,
it suffices to work with the estimate for 
\[
t^{-1}\left(D_{y}^{s_{1}}P_{\geq\sqrt{t}}w\right)\left(\overline{D_{y}^{s_{2}}w}\right)\left(D_{y}^{s_{3}}w\right),
\]
where 
\[
s_{1},s_{2},s_{3}\geq0,\quad s_{1}+s_{2}+s_{3}=s,\quad,\max\left\{ s_{1},s_{2},s_{3}\right\} \leq\max\{s-1,1\}.
\]

Let 
\[
p_{i}=\frac{s}{s_{i}},\quad\text{ for }i=1,2,3.
\]
If $\alpha_{i}=0$, we let $p_{i}=\infty.$
\begin{align*}
\left\Vert \left(D_{y}^{s_{1}}P_{\geq\sqrt{t}}w\right)\left(\overline{D_{y}^{s_{2}}w}\right)\left(D_{y}^{s_{3}}w\right)\right\Vert _{L_{y}^{2}} & \lesssim\left\Vert D_{y}^{s_{1}}P_{\geq\sqrt{t}}w\right\Vert _{L_{y}^{2p_{1}}}\left\Vert D_{y}^{s_{2}}w\right\Vert _{L_{y}^{2p_{2}}}\left\Vert D_{y}^{s_{3}}w\right\Vert _{L_{y}^{2p_{3}}}\\
 & \lesssim\left\Vert D_{y}^{s_{1}^{*}}P_{\geq\sqrt{t}}w\right\Vert _{L_{y}^{2}}\left\Vert D_{y}^{s_{2}^{*}}w\right\Vert _{L_{y}^{2}}\left\Vert D_{y}^{s_{3}^{*}}w\right\Vert _{L_{y}^{2}},
\end{align*}
where 
\[
s_{i}^{*}=s_{i}+\frac{d}{2}\left(1-\frac{s_{i}}{s}\right),\quad\text{for }s_{i}>0,\quad s_{i}^{*}=\frac{d}{2}+\delta^{*},\quad\text{for }s_{i}=0.
\]
Here $\delta^{*}$ is some sufficiently small positive number. Notice
that $s_{i}^{*}<s$ if $s_{i}<s$. 

There exist $q_{i}>1$, satisfying 
\[
\frac{1}{q_{1}}+\frac{1}{q_{2}}+\frac{1}{q_{3}}=1,
\]
and 
\[
\frac{1}{2}\left[1-\frac{1}{q_{i}}\right]+\frac{s_{i}^{*}}{s}<1.
\]
Define $\beta_{i}:=\frac{1}{2}\left(1-\frac{1}{q_{i}}\right)$. Then
we well have 
\begin{align*}
\left\Vert \left(D_{y}^{s_{1}}P_{\geq\sqrt{t}}w\right)\left(\overline{D_{y}^{s_{2}}w}\right)\left(D_{y}^{s_{3}}w\right)\right\Vert _{L_{v,y}^{2}} & \lesssim\left\Vert D_{y}^{s_{1}^{*}}P_{\geq\sqrt{t}}w\right\Vert _{L_{v}^{2q_{1}}L_{y}^{2}}\left\Vert D_{y}^{s_{2}^{*}}w\right\Vert _{L_{v}^{2q_{2}}L_{y}^{2}}\left\Vert D_{y}^{s_{3}^{*}}w\right\Vert _{L_{v}^{2q_{3}}L_{y}^{2}}\\
 & \lesssim\left\Vert D_{v}^{\beta_{1}}D_{y}^{s_{1}^{*}}P_{\geq\sqrt{t}}w\right\Vert _{L_{v,y}^{2}}\left\Vert D_{v}^{\beta_{2}}D_{y}^{s_{2}^{*}}w\right\Vert _{L_{v,y}^{2}}\left\Vert D_{v}^{\beta_{3}}D_{y}^{s_{3}^{*}}w\right\Vert _{L_{v,y}^{2}}.
\end{align*}
Since by interpolation 
\[
\left\Vert D_{v}^{\beta_{i}}D_{y}^{s_{i}^{*}}w\right\Vert _{L_{v,y}^{2}}\lesssim\left\Vert w\right\Vert _{L_{v,y}^{2}}^{1-\beta_{i}-s_{i}^{*}/s}\left\Vert \partial_{v}w\right\Vert _{L_{v,y}^{2}}^{\beta_{i}}\left\Vert D_{y}^{s}w\right\Vert _{L_{v,y}^{2}}^{s_{i}^{*}/s},
\]
for $D_{y}^{s_{i}}P_{\geq\sqrt{t}}w,$ we will have 
\[
\left\Vert D_{v}^{\beta_{1}}D_{y}^{s_{1}^{*}}P_{\geq\sqrt{t}}w\right\Vert _{L_{v,y}^{2}}\lesssim t^{-\frac{1}{2}\left(1-\beta_{1}-s_{1}^{*}/s\right)}\left\Vert \partial_{v}w\right\Vert _{L_{v,y}^{2}}^{1-s_{1}^{*}/s}\left\Vert D_{y}^{s}w\right\Vert _{L_{v,y}^{2}}^{s_{1}^{*}/s},
\]
which allows extra decay in time. Defining 
\[
\tilde{\beta}:=\min_{i}\left\{ \frac{1}{2}\left(1-\beta_{i}-s_{i}^{*}/s\right)\right\} 
\]
for all possible values of $i$, we will have
\begin{equation}
\left\Vert D_{y}^{s}\left(\left|w\right|^{2}w-\left|\gamma\right|^{2}\gamma\right)\right\Vert _{L_{v,y}^{2}}\lesssim D^{\frac{d}{s}+\tilde{\beta}}\epsilon^{2}t^{-\tilde{\beta}}\left(1+t\right)^{\frac{d}{s}+\tilde{\beta}}\left\Vert D_{y}^{s}u\right\Vert _{L_{x,y}^{2}}.\label{eq:Dyhigh}
\end{equation}
Combining (\ref{eq:Dylow}) and (\ref{eq:Dyhigh}), the lemma is proved.

\end{proof}

By (\ref{eq:Ducorerro}) and the same estimate for $\nu$ can be apply
to $2\left|u\right|^{2}D_{y}^{s}u+u^{2}\overline{D_{y}^{s}u}$; therefore
we have 

\begin{align*}
 & \left[1-D^{3}\epsilon^{2}T^{-\frac{5}{8}}\left(1+T\right)^{2\delta}-D^{\frac{8}{3}}\epsilon^{2}T^{-1}\left(1+T\right)^{\frac{5}{3}\delta}\right]\left\Vert D_{y}^{s}u\left(T\right)\right\Vert _{L_{x,y}^{2}}^{2}\\
\leq & \left(1+D^{3}\epsilon^{2}+D^{\frac{8}{3}}\epsilon^{2}\right)\left\Vert D_{y}^{s}u\left(1\right)\right\Vert _{L_{x,y}^{2}}^{2}+\int_{1}^{T}D^{3}\epsilon^{2}t^{-1}\left\Vert D_{y}^{s}(t)\right\Vert _{L_{x,y}^{2}}^{2}dt\\
 & +\int_{1}^{T}D^{\frac{d}{s}+\tilde{\beta}}\epsilon^{2}t^{-1-\tilde{\beta}}\left(1+t\right)^{\left(\frac{d}{s}+\tilde{\beta}\right)\delta}\left\Vert D_{y}^{s}u(t)\right\Vert _{L_{x,y}^{2}}^{2}dt\\
 & +D\int_{1}^{T}\left(D^{\frac{11}{6}}\epsilon^{3}+D^{\frac{11}{3}}\epsilon^{4}+D^{\frac{8}{3}}\epsilon^{4}+D^{2}\epsilon^{2}+D^{4}\epsilon^{4}+D^{\frac{5}{3}}\epsilon^{2}+D^{\frac{10}{3}}\epsilon^{4}\right)\\
 & \quad\quad\left(1+t\right)^{-\frac{65}{64}+\frac{11}{6}\delta}\left\Vert D_{y}^{s}u(t)\right\Vert _{L_{x,y}^{2}}^{2}dt.
\end{align*}
By Gronwall's inequality and the assumption $0<D^{3}\epsilon^{2}<\delta<\frac{1}{2}$,
there is the bound on $[1,T]$ that 
\begin{equation}
\left\Vert D_{y}^{s}u\left(t\right)\right\Vert _{L_{x,y}^{2}}^{2}\lesssim\left\Vert D_{y}^{s}u\left(1\right)\right\Vert _{L_{x,y}^{2}}^{2}\left(1+t\right)^{2D^{2}\epsilon^{3}}\lesssim\left\Vert D_{y}^{s}u\left(0\right)\right\Vert _{L_{x,y}^{2}}^{2}\left(1+t\right)^{2D^{2}\epsilon^{3}}.\label{eq:Dybound}
\end{equation}

\subsection{Proof of Theorem \ref{Theoremmain}}

We want to show that under the bootstrap assumptions (\ref{eq:bootsass2})
and (\ref{eq:bootsass1}) in time interval $[0,T]$ with $T$ arbitrary,
$u$ satisfies a better energy bound and a better decay bound after
a more careful computation. By (\ref{eq:energybd}) and (\ref{eq:Dybound}),
we know that when $u_{0}$ satisfies (\ref{eq:ini}), the solution
$u$ satisfies
\begin{equation}
\left\Vert u\left(t\right)\right\Vert _{X^{+}}\lesssim\epsilon\left(1+t\right)^{D^{3}\epsilon^{2}}.\label{eq:X+bound}
\end{equation}
The inequality holds with an implicit constant which does not depend
on $D$ and $T$.

Hence we turn to the global decay of $u$. From (\ref{eq:diffalpha}),
(\ref{eq:X+bound}) and $s=3\alpha>1$
\[
\left\Vert u(t)-\frac{1}{\sqrt{t}}e^{-i\frac{x^{2}}{2t}}\gamma\left(t,\frac{x}{t},y\right)\right\Vert _{L_{x}^{\infty}H_{y}^{1}}\lesssim t^{-\frac{7}{12}}\left\Vert u\right\Vert _{X^{+}}\lesssim\epsilon t^{-\frac{7}{12}+D^{3}\epsilon^{2}}.
\]
 There is also the Sobolev embedding $H_{y}^{1}\subset L_{y}^{4}$.
Assuming that $u$ satisfies (\ref{eq:ini}), by local well-posedness
we have
\[
\left\Vert \nabla_{y}\gamma(1)\right\Vert _{L_{v}^{\infty}L_{y}^{2}}\lesssim\left\Vert \nabla_{y}u(1)\right\Vert _{L_{x}^{\infty}L_{y}^{2}}\leq\epsilon,\quad\left\Vert \gamma(1)\right\Vert _{L_{v}^{\infty}H_{y}^{\alpha}}\lesssim\left\Vert u(1)\right\Vert _{L_{x}^{\infty}H_{y}^{\alpha}}\leq\epsilon.
\]
By (\ref{eq:gammaestimate}) we have that 
\begin{align*}
 & \left\Vert \nabla_{y}\gamma(T)\right\Vert _{L_{v}^{\infty}L_{y}^{2}}^{2}+\frac{1}{t}\left\Vert \gamma(T)\right\Vert _{L_{v}^{\infty}L_{y}^{4}}^{4}\\
\leq & \epsilon^{2}+\epsilon^{4}+\int_{1}^{T}D\left[D^{\frac{10}{3}}\epsilon^{4}\left(1+t\right)^{-2+\frac{10}{3}D^{3}\epsilon^{2}}+D^{\frac{7}{2}}\epsilon^{4}(1+t)^{-\frac{13}{12}+\frac{7}{2}D^{3}\epsilon^{2}}\right.\\
 & \quad\quad+\left.D^{\frac{31}{6}}\epsilon^{6}(1+t)^{-\frac{25}{12}+\frac{31}{6}D^{3}\epsilon^{2}}+D^{\frac{16}{3}}\epsilon^{6}(1+t)^{-\frac{13}{6}+\frac{16}{3}D^{3}\epsilon^{2}}\right]dt\\
\leq & \epsilon^{2}+\epsilon^{4}+D\left[D^{\frac{10}{3}}\epsilon^{4}+D^{\frac{7}{2}}\epsilon^{4}+D^{\frac{31}{6}}\epsilon^{6}+D^{\frac{16}{3}}\epsilon^{6}\right]\lesssim\epsilon^{2}.
\end{align*}
If we have $\epsilon\ll D^{-\frac{3}{2}}$ as before, we will have
\begin{equation}
\left\Vert u(t)\right\Vert _{L_{x}^{\infty}H_{y}^{1}}\leq t^{-\frac{1}{2}}\left\Vert \gamma(t)\right\Vert _{L_{v}^{\infty}H_{y}^{1}}+\epsilon t^{-\frac{7}{12}+D^{3}\epsilon^{2}}\lesssim\epsilon t^{-\frac{1}{2}}.\label{eq:udecaybd}
\end{equation}
The bounds (\ref{eq:udecaybd}), (\ref{eq:X+bound}) are better than
the original bootstrap assumptions (\ref{eq:bootsass1}), (\ref{eq:bootsass2}).
Thus we have closed the bootstrap argument, and the time interval
of local well-posedness $t\in[1,T]$ can be extended to $t\in[1,\infty).$ 

We go back to the estimates of asymptotic profile of $u$. By (\ref{eq:diffalpha}),
(\ref{eq:diffL2}), and (\ref{eq:energybd}),
\[
\left\Vert u\left(t,x,y\right)-\frac{1}{\sqrt{t}}e^{-i\frac{x^{2}}{2t}}\gamma\left(t,\frac{x}{t},y\right)\right\Vert _{L_{x,y}^{2}}\lesssim\epsilon t^{-\frac{1}{2}+D^{3}\epsilon^{2}},
\]

\[
\left\Vert u\left(t,x,y\right)-\frac{1}{\sqrt{t}}e^{-i\frac{x^{2}}{2t}}\gamma\left(t,\frac{x}{t},y\right)\right\Vert _{L_{x}^{\infty}H_{y}^{\alpha}}\lesssim\epsilon t^{-\frac{7}{12}+D^{3}\epsilon^{2}}.
\]
 Combining Lemma \ref{uasymptote} and (\ref{eq:energybd}), we obtain
\begin{equation}
\left\Vert I\right\Vert _{L_{v,y}^{2}}\lesssim\epsilon t^{-\frac{3}{2}+D^{3}\epsilon^{2}},\quad\left\Vert I\right\Vert _{L_{v}^{\infty}H_{y}^{\alpha}}\lesssim\epsilon t^{-\frac{13}{12}+D^{3}\epsilon^{2}}.\label{eq:Ierresti}
\end{equation}
The local well-posedness property of the equation 
\begin{equation}
\left(i\partial_{t}+\frac{1}{2}\triangle_{y}\right)\gamma(t,v,y)=\frac{1}{t}|\gamma|^{2}\gamma(t,v,y)+I(t,v,y)\label{eq:gammaerr}
\end{equation}
can be obtained by treating the $I(t,v,y)$ term as a perturbation.

\begin{lemma}\label{lemmagammaasym}If $\gamma$ satisfies the equation
(\ref{eq:gammaerr}) with the bounds (\ref{eq:Ierresti}), there exists
a solution $W(t)$ to (\ref{eq:eqtorus}) with 
\[
\gamma(t,v,y)=W(t,v,y)+\mathcal{O}_{L_{v,y}^{2}}\left(\epsilon t^{-\frac{1}{2}+2D^{3}\epsilon^{2}}\right),
\]
and
\[
\gamma(t,v,y)=W(t,v,y)+\mathcal{O}_{L_{v}^{\infty}H_{y}^{\alpha}}\left(\epsilon t^{-\frac{1}{12}+2D^{3}\epsilon^{2}}\right).
\]

\end{lemma}

\begin{proof} Let $W_{n}$ be a solution to the homogeneous equation
(\ref{eq:eqtorus}) with initial data $W_{n}\left(2^{n},v,y\right)=\gamma\left(2^{n},v,y\right)$,
where $n\in\mathbb{N}$, $n\geq1$. By (\ref{eq:mass}), we have the
conservation law $\left\Vert W_{n}(t)\right\Vert _{L_{v,y}^{2}}=\left\Vert \gamma\left(2^{n}\right)\right\Vert _{L_{v,y}^{2}}\lesssim\epsilon$.
By Lemma \ref{lemmaWreg}, there is the conservation law $\left\Vert W_{n}(t)\right\Vert _{L_{v}^{\infty}H_{y}^{1}}\lesssim\left\Vert W_{n}\left(2^{n}\right)\right\Vert _{L_{v}^{\infty}H_{y}^{1}}=\left\Vert \gamma\left(2^{n}\right)\right\Vert _{L_{v}^{\infty}H_{y}^{1}}\lesssim\epsilon$,
and the bound $\left\Vert W_{n}(2^{n})\right\Vert _{L_{v}^{\infty}H_{y}^{\alpha}}\lesssim\epsilon2^{\frac{5}{6}nD^{3}\epsilon^{2}}.$
The local and global well-posedness of the homogeneous equation (\ref{eq:eqtorus})
is given in Proposition \ref{prop:wproperty}, hence $W_{n}$ exists
on $[1,\infty)$ for each $n$. By (\ref{eq:eqtorus}) and (\ref{eq:gammaerr}),
we have
\begin{align*}
\left(i\partial_{t}+\frac{1}{2}\triangle_{y}\right)\left(\gamma-W_{n}\right) & =\frac{1}{t}\left(|\gamma|^{2}\gamma-\left|W_{n}\right|^{2}W_{n}\right)+I\\
 & =\frac{1}{t}\left(R\left[\gamma,\gamma,\gamma\right]-R\left[W_{n},W_{n},W_{n}\right]\right)+\frac{1}{t}\left(\mathcal{E}\left[\gamma,\gamma,\gamma\right]-\mathcal{E}\left[W_{n},W_{n},W_{n}\right]\right)+I.
\end{align*}
Therefore we have the bounds
\begin{align*}
\left\Vert \left(\gamma-W_{n}\right)(t)\right\Vert _{L_{v}^{\infty}H_{y}^{\alpha}} & \lesssim\int_{2^{n}}^{t}\frac{1}{\sigma}\left(\left\Vert \gamma\right\Vert _{L_{v}^{\infty}H_{y}^{1}}^{2}+\left\Vert W_{n}\right\Vert _{L_{v}^{\infty}H_{y}^{1}}^{2}\right)\left\Vert \left(\gamma-W_{n}\right)(\sigma)\right\Vert _{L_{v}^{\infty}H_{y}^{\alpha}}d\sigma+\int_{2^{n}}^{t}\left\Vert I(\sigma)\right\Vert _{L_{v}^{\infty}H_{y}^{\alpha}}d\sigma\\
 & \quad+\left\Vert \int_{2^{n}}^{t}\frac{1}{\sigma}\mathcal{E}\left[\gamma,\gamma,\gamma\right]d\sigma\right\Vert _{L_{v}^{\infty}H_{y}^{\alpha}}+\left\Vert \int_{2^{n}}^{t}\frac{1}{\sigma}\mathcal{E}\left[W_{n},W_{n},W_{n}\right]d\sigma\right\Vert _{L_{v}^{\infty}H_{y}^{\alpha}}.
\end{align*}
By (\ref{eq:Ierresti}), (\ref{eq:nonres}), and (\ref{eq:wbound})
,we have the bounds
\[
\int_{2^{n}}^{2^{n+1}}\left\Vert I(\sigma)\right\Vert _{L_{v}^{\infty}H_{y}^{\alpha}}d\sigma\lesssim\epsilon2^{n\left(-\frac{1}{12}+D^{3}\epsilon^{2}\right)},
\]
\begin{align*}
\sup_{2^{n}\leq t\leq2^{n+1}}\left\Vert \int_{2^{n}}^{t}\frac{1}{\sigma}\mathcal{E}\left[\gamma,\gamma,\gamma\right]d\sigma\right\Vert _{L_{v}^{\infty}H_{y}^{\alpha}}\lesssim\epsilon^{3}2^{-n+\frac{5}{2}D^{3}\epsilon^{2}},\\
\sup_{2^{n}\leq t\leq2^{n+1}}\left\Vert \int_{2^{n}}^{t}\frac{1}{\sigma}\mathcal{E}\left[W_{n},W_{n},W_{n}\right]d\sigma\right\Vert _{L_{v}^{\infty}H_{y}^{\alpha}}\lesssim\epsilon^{3}2^{-n+\frac{5}{2}D^{3}\epsilon^{2}}.
\end{align*}
Therefore by Gronwall's inequality, at $t=2^{n+1}$ we obtain 
\[
\left\Vert \left(\gamma-W_{n}\right)(2^{n+1})\right\Vert _{L_{v}^{\infty}H_{y}^{\alpha}}\lesssim\left[\epsilon2^{n\left(-\frac{1}{12}+D^{3}\epsilon^{2}\right)}+\epsilon^{3}2^{-n+\frac{5}{2}D^{3}\epsilon^{2}}\right]2^{nD\epsilon^{2}},
\]
\begin{align*}
 & \left\Vert W_{n}(2^{n+1})-W_{n+1}(2^{n+1})\right\Vert _{L_{v}^{\infty}H_{y}^{\alpha}}\\
\lesssim & \left\Vert \left(\gamma-W_{n}\right)(2^{n+1})\right\Vert _{L_{v}^{\infty}H_{y}^{\alpha}}+\left\Vert \left(\gamma-W_{n+1}\right)(2^{n+1})\right\Vert _{L_{v}^{\infty}H_{y}^{\alpha}}\lesssim\epsilon2^{n\left(-\frac{1}{12}+D\epsilon^{2}+D^{3}\epsilon^{2}\right)}.
\end{align*}
Solving the equation for the difference $W_{n}-W_{n+1}$ backward
from $t=2^{n+1}$,
\begin{align*}
\left(i\partial_{t}+\frac{1}{2}\triangle_{y}\right)\left(W_{n}-W_{n+1}\right) & =\frac{1}{t}\left(R\left[W_{n},W_{n},W_{n}\right]-R\left[W_{n+1},W_{n+1},W_{n+1}\right]\right)\\
 & \quad+\frac{1}{t}\left(\mathcal{E}\left[W_{n},W_{n},W_{n}\right]-\mathcal{E}\left[W_{n+1},W_{n+1},W_{n+1}\right]\right).
\end{align*}
 By Gronwall's inequality we obtain
\[
\left\Vert W_{n}(1)-W_{n+1}(1)\right\Vert _{L_{v}^{\infty}H_{y}^{\alpha}}\lesssim\epsilon2^{n\left(-\frac{1}{12}+D\epsilon^{2}+D^{3}\epsilon^{2}\right)}2^{(n+1)D\epsilon^{2}}=\epsilon2^{n\left(-\frac{1}{12}+2D\epsilon^{2}+D^{3}\epsilon^{2}\right)+D\epsilon^{2}}.
\]
Here we assume that $\epsilon\ll D^{-\frac{3}{2}}$, hence $D\epsilon^{2}<\frac{1}{2}D^{3}\epsilon^{2}$.
Therefore there exists $W_{\infty}\left(1\right)$ such that 
\[
\lim_{n\rightarrow\infty}\left\Vert W_{n}(1)-W_{\infty}(1)\right\Vert _{L_{v}^{\infty}H_{y}^{\alpha}}\lesssim\lim_{n\rightarrow\infty}\epsilon2^{n\left(-\frac{1}{12}+\frac{3}{2}D^{3}\epsilon^{2}\right)}.
\]
We obtain a solution $W_{\infty}(t)$ to the equation (\ref{eq:eqtorus})
with initial data $W_{\infty}(1)$ at $t=1$. Then $W_{\infty}$has
the desired property. By Lemma \ref{lemmaWreg},
\[
\left\Vert W_{n}(t)-W_{\infty}(t)\right\Vert _{L_{v}^{\infty}H_{y}^{\alpha}}\lesssim t^{D\epsilon^{2}}\left\Vert W_{n}(1)-W_{\infty}(t)\right\Vert _{L_{v}^{\infty}H_{y}^{\alpha}}\lesssim\epsilon2^{n\left(-\frac{1}{12}+2D^{3}\epsilon^{2}\right)}t^{D\epsilon^{2}}.
\]
For $t\in(2^{n},2^{n+1})$, we have 
\[
\begin{aligned}\left\Vert \gamma(t)-W_{\infty}(t)\right\Vert _{L_{v}^{\infty}H_{y}^{\alpha}} & \leq\left\Vert W_{n}(t)-W_{\infty}(t)\right\Vert _{L_{v}^{\infty}H_{y}^{\alpha}}+\left\Vert \gamma(t)-W_{n}(t)\right\Vert _{L_{v}^{\infty}H_{y}^{\alpha}}\\
 & \lesssim\epsilon2^{n\left(-\frac{1}{12}+\frac{3}{2}D^{3}\epsilon^{2}\right)}t^{D\epsilon^{2}}+\epsilon2^{n\left(-\frac{1}{12}+D^{3}\epsilon^{2}\right)}\left(t-2^{n}\right)^{D\epsilon^{2}}\lesssim\epsilon t^{\left(-\frac{1}{12}+2D^{3}\epsilon^{2}\right)}.
\end{aligned}
\]
The computations for the $L_{v,y}^{2}$ norm follows a similar steps
as the $L_{v}^{\infty}H_{y}^{\alpha}$ norm; therefore we obtain
\[
\left\Vert \gamma(t)-W_{\infty}(t)\right\Vert _{L_{v,y}^{2}}\lesssim\epsilon2^{n\left(-\frac{1}{2}+2D^{3}\epsilon^{2}\right)}.
\]

\end{proof}

Together with (\ref{eq:udecaybd}), the proof of Theorem \ref{Theoremmain}
is complete.

\section{The Cubic NLS on Torus\label{sec:The-Cubic-NLS}}

The main results of this section are to develop the same resonant
equation as in \cite{HPTV}, and prove some important properties for
the solution of (\ref{eq:eqtorus}), such as local (global) well-posedness,
asymptotic dynamics, asymptotic completeness and energy bounds. 

The asymptotic dynamics of small solutions to (\ref{eq main}) is
related to solutions of the resonant system: 
\begin{equation}
\begin{aligned}i\partial_{t}\mathcal{G}\left(t,v,y\right) & =\frac{1}{t}R\left[\mathcal{G}\left(t,v,y\right),\mathcal{G}\left(t,v,y\right),\mathcal{G}\left(t,v,y\right)\right].\end{aligned}
\label{eq dyna}
\end{equation}
Note that this system is none other than the resonant system for the
cubic NLS equation on $\mathbb{T}^{d}$ up to a change of timescale.
Here we introduce the global well-posedness of (\ref{eq dyna}), which
will be used in the proof of asymptotic dynamics and asymptotic completeness
for (\ref{eq:eqtorus}).
\begin{prop}
\label{prop:Gprop}(Properties of $\mathcal{G}$) (a) Let $d=1,2,3,4$.
For any $\mathcal{G}(1)\in L_{v}^{\infty}H_{y}^{\alpha}\cap L_{v,y}^{2}$,
there exists a unique global solution $\mathcal{G}(t)\in C^{1}\left(\left[1,\infty\right),L_{v}^{\infty}H_{y}^{\alpha}\cap L_{v,y}^{2}\right)$
for any $\alpha\geq1$.

(b) There are the conservation laws $\left\Vert \mathcal{G}(t)\right\Vert _{L_{v,y}^{2}}\equiv\left\Vert \mathcal{G}(1)\right\Vert _{L_{v,y}^{2}},\quad\left\Vert \mathcal{G}(t)\right\Vert _{L_{v}^{\infty}\dot{H}_{y}^{1}}\equiv\left\Vert \mathcal{G}(1)\right\Vert _{L_{v}^{\infty}\dot{H}_{y}^{1}}.$

(c) (Infinite cascade) If $d\geq2$ and $\alpha>1$, there exist global
solutions $\mathcal{G}(t)\in C^{1}\left(\left[1,\infty\right),L_{v}^{\infty}H_{y}^{\alpha}\cap L_{v,y}^{2}\right)$
such that 
\[
\sup_{t>1}\left\Vert \mathcal{G}(t)\right\Vert _{L_{v}^{\infty}H_{y}^{\alpha}}=\infty.
\]
\end{prop}
The proof is given in \cite{HPTV}, Section 4. See \cite{CKGTT,H1,GK}
for the energy cascade property, which is stated here for completeness,
but not used in any way.
\begin{prop}
\label{prop:wproperty}(a) The equation (\ref{eq:eqtorus}) is locally
well-posed for initial data $W_{1}\in Y$ ($\alpha>\frac{d}{2}$,
$d=1,2,3,4$). Precisely, for each such data there is a unique solution
$W\in C\left([1,b];Y\right)$ for some $b>1$ with $W(1)=W_{1}$,
and Lipschitz continuous dependence on initial data locally on time.

(b) There exists $\epsilon=\epsilon(d)>0$ such that if in addition
\begin{equation}
\left\Vert W_{1}\right\Vert _{Y}\leq\epsilon,\label{eq:iniW}
\end{equation}
then the equation (\ref{eq:eqtorus}) is globally well-posed in $[1,\infty)$,
and there exists a $\mathcal{G}\in Y$ solution to (\ref{eq dyna})
such that
\begin{equation}
\left\Vert W(t,v,y)-e^{it\triangle_{y}/2}\mathcal{G}(t,v,y)\right\Vert _{Y}\rightarrow0\quad\text{as }t\rightarrow\infty.\label{eq:Wasyerr}
\end{equation}

(c) If $\mathcal{G}_{1}$ satisfies $\left\Vert \mathcal{G}_{1}\right\Vert _{Y}\ll\epsilon<1$,
then there exists a global solution $\mathcal{G}\in C([1,\infty);Y)$
of (\ref{eq dyna}) with initial data $\mathcal{G}(1)=\mathcal{G}_{1}$
and $\mathcal{W}\in C([1,\infty);Y)$ which satisfies (\ref{eq:eqtorus})
with $\left\Vert \mathcal{W}(1)\right\Vert _{Y}\leq\epsilon$ such
that (\ref{eq:Wasyerr}) holds.
\end{prop}
\begin{remark}Here the variable $v$ is merely a parameter. We can
prove the well-posedness and asymptotic profile in the space $Y=H_{y}^{s}$
for the function $W(t,y)$ satisfying (\ref{eq:eqtorus}). The proof
follows in a similar manner to the proof of Proposition \ref{prop:wproperty}. 

\end{remark}

\begin{proof}[Proof of (a):] The equation can also be solved in critical
functional spaces. See \cite{HTT1,HTT2,IP}. Consider a short interval
$I=[a,b]\subset\mathbb{R}^{+}$, and let $f,g\in Y$. There is the
following iteration scheme:
\begin{align*}
 & \left\Vert \int_{0}^{t}s^{-1}e^{i(t-s)\triangle_{y}/2}\left(\left|f\right|^{2}f-\left|g\right|^{2}g\right)\left(s\right)ds\right\Vert _{L_{t}^{\infty}\left(I;L_{v}^{\infty}H_{y}^{\alpha}\right)}\\
\lesssim & \left\Vert t^{-1}\left(\left\Vert f\right\Vert _{L_{v}^{\infty}H_{y}^{\alpha}}^{2}+\left\Vert g\right\Vert _{L_{v}^{\infty}H_{y}^{\alpha}}^{2}\right)\left\Vert f-g\right\Vert _{L_{v}^{\infty}H_{y}^{\alpha}}\right\Vert _{L_{t}^{\infty}\left(I\right)}\\
\lesssim & \log\left(\frac{b}{a}\right)\left(\left\Vert f\right\Vert _{L_{t}^{\infty}\left(I;L_{v}^{\infty}H_{y}^{\alpha}\right)}^{2}+\left\Vert g\right\Vert _{L_{t}^{\infty}\left(I;L_{v}^{\infty}H_{y}^{\alpha}\right)}^{2}\right)\left\Vert f-g\right\Vert _{L_{t}^{\infty}\left(I;L_{v}^{\infty}H_{y}^{\alpha}\right),}
\end{align*}

\begin{align*}
 & \left\Vert \int_{0}^{t}s^{-1}e^{i(t-s)\triangle_{y}/2}\left(\left|f\right|^{2}f-\left|g\right|^{2}g\right)\left(s\right)ds\right\Vert _{L_{t}^{\infty}\left(I;L_{v,y}^{2}\right)}\\
\lesssim & \left\Vert t^{-1}\left(\left\Vert f\right\Vert _{L_{v}^{\infty}H_{y}^{\alpha}}^{2}+\left\Vert g\right\Vert _{L_{v}^{\infty}H_{y}^{\alpha}}^{2}\right)\left\Vert f-g\right\Vert _{L_{v,y}^{2}}\right\Vert _{L_{t}^{\infty}\left(I\right)}\\
\lesssim & \log\left(\frac{b}{a}\right)\left(\left\Vert f\right\Vert _{L_{t}^{\infty}\left(I;L_{v}^{\infty}H_{y}^{\alpha}\right)}^{2}+\left\Vert g\right\Vert _{L_{t}^{\infty}\left(I;L_{v}^{\infty}H_{y}^{\alpha}\right)}^{2}\right)\left\Vert f-g\right\Vert _{L_{t}^{\infty}\left(I;L_{v,y}^{2}\right).}
\end{align*}
If we take $|I|$ to be small enough, by the contraction principle
the local well-posedness is obtained in $Y$.

\end{proof}

\begin{proof}[Proof of (b):]

The solution of (\ref{eq:eqtorus}) with initial condition (\ref{eq:iniW})
will have the conservation laws and inequality for $t\geq1$:
\[
\left\Vert W(t)\right\Vert _{L_{v,y}^{2}}=\left\Vert W_{1}\right\Vert _{L_{v,y}^{2}},\quad\left\Vert W(t)\right\Vert _{L_{v}^{\infty}\dot{H}_{y}^{1}}^{2}+\frac{1}{t}\left\Vert W(t)\right\Vert _{L_{v}^{\infty}L_{y}^{4}}^{4}\leq\left\Vert W(1)\right\Vert _{L_{v}^{\infty}\dot{H}_{y}^{1}}^{2}+\left\Vert W(1)\right\Vert _{L_{v}^{\infty}L_{y}^{4}}^{4}.
\]
From local well-posedness, assume that the solution $W(t)$ exists
on the time interval $[1,T]$. Here we make an extra bootstrap assumption
\begin{equation}
\left\Vert W(t)\right\Vert _{L_{v}^{\infty}H_{y}^{\alpha}}\lesssim D\epsilon\left(1+t\right)^{\delta}\label{eq:wbootsassum}
\end{equation}
where $\delta$ is some small positive number. Denote $\mathcal{W}:=e^{-it\triangle_{y}/2}W$,
we have 
\[
\left\Vert \mathcal{W}(1)\right\Vert _{L_{v}^{\infty}H_{y}^{\alpha}}=\left\Vert W(1)\right\Vert _{L_{v}^{\infty}H_{y}^{\alpha}}\leq\epsilon
\]
and 
\[
\left\Vert \mathcal{W}(t)\right\Vert _{L_{v}^{\infty}H_{y}^{1}}=\left\Vert W(t)\right\Vert _{L_{v}^{\infty}H_{y}^{1}}\leq\epsilon
\]
for any $t\in[1,T]$. The equation of $\mathcal{W}$ is 
\[
i\partial_{t}\mathcal{W}(t,v,y)=\sum_{\omega}\sum_{\Gamma_{\omega}(\mathbf{k})}e^{i\omega t/2}\mathcal{W}\left(t,v,\mathbf{k}_{1}\right)\overline{\mathcal{W}}\left(t,v,\mathbf{k}_{2}\right)\mathcal{W}\left(t,v,\mathbf{k}_{3}\right)e^{i\mathbf{k}\cdot y}.
\]
To obtain the resonant equation, separate the nonlinear term of $\partial_{t}\mathcal{W}$
into the resonant term and nonresonant term, 

\begin{align*}
i\partial_{t}\mathcal{W}\left(t,v,y\right) & =\frac{1}{t}R\left[\mathcal{W},\mathcal{W},\mathcal{W}\right](t,v,y)+\frac{1}{t}\mathcal{E}\left[\mathcal{W},\overline{\mathcal{W}},\mathcal{W}\right](t,v,y).
\end{align*}
Using the definition of the non-resonant trilinear form $\mathcal{D}$
and we will have 
\begin{equation}
\begin{aligned} & \frac{1}{t}\mathcal{E}\left[\mathcal{W},\overline{\mathcal{W}},\mathcal{W}\right](t,v,y)\\
= & \partial_{t}\mathcal{D}\left[\mathcal{W},\mathcal{W},\mathcal{W}\right]+t^{-1}\mathcal{D}\left[\mathcal{W},\mathcal{W},\mathcal{W}\right]-\mathcal{D}\left[\mathcal{W}_{t},\mathcal{W},\mathcal{W}\right]-\mathcal{D}\left[\mathcal{W},\mathcal{W}_{t},\mathcal{W}\right]-\mathcal{D}\left[\mathcal{W},\mathcal{W},\mathcal{W}_{t}\right].
\end{aligned}
\label{eq:nonres}
\end{equation}
It is easy to verify that
\begin{equation}
\begin{aligned}\left\Vert \mathcal{D}\left[\mathcal{W},\mathcal{W},\mathcal{W}\right](t)\right\Vert _{L_{v}^{\infty}H_{y}^{\alpha}} & \lesssim t^{-1}\left\Vert \mathcal{W}\left(t\right)\right\Vert _{L_{v}^{\infty}H_{y}^{\alpha}}^{3},\\
\left\Vert \partial_{t}\mathcal{W}\right\Vert _{L_{v}^{\infty}H_{y}^{\alpha}} & \lesssim t^{-1}\left\Vert \mathcal{W}\left(t\right)\right\Vert _{L_{v}^{\infty}H_{y}^{\alpha}}^{3}.
\end{aligned}
\label{eq:wbound}
\end{equation}
By (\ref{eq:wbound}), (\ref{eq:wbootsassum}), we have the bound
\begin{align*}
 & \left\Vert \int_{1}^{T}t^{-1}\mathcal{E}\left[\mathcal{W},\overline{\mathcal{W}},\mathcal{W}\right](t,v,y)dt\right\Vert _{L_{v}^{\infty}H_{y}^{\alpha}}\\
\lesssim & T^{-1}\left\Vert \mathcal{W}\left(T\right)\right\Vert _{L_{v}^{\infty}H_{y}^{\alpha}}^{3}+\left\Vert \mathcal{W}\left(1\right)\right\Vert _{L_{v}^{\infty}H_{y}^{\alpha}}^{3}+\int_{1}^{T}t^{-2}\left(\left\Vert \mathcal{W}\left(t\right)\right\Vert _{L_{v}^{\infty}H_{y}^{\alpha}}^{3}+\left\Vert \mathcal{W}\left(t\right)\right\Vert _{L_{v}^{\infty}H_{y}^{\alpha}}^{6}\right)dt\\
\lesssim & D^{3}\epsilon^{3}T^{-1}\left(1+T\right)^{3\delta}+\epsilon^{3}+\int_{1}^{T}t^{-2}\left(D^{3}\epsilon^{3}\left(1+t\right)^{3\delta}+D^{6}\epsilon^{6}(1+t)^{6\delta}\right)dt.
\end{align*}
By ($\ref{eq:h1esti}$), we also have 
\[
\left\Vert \mathcal{R}\left[\mathcal{W},\mathcal{W},\mathcal{W}\right](t)\right\Vert _{L_{v}^{\infty}H_{y}^{\alpha}}\lesssim\left\Vert \mathcal{W}\left(t\right)\right\Vert _{L_{v}^{\infty}H_{y}^{1}}^{2}\left\Vert \mathcal{W}\left(t\right)\right\Vert _{L_{v}^{\infty}H_{y}^{\alpha}}\lesssim\epsilon^{2}\left\Vert \mathcal{W}\left(t\right)\right\Vert _{L_{v}^{\infty}H_{y}^{\alpha}}.
\]
Combining the above estimates together, 
\begin{align*}
\left\Vert \mathcal{W}(T)\right\Vert _{L_{v}^{\infty}H_{y}^{\alpha}} & \leq\epsilon+D^{4}\epsilon^{3}T^{-1}\left(1+T\right)^{3\delta}+D\epsilon^{3}\\
 & \quad+\int_{1}^{T}D^{2}\epsilon^{3}t^{-1}+t^{-2}D\left(D^{3}\epsilon^{3}\left(1+t\right)^{3\delta}+D^{6}\epsilon^{6}(1+t)^{6\delta}\right)dt.
\end{align*}
Hence if $\epsilon\ll D^{-\frac{3}{2}}$and $D^{2}\epsilon^{3}<\delta$,
we have 
\begin{equation}
\left\Vert W(t)\right\Vert _{L_{v}^{\infty}H_{y}^{\alpha}}=\left\Vert \mathcal{W}(t)\right\Vert _{L_{v}^{\infty}H_{y}^{\alpha}}\lesssim\left(\epsilon+D\epsilon^{3}+D^{4}\epsilon^{3}+D^{7}\epsilon^{6}\right)t^{D^{2}\epsilon^{3}}\lesssim\epsilon t^{D^{2}\epsilon^{3}},\label{eq:walphabd}
\end{equation}
which is a better bound. Due to the boundness of $\left\Vert W(t)\right\Vert _{L_{v}^{\infty}H_{y}^{\alpha}}$,
we can extend the local well-posedness result to well-posedness on
$[1,\infty)$.

To obtain the asymptotic equation, define the function 
\[
F(t,v,y):=-i\int_{t}^{\infty}\sum_{\omega,\omega\neq0}\sum_{\Gamma_{\omega}(\mathbf{k})}e^{i\omega\sigma/2}\mathcal{W}\left(\sigma,v,\mathbf{k}_{1}\right)\overline{\mathcal{W}}\left(\sigma,v,\mathbf{k}_{2}\right)\mathcal{W}\left(\sigma,v,\mathbf{k}_{3}\right)e^{i\mathbf{k}\cdot y}d\sigma.
\]
For this function we have the bounds:

\begin{lemma} For $t\geq1$, we have 
\[
\left\Vert F(t,v,y)\right\Vert _{L_{v}^{\infty}H_{y}^{\alpha}}\lesssim\epsilon^{6}t^{-1+6D^{2}\epsilon^{3}},\quad\left\Vert F(t,v,y)\right\Vert _{L_{v,y}^{2}}\lesssim\epsilon^{7}t^{-1+5D^{2}\epsilon^{3}}.
\]

\end{lemma}

\begin{proof}By (\ref{eq:walphabd}), (\ref{eq:wbound}) and integration
by parts in time $t$, 

\begin{align*}
\left\Vert F(t,v,y)\right\Vert _{L_{v}^{\infty}H_{y}^{\alpha}} & \lesssim t^{-1}\left\Vert \mathcal{W}\left(t\right)\right\Vert _{L_{v}^{\infty}H_{y}^{\alpha}}^{3}+\int_{t}^{\infty}\sigma^{-2}\left(\left\Vert \mathcal{W}\left(\sigma\right)\right\Vert _{L_{v}^{\infty}H_{y}^{\alpha}}^{3}+\left\Vert \mathcal{W}\left(\sigma\right)\right\Vert _{L_{v}^{\infty}H_{y}^{\alpha}}^{6}\right)d\sigma\\
 & \lesssim\epsilon^{3}t^{-1+3D^{2}\epsilon^{3}}+\epsilon^{6}t^{-1+6D^{2}\epsilon^{3}}\lesssim\epsilon^{6}t^{-1+6D^{2}\epsilon^{3}},
\end{align*}
\begin{align*}
\left\Vert F(t,v,y)\right\Vert _{L_{v,y}^{2}} & \lesssim t^{-1}\left\Vert \mathcal{W}\left(t\right)\right\Vert _{L_{v}^{\infty}H_{y}^{\alpha}}^{2}\left\Vert \mathcal{W}\left(t\right)\right\Vert _{L_{v,y}^{2}}\\
 & \quad+\int_{t}^{\infty}\sigma^{-2}\left(\left\Vert \mathcal{W}\left(\sigma\right)\right\Vert _{L_{v}^{\infty}H_{y}^{\alpha}}^{2}+\left\Vert \mathcal{W}\left(\sigma\right)\right\Vert _{L_{v}^{\infty}H_{y}^{\alpha}}^{5}\right)\left\Vert \mathcal{W}\left(\sigma\right)\right\Vert _{L_{v,y}^{2}}d\sigma\\
 & \lesssim\epsilon^{4}t^{-1+2D^{2}\epsilon^{3}}+\epsilon^{7}t^{-1+5D^{2}\epsilon^{3}}\lesssim\epsilon^{7}t^{-1+5D^{2}\epsilon^{3}}.
\end{align*}
\end{proof}

Then we have the property 

\[
i\partial_{t}\left(\mathcal{W}\left(t,v,y\right)+F(t,v,y)\right)=\frac{1}{t}R\left[\mathcal{W}\left(t,v,y\right),\mathcal{W}\left(t,v,y\right),\mathcal{W}\left(t,v,y\right)\right].
\]
In order to rewrite this equation into the form of (\ref{eq dyna}),
define a modified function of $\mathcal{W}$ to be
\[
\tilde{\mathcal{W}}(t,v,y):=\mathcal{W}\left(t,v,y\right)+F(t,v,y).
\]
Thus the equation becomes
\begin{align*}
i\partial_{t}\tilde{\mathcal{W}}(t,v,y) & =\frac{1}{t}R\left[\tilde{\mathcal{W}}-F,\mathcal{\tilde{\mathcal{W}}}-F,\tilde{\mathcal{W}}-F\right]\\
 & =\frac{1}{t}R\left[\tilde{\mathcal{W}},\mathcal{\tilde{\mathcal{W}}},\tilde{\mathcal{W}}\right]-\frac{1}{t}R\left[F,\mathcal{\mathcal{W}},\mathcal{W}\right]-\frac{1}{t}R\left[\tilde{\mathcal{W}},F,\mathcal{W}\right]-\frac{1}{t}R\left[\tilde{\mathcal{W}},\tilde{\mathcal{W}},F\right].
\end{align*}

\begin{lemma}For $t>1$, 
\[
i\partial_{t}\tilde{\mathcal{W}}(t,v,y)=\frac{1}{t}R\left[\tilde{\mathcal{W}},\mathcal{\tilde{\mathcal{W}}},\tilde{\mathcal{W}}\right]+\mathcal{O}_{L_{v,y}^{2}}\left(\epsilon^{9}t^{-2+5D^{2}\epsilon^{3}}\right)\cap\mathcal{O}_{L_{v}^{\infty}H_{y}^{\alpha}}\left(\epsilon^{8}t^{-2+8D^{2}\epsilon^{3}}\right).
\]

\end{lemma}

\begin{proof}

Since $\left\Vert \tilde{\mathcal{W}}\right\Vert _{L_{v,y}^{2}}\lesssim\left\Vert \mathcal{W}\right\Vert _{L_{v,y}^{2}}$
and $\left\Vert \tilde{\mathcal{W}}\right\Vert _{L_{v}^{\infty}H_{y}^{\alpha}}\lesssim\left\Vert \mathcal{W}\right\Vert _{L_{v}^{\infty}H_{y}^{\alpha}}$,
and the estimate
\begin{align*}
\left\Vert t^{-1}R\left[F,\mathcal{\mathcal{W}},\mathcal{W}\right]\right\Vert _{L_{v,y}^{2}} & \lesssim t^{-1}\left\Vert \mathcal{W}\right\Vert _{L_{v}^{\infty}H_{y}^{1}}^{2}\left\Vert F\right\Vert _{L_{v,y}^{2}}\lesssim\epsilon^{9}t^{-2+5D^{2}\epsilon^{3}},\\
\left\Vert t^{-1}R\left[F,\mathcal{\mathcal{W}},\mathcal{W}\right]\right\Vert _{L_{v}^{\infty}H_{y}^{\alpha}} & \lesssim t^{-1}\left\Vert \mathcal{W}\right\Vert _{L_{v}^{\infty}H_{y}^{\alpha}}^{2}\left\Vert F\right\Vert _{L_{v}^{\infty}H_{y}^{\alpha}}\lesssim\epsilon^{8}t^{-2+8D^{2}\epsilon^{3}},
\end{align*}
the lemma is proved.

\end{proof}

Therefore $\tilde{\mathcal{W}}$ solves the resonant equation (\ref{eq dyna})
with a perturbative error. By Proposition \ref{prop:Gprop}, there
is a solution to the homogeneous equation (\ref{eq dyna}) for any
initial data in $Y$. Hence through a similar argument as in Lemma
\ref{lemmagammaasym}, there exists a solution $\mathcal{G}$ to equation
(\ref{eq dyna}) satisfying 
\[
\tilde{\mathcal{W}}(t,v,y)=\mathcal{G}\left(t,v,y\right)+\mathcal{O}_{L_{v,y}^{2}}\left(\epsilon^{9}t^{-1+5D^{2}\epsilon^{3}}\right)\cap\mathcal{O}_{L_{v}^{\infty}H_{y}^{\alpha}}\left(\epsilon^{8}t^{-1+8D^{2}\epsilon^{3}}\right).
\]
 Using the estimate of $F=\tilde{\mathcal{W}}-W$, we obtain 
\[
W(t,v,y)=e^{it\triangle_{y}/2}\mathcal{G}\left(t,v,y\right)+\mathcal{O}_{L_{v,y}^{2}}\left(\epsilon^{7}t^{-1+5D^{2}\epsilon^{3}}\right)\cap\mathcal{O}_{L_{v}^{\infty}H_{y}^{\alpha}}\left(\epsilon^{8}t^{-1+8D^{2}\epsilon^{3}}\right).
\]

\end{proof}

\begin{proof}[Proof of (c):]

Let $\mathcal{G}$ be a solution to (\ref{eq dyna}), but where we
put extra assumptions on the initial data:
\[
\left\Vert \mathcal{G}\left(1\right)\right\Vert _{L_{v,y}^{2}}+\left\Vert \mathcal{G}\left(1\right)\right\Vert _{L_{v}^{\infty}H_{y}^{\alpha}}\leq M,
\]
where $0<M\ll\delta$. By the conservation law of (\ref{eq dyna})
and (\ref{eq:h1esti}) we have the following inequalities
\[
\left\Vert \mathcal{G}\left(t\right)\right\Vert _{L_{v,y}^{2}},\left\Vert \mathcal{G}\left(t\right)\right\Vert _{L_{v}^{\infty}H_{y}^{1}}\leq M,\quad\left\Vert \mathcal{G}\right\Vert _{L_{v}^{\infty}H_{y}^{\alpha}}\lesssim Mt^{\delta}.
\]
 If there exists a solution $W$ to (\ref{eq:eqtorus}) tending to
$e^{it\triangle_{y}/2}\mathcal{G}$ at $t=\infty$, then the difference
$\mathcal{V}:=W-e^{it\triangle_{y}/2}\mathcal{G}$ satisfies the following
equation:
\begin{equation}
\begin{cases}
\left(i\partial_{t}+\frac{1}{2}\triangle_{y}\right)\mathcal{V}:=h_{1}+h_{2},\\
\mathcal{V}(\infty)=0.
\end{cases}\label{eq:bigVeq}
\end{equation}
The functions $h_{1}$, $h_{2}$ are given by
\begin{align*}
h_{1} & :=\frac{1}{t}e^{it\triangle_{y}/2}\left[\mathcal{R}\left[e^{-it\triangle_{y}/2}\mathcal{V}+\mathcal{G},e^{-it\triangle_{y}/2}\mathcal{V}+\mathcal{G},e^{-it\triangle_{y}/2}\mathcal{V}+\mathcal{G}\right]-\mathcal{R}\left[\mathcal{G},\mathcal{G},\mathcal{G}\right]\right],
\end{align*}
and denoting $\mathcal{W}:=e^{-it\triangle_{y}/2}\mathcal{V}+\mathcal{G}$,
\[
h_{2}:=\frac{e^{i\frac{\omega}{2}t+it\triangle_{y}/2}}{t}\sum_{\omega}\sum_{\Gamma_{\omega}(\mathbf{k})}\mathcal{W}\left(t,v,\mathbf{k}_{1}\right)\overline{\mathcal{W}}\left(t,v,\mathbf{k}_{2}\right)\mathcal{W}\left(t,v,\mathbf{k}_{3}\right)e^{i\mathbf{k}\cdot y}.
\]
Here our goal is to solve the $\mathcal{V}$ equation from $t=\infty$.
The solution $\mathcal{V}$ will satisfy the equation
\begin{equation}
\mathcal{V}(t)=i\int_{t}^{\infty}e^{i(t-s)\triangle_{y}/2}\left(h_{1}+h_{2}\right)(s)ds.
\end{equation}
Hence we define a function space with time decay, and solve $\mathcal{V}$
in the space 
\[
\left\Vert f\right\Vert _{\mathscr{Z}}:=\sup_{T>1}T^{\delta}\left(\left\Vert f\right\Vert _{L_{t}^{\infty}\left(T,2T;L_{v}^{\infty}H_{y}^{\alpha}\right)}+\left\Vert f\right\Vert _{L_{t}^{\infty}\left(T,2T;L_{v,y}^{2}\right)}\right).
\]
Since $\mathcal{R}$ is a trilinear form, 
\begin{align*}
\left\Vert h_{1}\right\Vert _{L_{v}^{\infty}H_{y}^{\alpha}\left(L_{v,y}^{2}\right)} & \lesssim t^{-1}\left(\left\Vert R\left[\mathcal{G},\mathcal{G},e^{-it\triangle_{y}/2}\mathcal{V}\right]\right\Vert _{L_{v}^{\infty}H_{y}^{\alpha}\left(L_{v,y}^{2}\right)}+\left\Vert R\left[\mathcal{G},e^{-it\triangle_{y}/2}\mathcal{V},e^{-it\triangle_{y}/2}\mathcal{V}\right]\right\Vert _{L_{v}^{\infty}H_{y}^{\alpha}\left(L_{v,y}^{2}\right)}\right.\\
 & \quad\left.+\left\Vert R\left[e^{-it\triangle_{y}/2}\mathcal{V},e^{-it\triangle_{y}/2}\mathcal{V},e^{-it\triangle_{y}/2}\mathcal{V}\right]\right\Vert _{L_{v}^{\infty}H_{y}^{\alpha}\left(L_{v,y}^{2}\right)}\right).
\end{align*}
By (\ref{eq:h1esti}) there are the following bounds: 
\[
\left\Vert t^{-1}R\left[\mathcal{G},\mathcal{G},e^{-it\triangle_{y}/2}\mathcal{V}\right]\right\Vert _{L_{t}^{1}\left(T,2T;L_{v}^{\infty}H_{y}^{\alpha}\right)+L_{t}^{1}\left(T,2T;L_{v,y}^{2}\right)}\lesssim M^{2}T^{-\delta}\left\Vert \mathcal{V}\right\Vert _{\mathscr{Z}},
\]
\begin{align*}
\left\Vert t^{-1}R\left[\mathcal{G},e^{-it\triangle_{y}/2}\mathcal{V},e^{-it\triangle_{y}/2}\mathcal{V}\right]\right\Vert _{L_{t}^{1}\left(T,2T;L_{v}^{\infty}H_{y}^{\alpha}\right)+L_{t}^{1}\left(T,2T;L_{v,y}^{2}\right)}\lesssim MT^{-2\delta}\left\Vert \mathcal{V}\right\Vert _{\mathscr{Z}}^{2},
\end{align*}
\[
\left\Vert t^{-1}R\left[e^{-it\triangle_{y}/2}\mathcal{V},e^{-it\triangle_{y}/2}\mathcal{V},e^{-it\triangle_{y}/2}\mathcal{V}\right]\right\Vert _{L_{t}^{1}\left(T,2T;L_{v}^{\infty}H_{y}^{\alpha}\right)+L_{t}^{1}\left(T,2T;L_{v,y}^{2}\right)}\lesssim T^{-3\delta}\left\Vert \mathcal{V}\right\Vert _{\mathscr{Z}}^{3}.
\]
We obtain the $\mathscr{Z}$ bound for $h_{1}$,
\begin{equation}
\left\Vert \int_{t}^{\infty}e^{-i(t-s)\triangle_{y}/2}h_{1}(s)ds\right\Vert _{\mathscr{Z}}\lesssim M^{2}\left\Vert \mathcal{V}\right\Vert _{\mathscr{Z}}+M\left\Vert \mathcal{V}\right\Vert _{\mathscr{Z}}^{2}+\left\Vert \mathcal{V}\right\Vert _{\mathscr{Z}}^{3}.\label{eq:h_1esti}
\end{equation}
For the $h_{2}$ part, use the integration by parts in time and break
the time interval into dyadic subintervals and estimate $e^{-i(t-s)\triangle_{y}/2}h_{1}(s)$
in each interval and sum up: 
\begin{align*}
\int_{T}^{2T}e^{-is\triangle_{y}/2}h_{2}(s)ds= & \left.\mathcal{D}\left[\mathcal{W},\mathcal{W},\mathcal{W}\right]\right|_{T}^{2T}+\int_{T}^{2T}s^{-1}\mathcal{D}\left[\mathcal{W},\mathcal{W},\mathcal{W}\right]ds\\
 & +\int_{T}^{2T}\mathcal{D}\left[\partial_{t}\mathcal{W},\mathcal{W},\mathcal{W}\right]+\mathcal{D}\left[\mathcal{W},\partial_{t}\mathcal{W},\mathcal{W}\right]+\mathcal{D}\left[\mathcal{W},\mathcal{W},\partial_{t}\mathcal{W}\right]ds.
\end{align*}
By the formula for $\mathcal{G}$ and $\mathcal{V}$, we have 
\[
i\partial_{t}\mathcal{W}=\frac{e^{-it\triangle_{y}/2}}{t}\left|\mathcal{W}\right|^{2}\mathcal{W},
\]
and 
\[
\left\Vert \mathcal{W}\right\Vert _{L_{v}^{\infty}H_{y}^{\alpha}}\lesssim T^{-\delta}\left\Vert \mathcal{V}\right\Vert _{\mathscr{Z}}+MT^{\delta},\quad\left\Vert \mathcal{W}\right\Vert _{L_{v,y}^{2}}\lesssim T^{-\delta}\left\Vert \mathcal{V}\right\Vert _{\mathscr{Z}}+M.
\]
Hence we have
\begin{align*}
\left\Vert \int_{T}^{2T}e^{-is\triangle_{y}/2}h_{2}(s)ds\right\Vert _{L_{v}^{\infty}H_{y}^{\alpha}} & \lesssim T^{-1}\left(T^{-\delta}\left\Vert \mathcal{V}\right\Vert _{\mathscr{Z}}+MT^{\delta}\right)^{3}+T^{-1}\left(T^{-\delta}\left\Vert \mathcal{V}\right\Vert _{\mathscr{Z}}+MT^{\delta}\right)^{5},
\end{align*}
\begin{align*}
\left\Vert \int_{T}^{2T}e^{-is\triangle_{y}/2}h_{2}(s)ds\right\Vert _{L_{v,y}^{2}} & \lesssim T^{-1}\left[\left(T^{-\delta}\left\Vert \mathcal{V}\right\Vert _{\mathscr{Z}}+MT^{\delta}\right)^{2}+\left(T^{-\delta}\left\Vert \mathcal{V}\right\Vert _{\mathscr{Z}}+MT^{\delta}\right)^{4}\right]\left(T^{-\delta}\left\Vert \mathcal{V}\right\Vert _{\mathscr{Z}}+M\right).
\end{align*}
Therefore we have the bound
\begin{equation}
\left\Vert \int_{t}^{\infty}e^{-i(t-s)\triangle_{y}/2}h_{2}(s)ds\right\Vert _{\mathscr{Z}}\lesssim\left(\left\Vert \mathcal{V}\right\Vert _{\mathscr{Z}}+M\right)^{3}+\left(\left\Vert \mathcal{V}\right\Vert _{\mathscr{Z}}+M\right)^{5}.\label{eq:h_2estim}
\end{equation}
By (\ref{eq:h_1esti}), (\ref{eq:h_2estim}), and the assumption that
$M$ is a small positive number, we obtain the bound
\begin{equation}
\left\Vert \mathcal{V}\right\Vert _{\mathscr{Z}}\lesssim M^{3}+M^{2}\left\Vert \mathcal{V}\right\Vert _{\mathscr{Z}}+M\left\Vert \mathcal{V}\right\Vert _{\mathscr{Z}}^{2}+\left\Vert \mathcal{V}\right\Vert _{\mathscr{Z}}^{3}+\left\Vert \mathcal{V}\right\Vert _{\mathscr{Z}}^{4}+\left\Vert \mathcal{V}\right\Vert _{\mathscr{Z}}^{5}.\label{eq:bigVbound}
\end{equation}
In a similar manner, for $\mathcal{V}_{1}$ and $\mathcal{V}_{2}$
both satisfying the equation, we have the Lipschitz bounds 
\begin{equation}
\left\Vert \mathcal{V}_{1}-\mathcal{V}_{2}\right\Vert _{\mathscr{Z}}\lesssim\left(M^{2}+\left\Vert \mathcal{V}_{1}\right\Vert _{\mathscr{Z}}^{4}+\left\Vert \mathcal{V}_{2}\right\Vert _{\mathscr{Z}}^{4}\right)\left\Vert \mathcal{V}_{1}-\mathcal{V}_{2}\right\Vert _{\mathscr{Z}}.\label{eq:bigVLip}
\end{equation}
By (\ref{eq:bigVbound}), we have the bound $\left\Vert \mathcal{V}\right\Vert _{\mathscr{Z}}\lesssim M^{3}$
if $M$ is small enough. Hence by (\ref{eq:bigVLip}), we can solve
the equation (\ref{eq:bigVeq}) for $\gamma$ by the contraction principle
in the function space $\mathscr{Z}$. 

\end{proof}

In order to prove the asymptotic completeness of (\ref{eq main}),
we introduce the following lemma:

\begin{lemma}\label{lemmaWreg} Suppose $W$ is a solution to the
equation (\ref{eq:eqtorus}), let $\alpha>\frac{d}{2}$, and assume
that $\left\Vert W(1)\right\Vert _{L_{v}^{\infty}H_{y}^{1}}^{2}\lesssim\delta$.
By Lemma \ref{resonance form}, there are the growth bounds: 

(a) For any $s\geq1$, if at $t=1$, we have $\left\Vert W(1)\right\Vert _{L_{v}^{2}H_{y}^{s}+L_{v}^{\infty}H_{y}^{\alpha}}<\infty$,
then for any $t\geq1$, there is the bound $\left\Vert W(t)\right\Vert _{L_{v}^{2}H_{y}^{s}+L_{v}^{\infty}H_{y}^{\alpha}}\lesssim t^{\delta}\left\Vert W(1)\right\Vert _{L_{v}^{2}H_{y}^{s}+L_{v}^{\infty}H_{y}^{\alpha}}$$.$

(b) For any $s>0$, if at $t=1$, we have $\left\Vert W(1)\right\Vert _{H_{v}^{s}L_{y}^{2}+L_{v}^{\infty}H_{y}^{\alpha}}<\infty$,
then for any $t\geq1$, there is the bound $\left\Vert W(t)\right\Vert _{H_{v}^{s}L_{y}^{2}+L_{v}^{\infty}H_{y}^{\alpha}}\lesssim t^{\delta}\left\Vert W(1)\right\Vert _{H_{v}^{s}L_{y}^{2}+L_{v}^{\infty}H_{y}^{\alpha}}$.

\end{lemma}

The lemma can be proved by following the same steps in Proposition
\ref{prop:wproperty}. 

\section{Asymptotic Completeness\label{sec:Asymptotic-Completeness}}

In this section we prove Theorem \ref{theoremasycom}. Suppose that
$W$ is the solution to the equation (\ref{eq:eqtorus}) for $t\in[1,\infty)$,
it suffices to prove existence of a solution $u(t)$ on $[1,\infty)$
satisfying 
\begin{equation}
\left\Vert u(1)\right\Vert _{L_{x}^{2}H_{y}^{s}}+\left\Vert L_{x}u(1)\right\Vert _{L_{x,y}^{2}}\lesssim\epsilon,
\end{equation}
so that $u$ is close to $W$ at $t=\infty$.

Since we need extra regularity in $v$ in order to finish the asymptotic
completeness proof, to start with the proof, from Lemma \ref{lemmaWreg},
we make extra assumptions that the initial data of $W(t)$ satisfies

\[
\left\Vert D_{v}^{1+8\delta}W\left(1\right)\right\Vert _{L_{v,y}^{2}},\left\Vert D_{y}^{s}D_{v}^{8\delta}W\left(1\right)\right\Vert _{L_{v,y}^{2}},\left\Vert D_{y}^{s}W\left(1\right)\right\Vert _{L_{v,y}^{2}}\lesssim M.
\]
Therefore we will have 
\[
\left\Vert W(1)\right\Vert _{L_{v}^{\infty}H_{y}^{1}}^{2}\lesssim M^{2}\lesssim\delta,
\]
and the growth rate for any $t\geq1$ 
\begin{equation}
\left\Vert D_{v}^{1+8\delta}W\left(t\right)\right\Vert _{L_{v,y}^{2}},\left\Vert D_{y}^{s}D_{v}^{8\delta}W\left(t\right)\right\Vert _{L_{v,y}^{2}},\left\Vert D_{y}^{s}W\left(t\right)\right\Vert _{L_{v,y}^{2}}\leq Mt^{\delta},
\end{equation}
where $M,\delta>0$ and $\delta\gg M$. For regularity reasons, instead
of working with the asymptotic profile 
\[
u_{asy}=\frac{1}{\sqrt{t}}e^{i\frac{x^{2}}{2t}}W\left(t,\frac{x}{t},y\right),
\]
we work with the regularized approximate solution 
\[
u_{app}=\frac{1}{\sqrt{t}}e^{i\frac{x^{2}}{2t}}\left(P_{\leq\sqrt{t}}W\right)\left(t,\frac{x}{t},y\right).
\]
Then by Bernstein's inequality we have the bounds
\begin{equation}
\left\Vert u_{asy}-u_{app}\right\Vert _{L_{v,y}^{2}}\lesssim Mt^{-\frac{1}{2}-4\delta},\quad\left\Vert u_{asy}-u_{app}\right\Vert _{L_{v}^{2}H_{y}^{s}}\lesssim Mt^{-4\delta},\quad\left\Vert L_{x}u_{asy}-L_{x}u_{app}\right\Vert _{L_{v,y}^{2}}\lesssim Mt^{-4\delta}.\label{eq:asyappdiff}
\end{equation}
Here we use direct computations showing that 
\begin{equation}
L_{x}u_{app}=\frac{i}{\sqrt{t}}e^{i\frac{x^{2}}{2t}}\left(\partial_{v}W\right)\left(t,\frac{x}{t},y\right),\quad D_{y}^{s}u_{app}=\frac{1}{\sqrt{t}}e^{i\frac{x^{2}}{2t}}\left(P_{\leq\sqrt{t}}D_{y}^{s}W\right)\left(t,\frac{x}{t},y\right),\label{eq:reLxpartialv}
\end{equation}
as well as the inequalities
\[
\left\Vert u_{app}\right\Vert _{L_{x}^{\infty}H_{y}^{\alpha}}\lesssim Mt^{-\frac{1}{2}+\alpha},\quad\left\Vert u_{app}\right\Vert _{L_{x}^{\infty}H_{y}^{1}}\lesssim M.
\]

Thus we have that 
\begin{align*}
 & \left(i\partial_{t}+\frac{1}{2}\partial_{x}^{2}+\frac{1}{2}\triangle_{y}\right)u_{app}\\
= & \left|u_{app}\right|^{2}u_{app}+t^{-\frac{3}{2}}e^{i\frac{x^{2}}{2t}}\mathcal{F}_{\xi}^{-1}\left[-\mathcal{X}\left(\frac{\xi}{\sqrt{t}}\right)\frac{\xi^{2}}{t}-\mathcal{X}^{\prime}\left(\frac{\xi}{\sqrt{t}}\right)\frac{\xi}{2\sqrt{t}}\right]*W\left(t,\frac{x}{t},y\right)\\
 & \quad+t^{-\frac{3}{2}}e^{i\frac{x^{2}}{2t}}P_{\leq\sqrt{t}}\left(|W|^{2}W-\left|P_{\leq\sqrt{t}}W\right|^{2}P_{\leq\sqrt{t}}W\right)\left(t,\frac{x}{t},y\right)+t^{-\frac{3}{2}}e^{i\frac{x^{2}}{2t}}P_{\geq\sqrt{t}}\left(\left|P_{\leq\sqrt{t}}W\right|^{2}P_{\leq\sqrt{t}}W\right)\left(t,\frac{x}{t},y\right)\\
:= & \left|u_{app}\right|^{2}u_{app}+I_{1}^{\prime}+I_{2}^{\prime}+I_{3}^{\prime}.
\end{align*}

To find $u$ solving the cubic NLS equation (\ref{eq main}) and matching
$u_{app}$, let $\tilde{w}=u-u_{app}$ solve the following equation
for $\tilde{w}$: 
\begin{align*}
\left(i\partial_{t}+\frac{1}{2}\partial_{x}^{2}+\frac{1}{2}\triangle_{y}\right)\tilde{w} & =\left|u_{app}+\tilde{w}\right|^{2}\left(u_{app}+\tilde{w}\right)-\left|u_{app}\right|^{2}u_{app}-I_{1}^{\prime}-I_{2}^{\prime}-I_{3}^{\prime}.
\end{align*}
By direct computation
\begin{align*}
\left|u_{app}+\tilde{w}\right|^{2}\left(u_{app}+\tilde{w}\right)-\left|u_{app}\right|^{2}u_{app} & =2\tilde{w}\left|u_{app}\right|^{2}+\overline{\tilde{w}}u_{app}^{2}+\tilde{w}^{2}\overline{u_{app}}+2\left|\tilde{w}\right|^{2}u_{app}+\left|\tilde{w}\right|^{2}\tilde{w}.
\end{align*}
Let 
\[
L_{lin}\left(u_{app},\tilde{w}\right):=2\left|u_{app}\right|^{2}\tilde{w}+u_{app}^{2}\overline{\tilde{w}},
\]
\[
Q_{1}\left(u_{app},\tilde{w}\right):=\tilde{w}^{2}\overline{u_{app}}+2\left|\tilde{w}\right|^{2}u_{app}+\left|\tilde{w}\right|^{2}\tilde{w}.
\]
Hence we need to solve the equation from infinity: 
\begin{equation}
\begin{aligned}\left(i\partial_{t}+\frac{1}{2}\partial_{x}^{2}+\frac{1}{2}\triangle_{y}\right)\tilde{w} & =L_{lin}\left(u_{app},\tilde{w}\right)+Q_{1}\left(u_{app},\tilde{w}\right)-I_{1}^{\prime}-I_{2}^{\prime}-I_{3}^{\prime},\\
 & \tilde{w}(\infty)=0.
\end{aligned}
\label{eq:asymprofile}
\end{equation}
The solution operator for the inhomogeneous Schr\"{o}dinger equation
with zero Cauchy data at infinity is given by 
\[
\Phi f=i\int_{t}^{\infty}U\left(t-s\right)f(s)ds.
\]
Then the equation for $\tilde{w}$ can be written as 
\begin{equation}
\tilde{w}=i\int_{t}^{\infty}U(t-s)\left(L_{lin}\left(u_{app},\tilde{w}\right)+Q_{1}\left(u_{app},w\right)-I_{1}^{\prime}-I_{2}^{\prime}-I_{3}^{\prime}\right)(s)ds.\label{eq:solw}
\end{equation}
We also need the backward solvability for the linearized equation
to solve (\ref{eq:asymprofile}). Hence consider the linearized equation,
where $\tilde{\nu}$ can be either $L_{x}\tilde{w}$ or $D_{y}^{s}\tilde{w}$
with Leibnitz rule correction terms (The correction terms can be estimated
by the same computation in Lemma \ref{LemmaD_ycorrect}.) By direct
computations, we obtain the equation 

\[
\left(i\partial_{t}+\frac{1}{2}\partial_{x}^{2}+\frac{1}{2}\triangle_{y}\right)\tilde{\nu}:=L_{lin}\left(u_{app},\tilde{\nu}\right)+Q_{2}\left(u_{app},\tilde{\nu}\right)+g\left(u_{app},\tilde{w}\right)-\tilde{I_{1}^{\prime}}-\tilde{I_{2}^{\prime}}-\tilde{I_{3}^{\prime}},
\]
where
\[
L_{lin}\left(u_{app},\tilde{\nu}\right):=2\left|u_{app}\right|^{2}\tilde{\nu}-u_{app}^{2}\overline{L\tilde{\nu}},
\]
\[
Q_{2}(u_{app},\tilde{\nu}):=2|\tilde{w}|^{2}\tilde{\nu}-\tilde{w}^{2}\overline{\tilde{\nu}}+2\overline{u_{app}}\tilde{w}\tilde{\nu}+2\overline{\tilde{w}}u_{app}\tilde{\nu}-2\tilde{w}u_{app}\overline{\tilde{\nu}}.
\]
For $\tilde{\nu}=L_{x}\tilde{w}$, defining the following functions:
\[
g\left(u_{app},\tilde{w}\right):=2w\left(L_{x}u_{app}\right)\overline{u_{app}}-2wu_{app}\overline{L_{x}u_{app}}-w^{2}\overline{L_{x}u_{app}}+2\overline{w}u_{app}L_{x}u_{app}+2|w|^{2}L_{x}u_{app},
\]
and 
\[
\tilde{I_{1}^{\prime}}=L_{x}I_{1}^{\prime},\quad\tilde{I_{2}^{\prime}}=L_{x}\tilde{I_{2}^{\prime}},\quad\tilde{I_{3}^{\prime}}=L_{x}\tilde{I_{3}^{\prime}}.
\]
For $\tilde{\nu}=D_{y}^{s}\tilde{w}$, we define $g\left(u_{app},\tilde{w}\right)$
and $\tilde{I_{1}^{\prime}},\tilde{I_{2}^{\prime}},\tilde{I_{3}^{\prime}}$
in the same manner.

The equation for $\tilde{\nu}$ can be written as 
\begin{equation}
\tilde{\nu}=i\int_{t}^{\infty}U(t-s)\left[L_{lin}\left(u_{app},\tilde{\nu}\right)+Q_{2}\left(u_{app},\tilde{\nu}\right)+g\left(u_{app},\tilde{w}\right)-\tilde{I_{1}^{\prime}}-\tilde{I_{2}^{\prime}}-\tilde{I_{3}^{\prime}}\right](s)ds.\label{eq:solLw}
\end{equation}

The solution (\ref{eq:solw}),(\ref{eq:solLw}) will be solved together
through the contraction principle, using the Strichartz bound
\begin{equation}
\left\Vert \Phi f\right\Vert _{L_{t}^{\infty}\left(T,\infty;L_{x,y}^{2}\right)}\lesssim\left\Vert f\right\Vert _{L_{t}^{1}\left(T,\infty;L_{x,y}^{2}\right)}.
\end{equation}
In order to bound $\left\Vert f\right\Vert _{L_{t}^{1}\left(T,\infty;L_{x,y}^{2}\right)}$
we divide $[T,\infty)$ into dyadic subintervals, estimate $\left\Vert f\right\Vert _{L_{t}^{\infty}L_{x,y}^{2}}$in
each interval, and then sum up. Hence we define a function space with
appropriate time decay, and let $\tilde{w}$ be solved in the space
\[
\left\Vert f\right\Vert _{Z}=\sup_{T>1}T^{\frac{1}{2}+\delta}\left[\left\Vert f\right\Vert _{L_{t}^{\infty}\left(T,2T;L_{x,y}^{2}\right)}\right],
\]
and the larger space for $L\tilde{w}$ is 
\[
\left\Vert f\right\Vert _{\tilde{Z}}=\sup_{T>1}T^{\delta}\left[\left\Vert f\right\Vert _{L_{t}^{\infty}\left(T,2T;L_{x,y}^{2}\right)}\right].
\]
Since we are unable to solve $\tilde{w}$ and $\tilde{\nu}$ separately,
here also define a norm $Z^{+}$ to be 
\[
\left\Vert \tilde{w}\right\Vert _{Z^{+}}^{2}:=\left\Vert \tilde{w}\right\Vert _{Z}^{2}+\left\Vert \tilde{\nu}\right\Vert _{\tilde{Z}}^{2}.
\]
In order to solve the equations for $\tilde{w}$ and $\tilde{\nu}$
simultaneously in $Z^{+}$ using the contraction principle we need
to show that

(1) The map
\[
\left(\tilde{w},\tilde{\nu}\right)\rightarrow\left(\Phi\left[L_{lin}\left(u_{app},\tilde{w}\right)+Q_{1}\left(u_{app},\tilde{w}\right)\right],\Phi\left[L_{lin}\left(u_{app},\tilde{\nu}\right)+Q_{2}\left(u_{app},\tilde{\nu}\right)+g\left(u_{app},\tilde{w}\right)\right]\right)
\]
maps $Z^{+}$ into $Z^{+}$with a small Lipschitz constant for $\left(\tilde{w},\tilde{\nu}\right)$
in a ball of radius $CM$ where $1\ll C\ll M.$

(2) The nonlinear term $I_{1}+I_{2}+I_{3}$ satisfies 
\[
\left\Vert \Phi\left(I_{1}+I_{2}+I_{3}\right)\right\Vert _{Z^{+}}\lesssim M.
\]

Then there is a solution $\left(\tilde{w},\tilde{\nu}\right)$ satisfying
\[
\left\Vert \tilde{w}\right\Vert _{Z^{+}}\lesssim M.
\]

\begin{lemma}There are bounds associated with $\tilde{w}$:
\begin{equation}
\left\Vert \Phi L_{lin}\left(u_{app},\tilde{w}\right)\right\Vert _{Z}\lesssim M^{2}\left\Vert \tilde{w}\right\Vert _{Z},\label{eq:zlinbd}
\end{equation}
\begin{equation}
\left\Vert \Phi Q_{1}\left(u_{app},\tilde{w}\right)\right\Vert _{Z}\lesssim\delta^{-1}M\left\Vert \tilde{w}\right\Vert _{Z}^{\frac{7}{6}}\left\Vert \tilde{\nu}\right\Vert _{\tilde{Z}}^{\frac{5}{6}}+\left\Vert \tilde{w}\right\Vert _{Z}^{\frac{4}{3}}\left\Vert \tilde{\nu}\right\Vert _{\tilde{Z}}^{\frac{5}{3}}.\label{eq:zq1bd}
\end{equation}

\end{lemma}

\begin{proof}

For (\ref{eq:zlinbd}), notice that $L_{lin}\left(u_{app},\tilde{w}\right)$
is nothing but the nonlinear term of the linearized equation (\ref{eq:lin}),
if one replaces $u$ by $u_{app}$ and $L_{x}u$ by $i\tilde{w}$.
Using the same computations as in the energy estimate section, we
have 
\begin{align*}
 & \left\Vert \int_{T}^{2T}U(-s)L_{lin}\left(u_{app},\tilde{w}\right)(s)ds\right\Vert _{L_{t}^{\infty}\left(T,2T;L_{x,y}^{2}\right)}\\
\lesssim & \left\Vert t^{-1}\left\Vert u_{app}\right\Vert _{L_{x}^{\infty}H_{y}^{1}}^{2}\left\Vert \tilde{w}\right\Vert _{L_{x,y}^{2}}\right\Vert _{L_{t}^{1}\left(T,2T\right)}\lesssim M^{2}\left\Vert \tilde{w}\right\Vert _{L_{t}^{\infty}\left(T,2T;L_{x,y}^{2}\right)}.
\end{align*}

We start the proof of (\ref{eq:zq1bd}) by the estimate used in (\ref{eq:alpha bound}),
and get

\begin{align*}
\left\Vert \tilde{w}(t,x,y)\right\Vert _{L_{x}^{\infty}H_{y}^{\alpha}} & \lesssim t^{-\frac{1}{2}}\left\Vert \tilde{w}\right\Vert _{L_{x,y}^{2}}^{\frac{1}{6}}\left\Vert L_{x}\tilde{w}\right\Vert _{L_{x,y}^{2}}^{\frac{1}{2}}\left\Vert D_{y}^{s}\tilde{w}\right\Vert _{L_{x,y}^{2}}^{\frac{1}{6}}.
\end{align*}

The estimate for $Q_{1}\left(u_{app},\tilde{w}\right)$ is straightforward
by applying above inequality. Here we let $\tilde{\nu}$ be either
$L_{x}$ or $D_{y}^{s}$,
\begin{align*}
\left\Vert \left|\tilde{w}\right|^{2}u_{app}\right\Vert _{L_{t}^{1}\left(T,2T;L_{x,y}^{2}\right)} & \lesssim\left\Vert \left\Vert \tilde{w}\right\Vert _{L_{x,y}^{2}}\left\Vert \tilde{w}\right\Vert _{L_{x}^{\infty}H_{y}^{\alpha}}\left\Vert u_{app}\right\Vert _{L_{x}^{\infty}H_{y}^{\alpha}}\right\Vert _{L_{t}^{1}\left(T,2T\right)}\lesssim\left\Vert Mt^{-1+\delta}\left\Vert \tilde{w}\right\Vert _{L_{x,y}^{2}}^{\frac{7}{6}}\left\Vert \tilde{\nu}\right\Vert _{L_{x,y}^{2}}^{\frac{5}{6}}\right\Vert _{L_{t}^{1}(T,2T)}\\
 & \lesssim\delta^{-1}MT^{\delta}\left\Vert \tilde{w}\right\Vert _{L_{t}^{\infty}\left(T,2T;L_{x,y}^{2}\right)}^{\frac{7}{6}}\left\Vert \tilde{\nu}\right\Vert _{L_{t}^{\infty}\left(T,2T;L_{x,y}^{2}\right)}^{\frac{5}{6}}\lesssim\delta^{-1}MT^{-\frac{7}{12}-2\delta}\left\Vert \tilde{w}\right\Vert _{Z}^{\frac{7}{6}}\left\Vert \tilde{\nu}\right\Vert _{\tilde{Z}}^{\frac{5}{6}}.
\end{align*}
\begin{align*}
\left\Vert \left|\tilde{w}\right|^{2}\tilde{w}\right\Vert _{L_{t}^{1}\left(T,2T;L_{x,y}^{2}\right)} & \lesssim\left\Vert \left\Vert \tilde{w}\right\Vert _{L_{x,y}^{2}}\left\Vert \tilde{w}\right\Vert _{L_{x}^{\infty}H_{y}^{\alpha}}^{2}\right\Vert _{L_{t}^{1}\left(T,2T\right)}\lesssim\left\Vert t^{-1}\left\Vert \tilde{w}\right\Vert _{L_{x,y}^{2}}^{\frac{4}{3}}\left\Vert \tilde{\nu}\right\Vert _{L_{x}^{\infty}}^{\frac{5}{3}}\right\Vert _{L_{t}^{1}\left(T,2T\right)}\\
 & \lesssim\left\Vert \tilde{w}\right\Vert _{L_{t}^{\infty}\left(T,2T;L_{x,y}^{2}\right)}^{\frac{4}{3}}\left\Vert \tilde{\nu}\right\Vert _{L_{t}^{\infty}\left(T,2T;L_{x,y}^{2}\right)}^{\frac{5}{3}}\lesssim T^{-\frac{2}{3}-3\delta}\left\Vert \tilde{w}\right\Vert _{Z}^{\frac{4}{3}}\left\Vert \tilde{\nu}\right\Vert _{\tilde{Z}}^{\frac{5}{3}}.
\end{align*}
\end{proof}

\begin{lemma}For $\delta<\frac{1}{24}$, there are the bounds associated
with $\tilde{\nu}$ :
\begin{equation}
\left\Vert \Phi L_{lin}\left(u_{app},\tilde{\nu}\right)\right\Vert _{\tilde{Z}}\lesssim M^{2}\left\Vert \tilde{\nu}\right\Vert _{\tilde{Z}},\label{eq:Lzlinbd}
\end{equation}
\begin{equation}
\left\Vert \Phi Q_{2}(u_{app},\tilde{\nu})\right\Vert _{\tilde{Z}}\lesssim\left\Vert \tilde{w}\right\Vert _{Z}^{\frac{1}{3}}\left\Vert \tilde{\nu}\right\Vert _{\tilde{Z}}^{\frac{8}{3}}+\delta^{-1}M\left\Vert \tilde{w}\right\Vert _{Z}^{\frac{1}{6}}\left\Vert \tilde{\nu}\right\Vert _{\tilde{Z}}^{\frac{11}{6}},\label{eq:LzQ2bd}
\end{equation}
\begin{equation}
\left\Vert \Phi g\left(u_{app},\tilde{w}\right)\right\Vert _{\tilde{Z}}\lesssim\delta^{-1}M^{2}\left\Vert \tilde{w}\right\Vert _{Z}^{\frac{1}{6}}\left\Vert \tilde{\nu}\right\Vert _{\tilde{Z}}^{\frac{5}{6}}+\delta^{-1}M\left\Vert \tilde{w}\right\Vert _{Z}^{\frac{1}{3}}\left\Vert \tilde{\nu}\right\Vert _{\tilde{Z}}^{\frac{5}{3}}.\label{eq:Lzgbd}
\end{equation}

\end{lemma}

\begin{proof}

By a similar process as the estimates for the $Z$ norm, we have
\[
\left\Vert |\tilde{w}|^{2}\tilde{\nu}\right\Vert _{L_{t}^{1}\left(T,2T;L_{x,y}^{2}\right)}\lesssim\left\Vert \left\Vert \tilde{w}\right\Vert _{L_{x}^{\infty}H_{y}^{\alpha}}^{2}\left\Vert \tilde{\nu}\right\Vert _{L_{x,y}^{2}}\right\Vert _{L_{t}^{1}\left(T,2T\right)}\lesssim T^{-\frac{1}{6}-3\delta}\left\Vert \tilde{w}\right\Vert _{Z}^{\frac{1}{3}}\left\Vert \tilde{\nu}\right\Vert _{\tilde{Z}}^{\frac{8}{3}},
\]
\begin{align*}
\left\Vert \overline{u_{app}}\tilde{w}\tilde{\nu}\right\Vert _{L_{t}^{1}\left(T,2T;L_{x,y}^{2}\right)} & \lesssim\left\Vert \left\Vert \tilde{w}\right\Vert _{L_{x}^{\infty}H_{y}^{\alpha}}\left\Vert u_{app}\right\Vert _{L_{x}^{\infty}H_{y}^{\alpha}}\left\Vert \tilde{\nu}\right\Vert _{L_{x,y}^{2}}\right\Vert _{L_{t}^{1}\left(T,2T\right)}\lesssim\delta^{-1}MT^{-\frac{1}{12}-\delta}\left\Vert \tilde{w}\right\Vert _{Z}^{\frac{1}{6}}\left\Vert \tilde{\nu}\right\Vert _{\tilde{Z}}^{\frac{11}{6}}.
\end{align*}

Since the operator can either by $D_{y}^{s}$ or $L_{x}$ here, we
compute the $L_{x}$case, and the $D_{y}^{s}$ case follows the same
manner. There are the bounds
\begin{align*}
 & \left\Vert w\left(L_{x}u_{app}\right)\overline{u_{app}}\right\Vert _{L_{t}^{1}\left(T,2T;L_{x,y}^{2}\right)}\lesssim\left\Vert \left\Vert \tilde{w}\right\Vert _{L_{x}^{\infty}H_{y}^{\alpha}}\left\Vert u_{app}\right\Vert _{L_{x}^{\infty}H_{y}^{\alpha}}\left\Vert L_{x}u_{app}\right\Vert _{L_{x,y}^{2}}\right\Vert _{L_{t}^{1}\left(T,2T\right)}\\
\lesssim & \delta^{-1}M^{2}T^{2\delta}\left\Vert \tilde{w}\right\Vert _{L_{t}^{\infty}\left(T,2T;L_{x,y}^{2}\right)}^{\frac{1}{6}}\left\Vert \tilde{\nu}\right\Vert _{L_{t}^{\infty}\left(T,2T;L_{x,y}^{2}\right)}^{\frac{5}{6}}\lesssim\delta^{-1}M^{2}T^{-\frac{1}{12}+\delta}\left\Vert \tilde{w}\right\Vert _{Z}^{\frac{1}{6}}\left\Vert \tilde{\nu}\right\Vert _{\tilde{Z}}^{\frac{5}{6}},
\end{align*}

\begin{align*}
 & \left\Vert w^{2}\overline{L_{x}u_{app}}\right\Vert _{L_{t}^{1}\left(T,2T;L_{x,y}^{2}\right)}\lesssim\left\Vert \left\Vert \tilde{w}\right\Vert _{L_{x}^{\infty}H_{y}^{\alpha}}^{2}\left\Vert L_{x}u_{app}\right\Vert _{L_{x,y}^{2}}\right\Vert _{L_{t}^{1}\left(T,2T\right)}\\
\lesssim & \delta^{-1}MT^{\delta}\left\Vert \tilde{w}\right\Vert _{L_{t}^{\infty}\left(T,2T;L_{x,y}^{2}\right)}^{\frac{1}{3}}\left\Vert \tilde{\nu}\right\Vert _{L_{t}^{\infty}\left(T,2T;L_{x,y}^{2}\right)}^{\frac{5}{3}}\lesssim\delta^{-1}MT^{-\frac{1}{6}-\delta}\left\Vert \tilde{w}\right\Vert _{Z}^{\frac{1}{3}}\left\Vert \tilde{\nu}\right\Vert _{\tilde{Z}}^{\frac{5}{3}}.
\end{align*}
\end{proof}

By (\ref{eq:zlinbd}),(\ref{eq:zq1bd}),(\ref{eq:Lzlinbd}),(\ref{eq:LzQ2bd}),(\ref{eq:Lzgbd}),
and $M\ll\delta$ we obtain the Lipschitz dependence:
\begin{equation}
\begin{aligned} & \left\Vert \Phi\left[L_{lin}\left(u_{app},\tilde{w}_{1}\right)+Q_{1}\left(u_{app},\tilde{w}_{1}\right)-L_{lin}\left(u_{app},\tilde{w}_{2}\right)-Q_{1}\left(u_{app},\tilde{w}_{2}\right)\right]\right\Vert _{Z}\\
\lesssim & \left(M^{2}+\left\Vert \tilde{w}_{1}\right\Vert _{Z^{+}}^{2}+\left\Vert \tilde{w}_{2}\right\Vert _{Z^{+}}^{2}\right)\left\Vert \tilde{w}_{1}-\tilde{w}_{2}\right\Vert _{Z^{+}},
\end{aligned}
\label{eq:Lipzbd}
\end{equation}
\begin{equation}
\begin{aligned} & \left\Vert \Phi\left[L_{lin}\left(u_{app},\tilde{\nu}_{1}\right)+Q_{2}\left(u_{app},\tilde{w}_{1}\right)+g\left(u_{app},\tilde{w}_{1}\right)-L_{lin}\left(u_{app},\tilde{\nu}_{2}\right)-Q_{2}\left(u_{app},\tilde{w}_{2}\right)-g\left(u_{app},\tilde{w}_{2}\right)\right]\right\Vert _{\tilde{Z}}\\
\lesssim & \left(M^{2}+\left\Vert \tilde{w}_{1}\right\Vert _{Z^{+}}^{2}+\left\Vert \tilde{w}_{2}\right\Vert _{Z^{+}}^{2}\right)\left\Vert \tilde{w}_{1}-\tilde{w}_{2}\right\Vert _{Z^{+}}.
\end{aligned}
\label{eq:LipLz}
\end{equation}

\begin{lemma} \label{lemmaIzbd} We have estimates for $I_{1}^{\prime}$,
$I_{2}^{\prime}$ ,$I_{3}^{\prime}$:
\begin{equation}
\left\Vert \Phi I_{1}^{\prime}\right\Vert _{Z}\lesssim M,\quad\left\Vert \Phi I_{2}^{\prime}\right\Vert _{Z},\left\Vert \Phi I_{3}^{\prime}\right\Vert _{Z}\lesssim M^{3}.\label{eq:ZIbound}
\end{equation}
\begin{equation}
\left\Vert \Phi L_{x}I_{1}^{\prime}\right\Vert _{\tilde{Z}}\lesssim M,\quad\left\Vert \Phi L_{x}I_{2}^{\prime}\right\Vert _{\tilde{Z}},\left\Vert \Phi L_{x}I_{3}^{\prime}\right\Vert _{\tilde{Z}}\lesssim M^{3}.
\end{equation}
\begin{equation}
\left\Vert \Phi D_{y}^{s}I_{1}^{\prime}\right\Vert _{\tilde{Z}}\lesssim M,\quad\left\Vert \Phi D_{y}^{s}I_{2}^{\prime}\right\Vert _{\tilde{Z}},\left\Vert \Phi D_{y}^{s}I_{3}^{\prime}\right\Vert _{\tilde{Z}}\lesssim M^{3}.
\end{equation}

\end{lemma}

\begin{proof}

By Bernstein's inequality, we have the following inequality for any
$t\geq1$,

\begin{align*}
\left\Vert I_{1}^{\prime}(t)\right\Vert _{L_{x,y}^{2}} & \lesssim t^{-\frac{3}{2}-4\delta}\left\Vert D_{v}^{1+8\delta}W\right\Vert _{L_{v,y}^{2}}\lesssim Mt^{-\frac{3}{2}-3\delta},
\end{align*}
\begin{align*}
\left\Vert I_{2}^{\prime}\left(t\right)\right\Vert _{L_{x,y}^{2}} & \lesssim t^{-1}\left\Vert |W|^{2}W-\left|P_{\leq\sqrt{t}}W\right|^{2}P_{\leq\sqrt{t}}W\right\Vert _{L_{v,y}^{2}}\lesssim t^{-1}\left\Vert W\right\Vert _{L_{v}^{\infty}H_{y}^{\alpha}}^{2}\left\Vert P_{\geq\sqrt{t}}W\right\Vert _{L_{v,y}^{2}}\lesssim M^{3}t^{-\frac{3}{2}-\delta},
\end{align*}
\[
\left\Vert I_{3}^{\prime}\left(t\right)\right\Vert _{L_{x,y}^{2}}\lesssim t^{-1}\left\Vert P_{\geq\sqrt{t}}\left(\left|P_{\leq\sqrt{t}}W\right|^{2}P_{\leq\sqrt{t}}W\right)\right\Vert _{L_{v,y}^{2}}\lesssim t^{-\frac{3}{2}-4\delta}\left\Vert W\right\Vert _{L_{v}^{\infty}H_{y}^{\alpha}}^{2}\left\Vert D_{v}^{1+8\delta}W\right\Vert _{L_{v,y}^{2}}\lesssim M^{3}t^{-\frac{3}{2}-\delta}.
\]
Then integrating the above inequalities with respect to $t$ in $[T,2T]$,
(\ref{eq:ZIbound}) is straightforward.

By (\ref{eq:reLxpartialv}) and letting the projection operator $\tilde{P}=\mathcal{F}_{\xi}^{-1}\left[-\mathcal{X}\left(\frac{\xi}{\sqrt{t}}\right)\frac{\xi^{2}}{t}-\mathcal{X}^{\prime}\left(\frac{\xi}{\sqrt{t}}\right)\frac{\xi}{2\sqrt{t}}\right]$
we separate $L_{x}I_{1}^{\prime}$ into two parts: 
\begin{align*}
L_{x}I_{1}^{\prime}= & it^{-\frac{3}{2}}e^{i\frac{x^{2}}{2t}}\tilde{P}\left(\partial_{v}P_{\ll\sqrt{t}}W\right)\left(t,\frac{x}{t},y\right)+it^{-\frac{3}{2}}e^{i\frac{x^{2}}{2t}}\tilde{P}\left(\partial_{v}P_{\approx\sqrt{t}}W\right)\left(t,\frac{x}{t},y\right),
\end{align*}
The first part is very small due to the fact that $\tilde{P}P_{\ll\sqrt{t}}\approx0$,
hence it suffices only to do the estimate for the second part. The
second part we integrate with respect to $t$ first and get

\begin{align*}
\left\Vert t^{-\frac{3}{2}}e^{i\frac{x^{2}}{2t}}\tilde{P}\left(\partial_{v}P_{\approx\sqrt{t}}W\right)\left(t,\frac{x}{t},y\right)\right\Vert _{L_{t}^{1}\left(T,2T;L_{x,y}^{2}\right)} & \lesssim\left\Vert \tilde{P}\partial_{v}P_{\approx\sqrt{t}}W\right\Vert _{L_{t}^{\infty}\left(T,2T;L_{v,y}^{2}\right)}\\
 & \lesssim T^{-4\delta}\left\Vert D_{v}^{1+8\delta}W\right\Vert _{L_{t}^{\infty}\left(T,2T;L_{v,y}^{2}\right)}\lesssim MT^{-3\delta}.
\end{align*}
Similar estimates apply for $D_{y}^{s}I_{1}^{\prime}$.

By Bernstein's inequality,
\begin{align*}
\left\Vert L_{x}I_{2}^{\prime}\left(t\right)\right\Vert _{L_{t}^{1}\left(T,2T;L_{x,y}^{2}\right)} & \lesssim\left\Vert \partial_{v}P_{\leq\sqrt{t}}\left(|W|^{2}W-\left|P_{\leq\sqrt{t}}W\right|^{2}P_{\leq\sqrt{t}}W\right)\right\Vert _{L_{t}^{\infty}\left(T,2T;L_{v,y}^{2}\right)}\\
 & \lesssim T^{\frac{1}{2}}\left\Vert W\right\Vert _{L_{t}^{\infty}\left(T,2T;L_{v,}^{\infty}H_{y}^{\alpha}\right)}^{2}\left\Vert P_{\geq\sqrt{t}}W\right\Vert _{L_{t}^{\infty}\left(T,2T;L_{v,y}^{2}\right)}\lesssim M^{3}T^{-\delta},
\end{align*}
\begin{align*}
\left\Vert L_{x}I_{3}^{\prime}\left(t\right)\right\Vert _{L_{t}^{1}\left(T,2T;L_{x,y}^{2}\right)} & \lesssim\left\Vert P_{\geq\sqrt{t}}\partial_{v}\left(\left|P_{\leq\sqrt{t}}W\right|^{2}P_{\leq\sqrt{t}}W\right)\right\Vert _{L_{t}^{\infty}\left(T,2T;L_{v,y}^{2}\right)}\\
 & \lesssim T^{-4\delta}\left\Vert D_{v}^{1+8\delta}\left(\left|P_{\leq\sqrt{t}}W\right|^{2}P_{\leq\sqrt{t}}W\right)\right\Vert _{L_{t}^{\infty}\left(T,2T;L_{v,y}^{2}\right)}\lesssim M^{3}T^{-\delta}.
\end{align*}
Applying the inequality 
\begin{align*}
 & \left\Vert P_{\geq\sqrt{t}}W\right\Vert _{L_{t}^{\infty}\left(T,2T;L_{v}^{\infty}H_{y}^{\alpha}\right)}\\
\lesssim & \left\Vert P_{\geq\sqrt{t}}W\right\Vert _{L_{t}^{\infty}\left(T,2T;L_{v,y}^{2}\right)}^{\frac{1}{6}}\left\Vert D_{y}^{s}P_{\geq\sqrt{t}}W\right\Vert _{L_{t}^{\infty}\left(T,2T;L_{v,y}^{2}\right)}^{\frac{1}{3}}\left\Vert \partial_{v}P_{\geq\sqrt{t}}W\right\Vert _{L_{t}^{\infty}\left(T,2T;L_{v,y}^{2}\right)}^{\frac{1}{2}}\lesssim M^{3}T^{-\frac{1}{12}-\delta},
\end{align*}
we will have
\begin{align*}
 & \left\Vert D_{y}^{s}I_{2}^{\prime}\left(t\right)\right\Vert _{L_{t}^{1}\left(T,2T;L_{x,y}^{2}\right)}\lesssim\left\Vert D_{y}^{s}\left(|W|^{2}W-\left|P_{\leq\sqrt{t}}W\right|^{2}P_{\leq\sqrt{t}}W\right)\right\Vert _{L_{t}^{\infty}\left(T,2T;L_{v,y}^{2}\right)}\\
\lesssim & \left\Vert \left(D_{y}^{s}W\right)W\left(P_{\geq\sqrt{t}}W\right)\right\Vert _{L_{t}^{\infty}\left(T,2T;L_{v,y}^{2}\right)}+\left\Vert W^{2}\left(D_{y}^{s}P_{\geq\sqrt{t}}W\right)\right\Vert _{L_{t}^{\infty}\left(T,2T;L_{v,y}^{2}\right)}\\
\lesssim & \left\Vert W\right\Vert _{L_{t}^{\infty}\left(T,2T;L_{v,}^{\infty}H_{y}^{\alpha}\right)}\left\Vert P_{\geq\sqrt{t}}W\right\Vert _{L_{t}^{\infty}\left(T,2T;L_{v,}^{\infty}H_{y}^{\alpha}\right)}\left\Vert D_{y}^{s}W\right\Vert _{L_{t}^{\infty}\left(T,2T;L_{v,y}^{2}\right)}\\
 & \quad+\left\Vert W\right\Vert _{L_{t}^{\infty}\left(T,2T;L_{v,}^{\infty}H_{y}^{\alpha}\right)}^{2}\left\Vert D_{y}^{s}P_{\geq\sqrt{t}}W\right\Vert _{L_{t}^{\infty}\left(T,2T;L_{v,y}^{2}\right)}\lesssim M^{3}T^{-\delta}.
\end{align*}
\begin{align*}
\left\Vert D_{y}^{s}I_{3}^{\prime}\left(t\right)\right\Vert _{L_{t}^{1}\left(T,2T;L_{x,y}^{2}\right)} & \lesssim\left\Vert P_{\geq\sqrt{t}}D_{y}^{s}\left(\left|P_{\leq\sqrt{t}}W\right|^{2}P_{\leq\sqrt{t}}W\right)\right\Vert _{L_{t}^{\infty}\left(T,2T;L_{v,y}^{2}\right)}\\
 & \lesssim T^{-4\delta}\left\Vert D_{y}^{s}D_{v}^{8\delta}\left(\left|P_{\leq\sqrt{t}}W\right|^{2}P_{\leq\sqrt{t}}W\right)\right\Vert _{L_{t}^{\infty}\left(T,2T;L_{v,y}^{2}\right)}\lesssim M^{3}T^{-\delta}.
\end{align*}
\end{proof}

Combining Lemma \ref{lemmaIzbd} with (\ref{eq:Lipzbd}), (\ref{eq:LipLz}),
we have
\begin{align*}
\left\Vert \tilde{w}\right\Vert _{Z^{+}} & \lesssim M+M^{3}+\left(M^{2}+\left\Vert \tilde{w}\right\Vert _{Z^{+}}^{2}\right)\left\Vert \tilde{w}\right\Vert _{Z^{+}},
\end{align*}
 hence for $M$ small enough there is the desired property $\left\Vert \tilde{w}\right\Vert _{Z^{+}}\lesssim M,$
and $u=u_{app}+\tilde{w}$ is a solution to (\ref{eq main}). Recalling
(\ref{eq:asyappdiff}), $u_{asy}$ tends to $u$ in the following
sense that
\[
\left\Vert u_{asy}-u\right\Vert _{L_{v,y}^{2}}\lesssim Mt^{-\frac{1}{2}-\delta},\quad\left\Vert u_{asy}-u\right\Vert _{L_{v}^{2}H_{y}^{s}}\lesssim Mt^{-\delta},\quad\left\Vert L_{x}u_{asy}-L_{x}u\right\Vert _{L_{v,y}^{2}}\lesssim Mt^{-\delta}.
\]

\subsection*{Acknowledgments.}

The author would like to thank her advisor Daniel Tataru for suggesting
the problem and several key suggestions for the proof. She would also
like to thank Mihaela Ifrim and Jason Murphy for a number of helpful
discussions.

\section{Additional Estimates}

The proof is given in \cite{HPTV} Lemma 7.1 and Lemma 7.4: see also
\cite{B1,HTT2,B2}. The definiton of $\Gamma_{0}$ is given in (\ref{def:gamma_0}).

\begin{lemma}\label{resonance form} Let $R$ be defined as 
\begin{equation}
R\left[a^{1}(\mathbf{k}_{1}),a^{2}(\mathbf{k}_{2}),a^{3}(\mathbf{k}_{3})\right](\mathbf{k}_{4}):=\sum_{\Gamma_{0}\left(\mathbf{k}_{4}\right)}a^{1}(\mathbf{k}_{1})\overline{a^{2}(\mathbf{k}_{2})}a^{3}(\mathbf{k}_{3}).
\end{equation}
For every sequences $a^{1}(\mathbf{k}),a^{2}(\mathbf{k}),a^{3}(\mathbf{k})$
indexed by $\mathbb{Z}^{d},d\leq4$, we have
\begin{equation}
\left\Vert R\left[a^{1},a^{2},a^{3}\right]\right\Vert _{l_{\mathbf{k}}^{2}}\lesssim_{d}\min_{\tau\in\mathfrak{S}(3)}\left\Vert a^{\tau(1)}\right\Vert _{l_{\mathbf{k}}^{2}}\left\Vert a^{\tau(2)}\right\Vert _{h_{\mathbf{k}}^{1}}\left\Vert a^{\tau(3)}\right\Vert _{h_{\mathbf{k}}^{1}}.\label{eq:h1esti}
\end{equation}
For any $s>0$, there is the inequality 
\[
\left\Vert R\left[a^{1},a^{2},a^{3}\right]\right\Vert _{h_{\mathbf{k}}^{s}}\lesssim_{d}\max_{\tau\in\mathfrak{S}(3)}\left\Vert a^{\tau(1)}\right\Vert _{h_{\mathbf{k}}^{s}}\left\Vert a^{\tau(2)}\right\Vert _{h_{\mathbf{k}}^{1}}\left\Vert a^{\tau(3)}\right\Vert _{h_{\mathbf{k}}^{1}},
\]
and if $a^{i}=a^{i}(v,\mathbf{k})$ for $i=1,2,3$, there is also
the inequality
\[
\left\Vert D_{v}^{s}R\left[a^{1},a^{2},a^{3}\right]\right\Vert _{L_{v}^{2}l_{\mathbf{k}}^{2}}\lesssim_{d}\max_{\tau\in\mathfrak{S}(3)}\left\Vert D_{v}^{s}a^{\tau(1)}\right\Vert _{L_{v}^{2}l_{\mathbf{k}}^{2}}\left\Vert a^{\tau(2)}\right\Vert _{L_{v}^{\infty}h_{\mathbf{k}}^{1}}\left\Vert a^{\tau(3)}\right\Vert _{L_{v}^{\infty}h_{\mathbf{k}}^{1}}.
\]

\end{lemma}

\begin{lemma}\label{elementarylp}Consider three sequences $\left\{ c_{1}\right\} ,$
$\left\{ c_{1}\right\} $ and $\left\{ c_{3}\right\} $. We will have
the following elementary bound

\[
\left\Vert \sum_{\mathcal{M}\left(k\right)}c_{1}\left(k_{1}\right)c_{2}\left(k_{2}\right)c_{3}\left(k_{3}\right)\right\Vert _{l_{k}^{2}}\lesssim\min_{\tau\in\mathfrak{S}\left(3\right)}\left\Vert c_{\tau\left(1\right)}\right\Vert _{l_{k}^{2}}\left\Vert c_{\tau\left(2\right)}\right\Vert _{l_{k}^{1}}\left\Vert c_{\tau\left(3\right)}\right\Vert _{l_{k}^{1}}.
\]

\end{lemma}

\begin{lemma}\label{strichartz}

Let $x\in\mathbb{R}$ , $y\in\mathbb{T}^{d}$, and $\alpha>\frac{d}{2}$.
Denote $U(t)=e^{it\partial_{x}^{2}/2}e^{it\triangle_{y}/2}$, then
we have the following Strichartz inequality:
\begin{align*}
(1)\quad & \left\Vert U(t)u_{0}\right\Vert _{L_{x,y}^{\infty}}\lesssim\left|t\right|^{-\frac{1}{2}}\left\Vert u_{0}\right\Vert _{L_{x}^{1}H_{y}^{\alpha}},\quad\left\Vert U(t)u_{0}\right\Vert _{L_{t}^{4}\left(I,L_{x,y}^{\infty}\right)}\lesssim\left\Vert u_{0}\right\Vert _{L_{x}^{2}H_{y}^{\alpha}},\\
(2)\quad & \left\Vert \int_{I}U^{*}\left(t\right)F\left(t\right)dt\right\Vert _{L_{x}^{2}H_{y}^{-\alpha}}\lesssim\left\Vert F(t)\right\Vert _{L_{t}^{4/3}\left(I,L_{x,y}^{1}\right)},\\
(3)\quad & \left\Vert \int_{I}U\left(t-s\right)F(s)ds\right\Vert _{L_{t}^{4}\left(L_{x,y}^{\infty}\right)}\lesssim\left\Vert F\left(t\right)\right\Vert _{L_{t}^{4/3}\left(I,L_{x}^{1}H_{y}^{\alpha}\right)},\quad\left\Vert \int_{I}U\left(t-s\right)F(s)ds\right\Vert _{L_{t}^{4}\left(L_{x,y}^{\infty}\right)}\lesssim\left\Vert F\right\Vert _{L_{t}^{1}\left(I;L_{x}^{2}H_{y}^{\alpha}\right)}.
\end{align*}

\end{lemma}

\begin{proof}The estimate:
\[
\left\Vert U(t)u_{0}\right\Vert _{L_{x,y}^{\infty}}\lesssim\left|t\right|^{-\frac{1}{2}}\left\Vert u_{0}\right\Vert _{L_{x}^{1}H_{y}^{\alpha}}
\]
can be obtained by direct computation.

By standard $UU^{*}$ argument, consider $\varphi\in L_{t}^{4/3}\left(L_{x,y}^{\infty}\right)^{*}\cap L_{t}^{4/3}L_{x,y}^{1}$
with $\left\Vert \varphi\right\Vert _{L_{t}^{4/3}\left(I,L_{x,y}^{1}\right)}=1$,
we will have
\begin{align*}
\left\Vert U(t)u_{0}\right\Vert _{L_{t}^{4}\left(I,L_{x,y}^{\infty}\right)} & \leq\sup_{\varphi}\int_{I}\left\langle U\left(t\right)u_{0},\varphi\right\rangle _{L_{x,y}^{2}}dt\leq\sup_{\varphi}\left\Vert u_{0}\right\Vert _{L_{x}^{2}H_{y}^{\alpha}}\left\Vert \int_{I}U^{*}\left(t\right)\varphi\left(t\right)dt\right\Vert _{L_{x}^{2}H_{y}^{-\alpha}}.
\end{align*}
Then estimate $\left\Vert \int_{I}U^{*}\left(t\right)\varphi\left(t\right)dt\right\Vert _{L_{x}^{2}H_{y}^{-\alpha}}$
by inner product, then for $\alpha_{1},\alpha_{2}>0$ satisfying $\alpha_{1}-\alpha>0$,
$\alpha_{1}+\alpha_{2}=2\alpha$. By the dual space property, $H^{\alpha^{\prime}}\subset H^{\frac{d}{2}+}\subset L^{\infty}$
hence $\left(L^{\infty}\right)^{*}\subset H^{-\frac{d}{2}-}\subset H^{-\alpha^{\prime}}$
if $0<\alpha^{\prime}<\frac{d}{2}$ , we will have
\begin{align*}
 & \int_{I}\int_{I}\left\langle U\left(t^{\prime}\right)\varphi\left(t^{\prime}\right),U\left(t\right)\varphi\left(t\right)\right\rangle _{L_{x}^{2}H_{y}^{-\alpha}}dt^{\prime}dt=\int_{I}\int_{I}\left\langle U^{*}\left(t\right)U\left(t^{\prime}\right)\varphi\left(t^{\prime}\right),\varphi\left(t\right)\right\rangle _{L_{x}^{2}H_{y}^{-\alpha}}dt^{\prime}dt\\
\lesssim & \int_{I}\int_{I}\left\Vert U^{*}\left(t\right)U\left(t^{\prime}\right)\varphi\left(t^{\prime}\right)\right\Vert _{L_{x}^{\infty}H_{y}^{-\alpha_{1}}}\left\Vert \varphi\left(t\right)\right\Vert _{\left(L_{x}^{\infty}\right)^{*}H_{y}^{-\alpha_{2}}}dt^{\prime}dt.
\end{align*}
Since 
\begin{align*}
\left\Vert U^{*}\left(t\right)U\left(t^{\prime}\right)\varphi\left(t^{\prime}\right)\right\Vert _{L_{x}^{\infty}H_{y}^{-\alpha_{1}}} & \lesssim\left|t^{\prime}-t\right|^{-\frac{1}{2}}\left\Vert \varphi\left(t^{\prime}\right)\right\Vert _{L_{x}^{1}H_{y}^{-\alpha_{1}+\alpha}}\lesssim\left|t^{\prime}-t\right|^{-\frac{1}{2}}\left\Vert \varphi\left(t^{\prime}\right)\right\Vert _{L_{x}^{1}\left(L_{y}^{\infty}\right)^{*}}
\end{align*}
and 
\[
\left\Vert \varphi\left(t\right)\right\Vert _{\left(L_{x}^{\infty}\right)^{*}H_{y}^{-\alpha_{2}}}\lesssim\left\Vert \varphi\left(t\right)\right\Vert _{\left(L_{x}^{\infty}\right)^{*}\left(L_{y}^{\infty}\right)^{*}}\lesssim\left\Vert \varphi\left(t\right)\right\Vert _{L_{x}^{1}L_{y}^{1}},
\]
we prove the second inequality, 
\[
\left\Vert \int_{I}U^{*}\left(t\right)\varphi\left(t\right)dt\right\Vert _{L_{x}^{2}H_{y}^{-\alpha}}\lesssim\left\Vert \varphi\right\Vert _{L_{t}^{4/3}\left(I,\left(L_{x,y}^{\infty}\right)^{*}\right)}\left\Vert \left|t\right|^{-\frac{1}{2}}*_{t}\varphi\left(t\right)\right\Vert _{L_{t}^{4}\left(I,L_{x}^{1}L_{y}^{1}\right)}\lesssim\left\Vert \varphi\right\Vert _{L_{t}^{4/3}\left(I,L_{x,y}^{1}\right)}\left\Vert \varphi\right\Vert _{L_{t}^{4/3}\left(I,L_{x,y}^{1}\right)}.
\]

The last inequality comes from 
\begin{align*}
 & \left\Vert \int_{I}U\left(t-s\right)F(s)ds\right\Vert _{L_{t}^{4}\left(I;L_{x,y}^{\infty}\right)}\lesssim\left\Vert \int_{I}\left\Vert U(t-s)F(s)\right\Vert _{L_{x,y}^{\infty}}ds\right\Vert _{L_{t}^{4}\left(I\right)}\\
\lesssim & \left\Vert \int_{I}\left|t-s\right|^{-\frac{1}{2}}\left\Vert F\left(s\right)\right\Vert _{L_{x}^{1}H_{y}^{\alpha}}ds\right\Vert _{L_{t}^{4}\left(I\right)}\lesssim\left\Vert F\right\Vert _{L_{t}^{4/3}\left(I;L_{x}^{1}H_{y}^{\alpha}\right)}.
\end{align*}
By Christ-Kiselev lemma we have 
\[
\left\Vert \int_{s<t}U\left(t-s\right)F(s)ds\right\Vert _{L_{t}^{4}\left(I;L_{x,y}^{\infty}\right)}\lesssim\left\Vert F\right\Vert _{L_{t}^{4/3}\left(I;L_{x}^{1}H_{y}^{\alpha}\right)}.
\]
By direct computation we have
\begin{align*}
 & \left\Vert \int_{I}U\left(t-s\right)F(s)ds\right\Vert _{L_{t}^{4}\left(L_{x,y}^{\infty}\right)}=\left\Vert U(t)\int_{I}U\left(-s\right)F(s)ds\right\Vert _{L_{t}^{4}\left(L_{x,y}^{\infty}\right)}\\
\lesssim & \left\Vert \int_{I}U\left(-s\right)F(s)ds\right\Vert _{L_{x}^{2}H_{y}^{\alpha}}\lesssim\left\Vert F\right\Vert _{L_{t}^{1}\left(I;L_{x}^{2}H_{y}^{\alpha}\right)},
\end{align*}
and by Christ-Kiselev lemma 
\[
\left\Vert \int_{s<t}U\left(t-s\right)F(s)ds\right\Vert _{L_{t}^{4}\left(I;L_{x,y}^{\infty}\right)}\lesssim\left\Vert F\right\Vert _{L_{t}^{1}\left(I;L_{x}^{2}H_{y}^{\alpha}\right)}.
\]

\end{proof}

\begin{corollary} For any function $f_{1}$, $f_{2}$ and $f_{3}$,
suppose that $f_{i}\in L_{x}^{2}H_{y}^{\alpha}\cap L_{t}^{4}L_{x,y}^{\infty}$.
We have the following inequality:
\begin{align*}
 & \left\Vert \int_{0}^{t}U\left(t-s\right)\left(f_{1}\left(s\right)f_{2}\left(s\right)f_{3}\left(s\right)\right)ds\right\Vert _{L_{t}^{\infty}\left(I;L_{x}^{2}H_{y}^{\alpha}\right)}\\
\lesssim & \sum_{\tau\in\mathfrak{S}\left(3\right)}\left|I\right|^{\frac{1}{2}}\left\Vert f_{\tau\left(1\right)}\right\Vert _{L_{t}^{4}\left(I,L_{x,y}^{\infty}\right)}\left\Vert f_{\tau\left(2\right)}\right\Vert _{L_{t}^{4}\left(I;L_{x,y}^{\infty}\right)}\left\Vert f_{\tau\left(3\right)}\right\Vert _{L_{t}^{\infty}\left(I;L_{x}^{2}H_{y}^{\alpha}\right)},
\end{align*}
\begin{align*}
 & \left\Vert \int_{0}^{t}U\left(t-s\right)\left(f_{1}\left(s\right)f_{2}\left(s\right)f_{3}\left(s\right)\right)ds\right\Vert _{L_{t}^{4}\left(I,L_{x,y}^{\infty}\right)}\\
\lesssim & \sum_{\tau\in\mathfrak{S}\left(3\right)}\left|I\right|^{\frac{3}{4}}\left\Vert f_{\tau(1)}\right\Vert _{L_{t}^{4}\left(I,L_{x,y}^{\infty}\right)}\left\Vert f_{\tau(2)}\right\Vert _{L_{t}^{\infty}\left(I,L_{x}^{2}H_{y}^{\alpha}\right)}\left\Vert f_{\tau(3)}\right\Vert _{L_{t}^{\infty}\left(I,L_{x}^{2}H_{y}^{\alpha}\right)}.
\end{align*}

\end{corollary}

\begin{lemma}\label{lemma:correction} Let $a_{s}\left(\xi,\eta,\theta\right)=\left|\eta+\theta\xi\right|^{s}$.
Define 
\[
A_{s}^{m}\left(\theta\right)\left(f,g\right)=\int_{\mathbb{R}^{n}}\int_{\mathbb{R}^{n}}e^{ix(\xi+\eta)}\frac{1}{m!}\partial_{\theta}^{m}a_{s}\left(\xi,\eta,\theta\right)\hat{f}\left(\xi\right)\hat{g}\left(\eta\right)d\xi d\eta.
\]
 Let $l\in\mathbb{N}$. Let $p,p_{1},p_{2}$ satisfy $1<p,p_{1},p_{2}<\infty$
and $\frac{1}{p}=\frac{1}{p_{1}}+\frac{1}{p_{2}}$. Let $s,s_{1},s_{2}$
satisfy $0\leq s_{2},s_{2}$ and $l-1\leq s=s_{1}+s_{2}\leq l$. Then
the following bilinear estimate
\[
\left\Vert D^{s}\left(fg\right)-\sum_{k\in\mathbb{Z}}\sum_{m=0}^{l-1}A_{s}^{m}\left(0\right)\left(P_{\leq k-3}f,P_{k}g\right)-\sum_{k\in\mathbb{Z}}\sum_{m=0}^{l-1}A_{s}^{m}\left(0\right)\left(P_{\leq k-3}g,P_{k}f\right)\right\Vert _{L^{p}}\leq C\left\Vert D^{s_{1}}f\right\Vert _{L^{p_{1}}}\left\Vert D^{s_{2}}g\right\Vert _{L^{p_{2}}}
\]
holds for all $f,g\in\mathcal{S}$, where $C=C(n,p,p_{1},p_{2}).$

\end{lemma}

See \cite{FGO}.

\bibliographystyle{plain}
\bibliography{cubicNLS_revised}

\end{document}